\documentclass[10pt,twoside, a4paper,reqno]{amsart}

\usepackage{amscd,amssymb,amsmath,graphicx,verbatim,mathrsfs,xcolor}
 \usepackage[foot]{amsaddr}
\usepackage{wasysym}
\usepackage{hyperref}
\usepackage{enumitem}
\usepackage{microtype}
\usepackage[capitalize]{cleveref}
\usepackage{mathtools}
\usepackage[margin=2.5cm]{geometry}

\binoppenalty=\maxdimen
\relpenalty=\maxdimen

\usepackage{tikz,mathrsfs}
\usetikzlibrary{babel,arrows,decorations.markings,arrows.meta,decorations.pathmorphing,cd}
\tikzset{
	symbol/.style={
		draw=none,
		every to/.append style={
			edge node={node [sloped, allow upside down, auto=false]{$#1$}}}
}}
\tikzset{>=stealth'}
\tikzcdset{arrow style=tikz, diagrams={>=stealth'}, row sep=6em, column sep=6em}
\def\arrowLengthDisplayStyle{4ex}
\def\arrowHeightDisplayStyle{.8ex}
\def\arrowSkipDisplayStyle{.5ex}
\def\arrowLengthTextStyle{3ex}
\def\arrowHeightTextStyle{.8ex}
\def\arrowSkipTextStyle{.4ex}
\def\arrowLengthScriptStyle{2.5ex}
\def\arrowHeightScriptStyle{.6ex}
\def\arrowSkipScriptStyle{.3ex}
\def\arrowLengthScriptScriptStyle{2ex}
\def\arrowHeightScriptScriptStyle{.4ex}
\def\arrowSkipScriptScriptStyle{.2ex}

\renewcommand{\to}{\arrow{->}}
\newcommand{\mono}{\arrow{>->}}
\newcommand{\epi}{\arrow{->>}}
\newcommand{\embed}{\arrow{right hook->}}
\renewcommand{\mapsto}{\arrow{|->}}

\newcommand{\MakeTikzArrowWithSuperscriptSubscript}[4]
{
	\mathchoice
	{ 
		\hspace*{\arrowSkipDisplayStyle}
		\begin{tikzpicture}[baseline]
		\draw [#1] (0,\arrowHeightDisplayStyle) -- node [above] {$#2$} node [below] {$#3$} (#4 * \arrowLengthDisplayStyle, \arrowHeightDisplayStyle);
		\end{tikzpicture}
		\hspace*{\arrowSkipDisplayStyle}
	}
	{ 
		\hspace*{\arrowSkipTextStyle}
		\begin{tikzpicture}[baseline]
		\draw [#1] (0,\arrowHeightTextStyle) -- node [above] {$\scriptstyle #2$} node [below] {$\scriptstyle #3$} (#4 * \arrowLengthTextStyle, \arrowHeightTextStyle);
		\end{tikzpicture}
		\hspace*{\arrowSkipTextStyle}
	}
	{ 
		\hspace*{\arrowSkipScriptStyle}
		\begin{tikzpicture}[baseline]
		\draw [#1] (0,\arrowHeightScriptStyle) -- node [above] {$\scriptscriptstyle #2$} node [below] {$\scriptscriptstyle #3$} (#4 * \arrowLengthScriptStyle, \arrowHeightScriptStyle);
		\end{tikzpicture}
		\hspace*{\arrowSkipScriptStyle}
	}
	{ 
		\hspace*{\arrowSkipScriptScriptStyle}
		\begin{tikzpicture}[baseline]
		\draw [#1] (0,\arrowHeightScriptScriptStyle) -- node [above] {$\scriptscriptstyle #2$} node [below] {$\scriptscriptstyle #3$} (#4 * \arrowLengthScriptScriptStyle, \arrowHeightScriptScriptStyle);
		\end{tikzpicture}
		\hspace*{\arrowSkipScriptScriptStyle}
	}
}

\newcommand{\MakeTikzArrowWithCentralLabel}[3]
{
	\mathchoice
	{ 
		\hspace*{\arrowSkipDisplayStyle}
		\begin{tikzpicture}[baseline]
		\draw [#1] (0,\arrowHeightDisplayStyle) -- node [fill=white,inner sep=1pt] {$#2$} (#3 * \arrowLengthDisplayStyle, \arrowHeightDisplayStyle);
		\end{tikzpicture}
		\hspace*{\arrowSkipDisplayStyle}
	}
	{ 
		\hspace*{\arrowSkipTextStyle}
		\begin{tikzpicture}[baseline]
		\draw [#1] (0,\arrowHeightTextStyle) -- node [fill=white,inner sep=1pt] {$\scriptstyle #2$} (#3 * \arrowLengthTextStyle, \arrowHeightTextStyle);
		\end{tikzpicture}
		\hspace*{\arrowSkipTextStyle}
	}
	{ 
		\hspace*{\arrowSkipScriptStyle}
		\begin{tikzpicture}[baseline]
		\draw [#1] (0,\arrowHeightScriptStyle) -- node [fill=white,inner sep=1pt] {$\scriptscriptstyle #2$} (#3 * \arrowLengthScriptStyle, \arrowHeightScriptStyle);
		\end{tikzpicture}
		\hspace*{\arrowSkipScriptStyle}
	}
	{ 
		\hspace*{\arrowSkipScriptScriptStyle}
		\begin{tikzpicture}[baseline]
		\draw [#1] (0,\arrowHeightScriptScriptStyle) -- node [fill=white,inner sep=1pt] {$\scriptscriptstyle #2$} (#3 * \arrowLengthScriptScriptStyle, \arrowHeightScriptScriptStyle);
		\end{tikzpicture}
		\hspace*{\arrowSkipScriptScriptStyle}
	}
}

\def\arrow#1{\def\lastArrowStyle{#1}
	\futurelet\testchar\arrowMaybeStreched}
\def\arrowMaybeStreched{\ifx[\testchar \let\next\arrowStreched
	\else \let\next\arrowUnstreched \fi
	\next}

\def\arrowStreched[#1]{\def\lastArrowStrech{#1}
	\futurelet\testchar\arrowMaybeLabel}
\def\arrowUnstreched{\def\lastArrowStrech{1}
	\futurelet\testchar\arrowMaybeLabel}

\def\arrowMaybeLabel{\ifx^\testchar \let\next\arrowSuperscript
	\else \ifx_\testchar \let\next\arrowSubscript
	\else \ifx~\testchar \let\next\arrowCentralLabel
	\else \let\next\arrowNoLabel
	\fi
	\fi
	\fi
	\next}

\def\arrowSuperscript^#1{\def\lastArrowSuperscript{#1}
	\futurelet\testchar\arrowSuperMaybeSub}
\def\arrowSuperMaybeSub{\ifx_\testchar \let\next\arrowSuperscriptSubscript
	\else \let\next\arrowSuperscriptNoSubscript \fi
	\next}

\def\arrowSubscript_#1{\def\lastArrowSubscript{#1}
	\futurelet\testchar\arrowSubMaybeSuper}
\def\arrowSubMaybeSuper{\ifx^\testchar \let\next\arrowSubscriptSuperscript
	\else \let\next\arrowSubscriptNoSuperscript \fi
	\next}

\def\arrowSuperscriptSubscript_#1{\def\lastArrowSubscript{#1}
	\arrowDrawSupSub}
\def\arrowSuperscriptNoSubscript{\def\lastArrowSubscript{}
	\arrowDrawSupSub}
\def\arrowSubscriptSuperscript^#1{\def\lastArrowSuperscript{#1}
	\arrowDrawSupSub}
\def\arrowSubscriptNoSuperscript{\def\lastArrowSuperscript{}
	\arrowDrawSupSub}
\def\arrowNoLabel{\def\lastArrowSuperscript{}
	\def\lastArrowSubscript{}
	\arrowDrawSupSub}

\def\arrowCentralLabel~#1{\MakeTikzArrowWithCentralLabel{\lastArrowStyle}{#1}{\lastArrowStrech}}
\def\arrowDrawSupSub{\MakeTikzArrowWithSuperscriptSubscript{\lastArrowStyle}{\lastArrowSuperscript}{\lastArrowSubscript}{\lastArrowStrech}}

\tikzcdset{arrow style=tikz, diagrams={>=stealth'}, row sep=4em, column sep=4em}

\counterwithin*{equation}{section}

\newtheorem{theorem}{Theorem}[section]
\newtheorem{lemma}[theorem]{Lemma}
\newtheorem{corollary}[theorem]{Corollary}
\newtheorem{proposition}[theorem]{Proposition}

\theoremstyle{definition}

\newtheorem{definition}[theorem]{Definition}
\newtheorem{remark}[theorem]{Remark}

\newtheorem{setup}[theorem]{Setup}
\newtheorem{proposition-definition}[theorem]{Proposition-Definition}

\makeatletter
\newtheorem*{rep@theorem}{\rep@title}
\newcommand{\newreptheorem}[2]{%
	\newenvironment{rep#1}[1]{%
		\def\rep@title{#2 \ref{##1}}%
		\begin{rep@theorem}}%
		{\end{rep@theorem}}}
\makeatother

\newreptheorem{theorem}{Theorem}
\newreptheorem{lemma}{Lemma}
\newreptheorem{corollary}{Corollary}
\newreptheorem{setup}{Setup}

\mathchardef\mhyphen="2D 
\newlist{subenumerate}{enumerate}{1}
\setlist[subenumerate,1]{label=(\arabic*)}
\newcommand\defeq{\stackrel{\mathrm{\mbox{\scriptsize{def}}}}{\,=\,}}
\newcommand{\sth}{\,\,\vert\,\,}

\newcommand{\Z}{\mathbb{Z}}
\DeclareMathOperator{\ind}{ind}

\DeclareMathOperator{\add}{add}
\DeclareMathOperator{\thick}{thick}
\DeclareMathOperator{\Hom}{Hom}
\DeclareMathOperator{\Ext}{Ext}
\DeclareMathOperator{\End}{End}
\DeclareMathOperator{\gen}{gen}
\DeclareMathOperator{\Der}{\mathcal{D}}

\DeclareMathOperator{\per}{per}

\DeclareMathOperator{\fgMod}{\mathrm{mod}}
\DeclareMathOperator{\Kb}{\mathcal{K}^b}
\DeclareMathOperator{\proj}{proj}

\newcommand{\T}{\mathcal{T}}

\newcommand{\A}{\mathcal{A}}

\newcommand{\p}{\mathcal{P}}

\newcommand{\z}{\mathcal{Z}}
\newcommand{\x}{\mathcal{X}}

\newcommand{\s}{\mathcal{S}}

\renewcommand{\c}{\mathcal{C}}
\renewcommand{\d}{\mathcal{D}}
\newcommand{\U}{\mathcal{U}}
\DeclareMathOperator{\tcmc}{\mathfrak{W}}

\renewcommand{\c}{\mathcal{C}}
\DeclareMathOperator{\Loc}{Loc}
\newcommand{\presilt}{\mathrm{presilt}}
\newcommand{\silt}{\mathrm{silt}}
\newcommand{\npresilt}[2]{#1_{#2}\mhyphen\mathrm{presilt}}
\newcommand{\nsilt}[2]{#1_{#2}\mhyphen\mathrm{silt}}
\DeclareMathOperator{\wide}{wide}

\DeclareMathOperator{\stautiltpair}{s\tau\mhyphen tilt{\,}pair}
\DeclareMathOperator{\staurigidpair}{\tau\mhyphen rigid{\,}pair}

\DeclareMathOperator{\ftors}{f\mhyphen tors}

\newcommand{\w}{\mathcal{W}}

\DeclareMathOperator{\texact}{t\mhyphen exact}

\makeatletter
\newsavebox{\@brx}
\newcommand{\llangle}[1][]{\savebox{\@brx}{\(\m@th{#1\langle}\)}%
  \mathopen{\copy\@brx\kern-0.5\wd\@brx\usebox{\@brx}}}
\newcommand{\rrangle}[1][]{\savebox{\@brx}{\(\m@th{#1\rangle}\)}%
  \mathclose{\copy\@brx\kern-0.5\wd\@brx\usebox{\@brx}}}
\makeatother

\title[Two-term silting and $\tau$-cluster morphism categories]{Two-term silting and $\tau$-cluster morphism categories}

\author{Erlend D. Børve}

\address{Department of Mathematical Sciences, NTNU, NO-7491 Trondheim, Norway (Current address: Institut for Matematik, Aarhus Universitet, Ny Munkegade 118, 8000 Aarhus C, Denmark).}
\email{erlend.d.borve@gmail.com} 
\begin{document}

		\maketitle
		
			\begin{abstract}
			We generalise $\tau$-cluster morphism categories to {the setting of} non-positive dg algebras {with finite dimensional cohomology in all degrees}. {The objects in this category will be certain t-exact subcategories of the bounded derived category, and the morphisms are determined by two-term presilting objects.}
			The compatibility of silting reduction with support $\tau$-tilting reduction will be an essential ingredient when linking our definition to {those} of Buan--Marsh {and Buan--Hanson}.
		\end{abstract}
		
		\textbf{Acknowledgements. }
		I would like to express my gratitude to my PhD supervisor, Aslak B. Buan, for introducing me to the topic, countless helpful discussions, and meticulous proofreading. I am also grateful to Eric J. Hanson, {Yann Palu, and anonymous referees} for very helpful comments on previous versions of this manuscript. {I would also like to thank Gustavo Jasso for useful conversations.}
		
		\textbf{Statements and Declarations. }
		A previous version of this manuscript is part of the author's PhD thesis. The author was funded directly by his employer, NTNU.
	
	\tableofcontents
	
	\newpage
	
\section{Introduction}

In tilting theory, it is valuable to find all tilting modules over a fixed finite dimensional algebra $\Lambda$, or better still, all tilting objects in the bounded homotopy category $\Kb(\proj {\Lambda})$, where $\proj {\Lambda}$ is the category of finitely generated projective $\Lambda$-modules. If all tilting objects in $\Kb(\proj {\Lambda})$ could be identified, one could construct all finite dimensional algebras $\Gamma$ that are derived Morita equivalent to ${\Lambda}$ \cite{Ric89}, which is to say that their derived categories are equivalent as triangulated categories. Many invariants of finite dimensional algebras are preserved by derived Morita equivalence, and it might be more efficient to compute these for some $\Gamma$ rather than for $\Lambda$ directly.

Finding all tilting objects turns out to be an unrealistic project in general. However, one could instead study the larger class of silting objects, for which there is  a mutation process generating new examples in a reliable manner \cite{AI12}. In addition, one has a reduction procedure for silting objects. Iyama--Yang show that if $P$ is a presilting object of $\Kb(\proj {\Lambda})$, then the (pre)silting objects in $\Kb(\proj {\Lambda})$ having $P$ as a direct summand {bijectively} correspond to the (pre)silting objects in the Verdier quotient $\Kb(\proj {\Lambda})/\thick(P)$ \cite[Theorem~3.7]{IY18}. 
	The class of two-term silting objects closely resembles that of tilting modules. Indeed, a generalisation of the Brenner--Butler Theorem applies \cite[Theorem~1.1]{BZ16} \cite[Theorem~2.1]{BZ21}. Also, when performing silting reduction with respect to a two-term presilting object, one can restrict Iyama--Yang's bijection to a bijection of two-term (pre)silting objects \cite[Proposition~4.11]{Jas15}. 
	
	The mutation of two-term silting objects is closely tied to cluster combinatorics. For path algebras of a quiver of simply laced Dynkin type, the two-term silting objects of the bounded homotopy category are in bijection with the cluster-tilting objects in the cluster category \cite[Theorem~4.1]{AIR14}. Moreover, this bijection respects the mutation process for cluster-tilting objects \cite[Corollary~4.8]{AIR14}. Cluster-tilting objects are in turn in bijection with the clusters of the corresponding cluster algebra \cite[Corollary~4.4]{BMRRT06}.
	
	One may also model cluster combinatorics using $\tau$-tilting theory \cite{AIR14}. Adachi--Iyama--Reiten define support $\tau$-tilting modules and $\tau$-rigid pairs, and show that they correspond bijectively with {two-term} silting objects and {two-term} presilting objects, respectively \cite[Theorem~3.2]{AIR14}. The reduction procedure of Iyama--Yang is transferable to $\tau$-tilting theory. Indeed, given any finite dimensional algebra, Jasso provides a suitable bijection between classes of support $\tau$-tilting modules \cite[Theorem {3.16}]{Jas15}, and shows that it is compatible with Iyama--Yang's silting reduction \cite[Theorem~4.12(b)]{Jas15}. Buan--Marsh later extended Jasso's bijection to $\tau$-rigid pairs \cite[Theorem~3.6]{BM18w}.
	
	The bijection of Buan--Marsh was constructed in order to define the $\tau$-cluster morphism category of a $\tau$-tilting finite algebra. This builds on Igusa--Todorov's cluster morphism category, {which is} defined for representation finite hereditary algebras \cite[Section~1]{IT17}, as well as a categorification of non-crossing partitions \cite{Igu14}.
	The definition of $\tau$-cluster morphism categories has recently been extended to all finite dimensional algebras \cite{BH21}. Although Buan--Marsh define the $\tau$-cluster morphism category in terms of $\tau$-tilting theory, they occasionally translate the setting to two-term silting in order to prove necessary results.
	
	$\tau$-cluster morphism categories are useful in the study of picture groups and picture spaces, defined by Igusa--Todorov--Weyman \cite{ITW16}. In the case of representation finite hereditary algebras, we have that the geometric realisation of the $\tau$-cluster morphism category is a $K(\pi,1)$ for the corresponding picture group \cite[Theorem~3.1]{IT17}. This result was later extended to include Nakayama algebras \cite[Theorem~4.16]{HI21}.
	
	\textbf{Outline.}
	\Cref{sec:prelims} sets up the notation and covers the necessary preliminaries. We review how t-exact subcategories (see \Cref{def:texact}) of a triangulated category with a bounded t-structure correspond with wide subcategories in the heart (\Cref{ZC17.2.5}), the interaction between two-term silting and $\tau$-tilting theory, and $\tau$-tilting reduction. In particular, we recall in \Cref{Jas15.3.6} that a $\tau$-rigid pair $(M,Q)$ in $\fgMod {\Lambda}$, {where $\Lambda$ is a finite dimensional algebra,} gives rise to a wide subcategory $J(M,Q)\defeq M^{\perp}\cap {^{\perp}(\tau M)}\cap Q^{\perp}$ of $\fgMod{\Lambda}$.
	
	In \Cref{sec:silttstr}, we work over a non-positive dg algebra $A$ {such that $H^iA$ is finite dimensional for all $i\in \Z$}. Let $\Der_{\rm fd}(A)$ be the subcategory of the derived category $\Der(A)$ spanned by complexes of finite dimensional total cohomology. In \Cref{lem:perptexsummand}{\eqref{lem:perptexsummand2}, we show that a two-term presilting object $P\in \per(A)$ determines a t-exact subcategory $$P^{\perp_{\Z}}_{\rm fd}\defeq \{X\in \Der_{\rm fd}(A) \sth \Hom_{ \Der(A)}(P,\Sigma^i X)=0 \,\, \forall i\in \Z \}$$ of $\Der_{\rm fd}(A)$}. By \Cref{prop:H0J}{\eqref{item:H0J1}}, the {subcategory} $P^{\perp_{\Z}}_{{\rm fd}}{\cap \fgMod H^0A \subseteq \fgMod H^0A}$ is {equal to} the $J(M,Q)$, where $(M,Q)$ is the $\tau$-rigid pair corresponding to $P$ {(see \Cref{IJY14.4.6})}.
	
	\Cref{sec:comp} presents a proof that Iyama--Yang's silting reduction is compatible with Buan--Marsh's support $\tau$-tilting reduction, generalising Jasso's compatibility result \cite[Theorem 4.12(b)]{Jas15}. 	
	\begin{reptheorem}[\Cref{th:Jas15.4.12}]
		{Let $A$ be a non-positive dg algebra such that $H^iA$ is finite dimensional for all $i\in \Z$.}
		Consider a {two}-term presilting object $P$ {in $\per(A)$}, 
		{let $T^{+}_P$ be the Bongartz completion of $P$ (defined by the triangle \eqref{eq:Bon-tri}), let $\z_P={^{\perp_{>0}}P}\cap {P^{\perp_{>0}}}{\subseteq \per(A)}$}, and let ${\Gamma}$ be the $\tau$-tilting reduction of ${H^0A}$ with respect to the $\tau$-rigid pair {$H_{{A}}(P)$} {in $\fgMod{\Lambda}$} corresponding to $P$. We have a commutative diagram of bijections
		\begin{center}
			\begin{tikzcd}
				\npresilt{2}{}_{{ P}}({A}) \arrow[r,"H_{{{A}}}"]\arrow[d,"\varphi_{ P}"] & \staurigidpair_{H_{{A}}({ P})}({{H^0A}})\arrow[d,"\psi_{H_{{A}}({ P})}"] \\
				\npresilt{2}{ T^{+}_P}(\z_{ P}/[{ P}]) \arrow[r,"H_{ T^{+}_P}"] & \staurigidpair({\Gamma}) 
			\end{tikzcd}
		\end{center}
		where $\varphi_P$ is the Iyama--Yang bijection and $\psi_{H_{{A}}({ P})}$ is the Buan--Marsh bijection. The horizontal maps are the {bijections} between two-term presilting objects and $\tau$-rigid pairs (see \Cref{IJY14.4.6}).
	\end{reptheorem}
	
	Our main result appears in \Cref{section:tcmc}. Given a non-positive dg algebra $A$ {such that $H^iA$ is finite dimensional for all $i\in \Z$}, we construct a category $\tcmc_{A}$, which we call the \textit{$\tau$-cluster morphism category} of $A$. Developing the theory purely in terms of {two-term} silting, as opposed to $\tau$-tilting, has a major advantage; since silting reduction is induced by a functor, it is easier to prove that the composition law in $\tcmc_{A}$ is associative, as we do in \Cref{thm:associative}.
	When we claim that the work of Buan--Marsh and Buan--Hanson has been generalised, it is in the following specific sense.
	
	\begin{reptheorem}[\Cref{thm:sameasBM}]
		Let $A$ be a non-positive dg algebra {such that $H^iA$ is finite dimensional for all $i\in \Z$}, and let $\tcmc^{\rm BM}_{H^0A}$ be the $\tau$-cluster morphism category (as defined by Buan--Hanson \cite{BH21}) of the zeroth cohomology of $A$. 
		{There is an} equivalence of categories
		\begin{equation*}
		{E\colon }\tcmc_{A}\to \tcmc^{\rm BM}_{H^0A}.
		\end{equation*}
		In particular, if $A$ is a finite dimensional algebra, then $\tcmc_{A}$ is equivalent to $\tcmc_{A}^{\rm BM}$.
	\end{reptheorem}
	
		\section{Notation and preliminaries}\label{sec:prelims}
	
	Throughout, all subcategories {we consider} will be full and closed under isomorphism. We fix a field $k$, and declare all triangulated categories and their (not necessarily triangulated) subcategories to be $k$-linear. {A $k$-linear category $\c$ is \textit{Hom-finite} if the Hom-spaces $\c({ X},{ X}')$ (alternatively denoted $\Hom_{\c}({ X},{ X}')$) are finite dimensional for all $X,X'\in \c$. It is \textit{Krull--Schmidt} if any object ${ X}\in\c$ is isomorphic to a finite direct sum $\bigoplus_{i=1}^n { X}_i$, where the endomorphism $k$-algebra of each ${ X}_i$ is a local. An object in a Krull--Schmidt $k$-linear category with a local endomorphism $k$-algebra is called \textit{indecomposable}.}
	
	Recall that a triangulated subcategory is \textit{thick} if it is closed under direct summands.
		For a triangulated category $\T$, we write $\thick_\T(\x)$ for the smallest thick subcategory containing the subcategory $\x$, or simply $\thick(\x)$ when there is no risk of confusion. {If $\mathcal{X}$ consists of a single object $X$, we write $\thick_{\T}(X)$ when we mean $\thick_{\T}(\x)$.}
	
	All modules are right modules, unless otherwise specified. For a $k$-algebra ${\Lambda}$, the category of finitely generated right ${\Lambda}$-modules will be denoted by $\fgMod {\Lambda}$, and the category of {finitely generated} projective {right} ${\Lambda}$-modules by $\proj {\Lambda}$. The triangulated category $\Der_{{\rm fd}}(\fgMod {\Lambda})$ {denotes} the bounded derived category of the former, whereas $\Kb(\proj {\Lambda})$ {denotes} the bounded homotopy category of the latter. {They are both subcategories of the derived category $\Der({\Lambda})$. More generally, if $A$ is a differential graded (henceforth \textit{dg}) $k$-algebra $A$, we denote its derived category by $\Der(A)$, define $\per(A)$ to be {$\thick_{\Der(A)}(A)$}, and $\Der_{\rm fd}{(A)}$ to be the (thick) subcategory {of $\Der(A)$} spanned by complexes {with} finite dimensional total cohomology.} {A dg $k$-algebra $A$ is said to be \textit{non-positive} if the cohomology $k$-vector spaces $H^iA$ vanish for all $i>0$.}
	
	Let $\mathcal{C}$ be a $k${-}linear category and let $\mathcal{X}\subseteq\mathcal{C}$ be a full subcategory. The \textit{additive closure} of $\mathcal{X}$ is the full subcategory of $\mathcal{C}$ {consisting of} all finite direct sums of direct summands of objects in $\mathcal{X}$. We denote this category by $\add(\mathcal{X})$, or by $\add(X)$ if $\mathcal{X}$ contains {only} a single object $X$.
	
	Let $\mathcal{P}$ be an additive subcategory of $\mathcal{C}$. A morphism $Q\to^{\beta}Y$ (resp. $X\to^{\alpha}Q$) is \textit{$\mathcal{P}$-epic} (resp. \textit{$\mathcal{P}$-monic}) if the induced morphism ${\mathcal{C}}(P,\beta)$ (resp. ${\mathcal{C}}(\alpha,P)$) is surjective for all objects $P\in\mathcal{P}$. It is a \textit{right $\mathcal{P}$-approximation} of $Y$ (resp. \textit{left $\mathcal{P}$-approximation} of $X$) if, in addition, we have that $Q\in\mathcal{P}$. We say that $\mathcal{P}$ is a \textit{contravariantly finite subcategory} (resp. \textit{covariantly finite subcategory}) of $\mathcal{C}$ if every object in $\mathcal{C}$ admits a right (resp. left) $\mathcal{P}$-approximation. A \textit{functorially finite subcategory} of $\mathcal{C}$ is one that is both contravariantly and covariantly finite.
	
	The \textit{left} and \textit{right perpendicular subcategories} of $\p$ in $\c$ are the full subcategories
	\begin{align*}
	{^{\perp}\p}&\defeq\{X\in\c\sth {\c}(X,\p)=0\}, \\
	{\p^{\perp}}&\defeq\{Y\in\c\sth {\c}(\p,Y)=0\},
	\end{align*}
	{respectively. If $\p=\add(P)$ for some object $P\in\c$, we write ${^{\perp}P}\defeq {^{\perp}(\add(P))}$ and ${P^{\perp}}\defeq (\add(P))^{\perp}$, respectively.}
	
	If $\T$ is a triangulated category, we denote its suspension functor by $\Sigma$, unless otherwise specified. {Let $\p$ be an additive subcategory of $\T$.} For each subset $I\subseteq \Z$, we define the \textit{{left and right} perpendicular subcategories}
	\begin{align}
	{^{\perp_I}\p}&\defeq\{X\in\T\sth {\T}(X,\Sigma^i\p)=0\quad \forall i\in I\}, \nonumber \\
	{\p^{\perp_I}}&\defeq\{Y\in\T\sth {\T}(\p,\Sigma^i Y)=0\quad \forall i\in I\},  \label{eq:perpr}
	\end{align}
	{respectively. If $\p=\add(P)$ for some object $P\in\c$, we write ${^{\perp_I}P}\defeq {^{\perp_I}(\add(P))}$ and ${P^{\perp_I}}\defeq (\add(P))^{\perp_I}$, respectively. In the case where $I$ is the subset of $\Z$ consisting of positive (resp. non-positive, resp. non-zero) integers, we replace the subscript $I$ by $>0$ (resp. $\leq 0$, resp. $\neq 0$), which for example introduces the notation $\p^{\perp_{>0}}$ (resp. $\p^{\perp_{\leq 0}}$, resp. $\p^{\perp_{\neq 0}}$) for the right perpendicular subcategory in question.}
	
	Suppose that $\T$ is triangulated and let $\p_1$ and $\p_2$ be additive subcategories of $\T$. The \textit{category of extensions {of $\p_2$ in $\p_1$}} {(also known as the \textit{Verdier product} of $\p_1$ and $\p_2$)} is the full subcategory $\p_1*\p_2$ of $\T$ {consisting of} the objects $E$ that fit in a triangle 
	\begin{center}
		\begin{tikzcd} X_1 \arrow[r,""]& E\arrow[r]& X_2 \arrow[r]& \Sigma X_1,  \end{tikzcd}
	\end{center}
	where $X_1\in\p_1$ and $X_2\in\p_2$. It can be shown using the octahedral axiom that $*$ is an associative operation \cite[Lemme~1.3.10]{BBD82}. 
	A full subcategory $\p\subseteq\T$ is \textit{extension-closed in $\T$}  if $\p=\p*\p$. {One also has a notion of extension-closure for subcategories of Quillen exact categories (e.g. abelian categories), where exact sequences play the role of triangles.}
	
	An additive subcategory $\mathcal{P}\subseteq \T$ is \textit{suspended} (resp. \textit{co-suspended}) if it is extension-closed and {closed under} the suspension functor $\Sigma$ (resp. inverse {suspension functor} $\Sigma^{-1}$). The ideal $[\p]$ in $\c$ contains precisely the morphisms that factor through an object in $\p$. The ideal quotient will be denoted by ${\c \over [\p]}$. If $\p=\add(P)$ for some object $P$, we denote this quotient by ${\c \over [P]}$.
	
	We use Deligne's convention, writing $F=G$ when we mean that these functors are naturally isomorphic. 
	
	\subsection{t-structures and wide subcategories of the heart}
	Truncation structures (t-structures for short) were introduced by Beilinson--Bernstein--Deligne--Gabber \cite{BBD82,BBDG18}. We give a quick survey of their elementary properties. A triangulated category $\d$ {with suspension functor $\Sigma$} is fixed.
	
	\begin{definition}\label{def:t-str}
		A pair of full additive subcategories $(\d^{\leq 0},\d^{> 0})$ of $\d$ constitute a \textit{truncation structure} (henceforth \textit{t-structure}) if the conditions \ref{t1} and \ref{t2} below are met.
		\begin{enumerate}[label=(t.\arabic*)]
			\item\label{t1} $(\d^{\leq 0},\d^{> 0})$ is a \textit{torsion pair} in $\d$, i.e. 
			\begin{enumerate}[label=(t.1.\arabic*)]
				\item\label{t11} The Hom-spaces $\d(X,Y)$ are trivial for all $X\in \d^{\leq 0}$ and $Y\in \d^{> 0}$,
				\item\label{t12} every object in $\d$ is an extension of an object in $\d^{> 0}$ with an object in $\d^{\leq 0}$. This is to say that $\d^{\leq 0}*\d^{>0}=\d$.
			\end{enumerate}
			\item\label{t2} The subcategory $\d^{\leq 0}$ is suspended (equivalently, the subcategory $\d^{> 0}$ is co-suspended).
		\end{enumerate}
		We define $\d^{\leq -1}\defeq \Sigma^{}\d^{\leq 0}$. More generally, we define
		\begin{align*}
		\d^{\leq n}&\defeq \Sigma^{-n}\d^{\leq 0}, \\
		\d^{> m}&\defeq \Sigma^{-m}\d^{> 0}, \\ \d^{[m,n]}&\defeq \d^{> m-1}\cap\d^{\leq n}, \\ \text{and } \d^{m}&\defeq \d^{[m,m]},
		\end{align*}
		for integers $n$ and $m$. The t-structure $(\d^{\leq 0},\d^{> 0})$ is \textit{bounded} if $\bigcup\limits_{m,n\in\Z}\d^{[m,n]}=\d$. The subcategory $\d^0$ of $\d$ is called the \textit{heart} of the t-structure.
	\end{definition}
	
	\begin{remark}\label{rem:closedunder}
		The subcategories $\d^{\leq 0}$ and $\d^{>0}$ determine each other; we have that $(\d^{\leq 0})^{\perp_0} = \d^{>0}$ and that ${^{\perp_0}(\d^{> 0})} = \d^{\leq 0}$. As a result of this fact, the subcategories $\d^{\leq 0}$ and $\d^{>0}$ are closed under taking direct summands. 
	\end{remark}
	
	Let $\A$ be an abelian category. Then the derived category {$\Der(\A)$} can be equipped with the \textit{standard {t-}structure}:
	\begin{align*}
	\Der^{\leq 0}(\A)  &=  \{X\in\Der(\A)\sth {H^i X}=0 \text{ for all } i>0\}, \\
	\Der^{> 0}(\A)  &=  \{X\in\Der(\A)\sth {H^i X}=0 \text{ for all } i\leq 0\},
	\end{align*}
	{where $H^i$ is a cohomology functor.}
	{It restricts to a bounded t-structure on $\Der_{\rm fd}(\A)$, also called the \textit{standard {t-}structure}.}
	The hearts {of these t-structures are both} equivalent to the abelian category $\A$. In general, the heart of any t-structure is an abelian category {\cite[Théorème 1.3.6]{BBD82}}.
		
	For a suspended subcategory $\mathcal{U}\subseteq \d$, the pair $(\mathcal{U},\mathcal{U}^{\perp_{{0}}})$ forms a t-structure on $\d$ if and only if $\U$ is \textit{co-reflective} \cite[Proposition 1.1]{KV88}, i.e. the inclusion admits a right adjoint.
	A subcategory of $\d$ is called an \textit{aisle} (resp. a \textit{co-aisle}) if it is suspended (resp. co-suspended) and co-reflective (resp. \textit{reflective}, i.e. the inclusion admits a left adjoint).
	Since the pair $(\d^{\leq 0},\d^{>0})$ is a t-structure on $\d$ precisely when $\d^{\leq 0}$ is an aisle in $\d$ (or, equivalently $\d^{> 0}$ is a co-aisle), we say that $\d^{\leq 0}$ is the \textit{aisle of the t-structure} (and that $\d^{> 0}$ is the \textit{co-aisle}). 
	
	We may hence define a \textit{truncation functor} $\sigma^{\leq 0}\colon{\d}\to {\d^{\leq 0}}$ (resp. $\sigma^{> 0}\colon{\d}\to {\d^{> 0}}$) as the right adjoint of the inclusion functor $\iota^{\leq 0}\colon{\d^{\leq 0}}\to {\d}{}$ (resp. as the left adjoint of the inclusion functor $\iota^{> 0}\colon{\d^{> 0}}\to{\d}{}$). {More generally, one defines truncation functors $\sigma^{\leq i}\colon{\d}\to {\d^{\leq i}}$ (resp. $\sigma^{> i}\colon{\d}\to {\d^{> i}}$) for any integer $i$, \textit{mutatis mutandis}. {The \textit{cohomology functor in degree 0} of the t-structure is defined by ${\sigma}^0\defeq \sigma^{\leq 0}\sigma^{>-1}$}. For any integer $i$, we define the \textit{cohomology functor} in degree $i$ by ${\sigma}^i\defeq {\sigma}^0\circ \Sigma^i$.}
	
	\begin{remark}
		It is conventional to denote the truncation functors by $\tau^{\leq 0}$ and $\tau^{>0}$. We have chosen to use $\sigma^{\leq 0}$ and $\sigma^{>0}$ instead, so that they will not be confused with the Auslander--Reiten translation.
	\end{remark}
	
	For all objects $X\in\d$, there is a unique morphism ${\sigma^{> 0}X}\to^{\partial_X}\Sigma {\sigma^{\leq 0} X}{}$ such that we have a triangle 
	\begin{equation}\label{eq:trunca}
		\begin{tikzcd}
			\sigma^{\leq 0}X\arrow[r,"\varepsilon_X"] & X\arrow[r,"\eta_X"] & \sigma^{> 0}X \arrow[r,"\partial_X"] & \Sigma\sigma^{\leq 0} X,
		\end{tikzcd}
	\end{equation}
	where $\varepsilon_X$ (resp. $\eta_X$) is induced by the co-unit of the adjunction $(\iota^{\leq 0},\sigma^{\leq 0})$ (resp. by the unit of the adjunction $(\sigma^{> 0},\iota^{> 0})$). {We will refer to the triangle in \eqref{eq:trunca} as the \textit{truncation triangle of $X$.}}	
	\begin{proposition}\label{prop:BBD2.136}
		Let $\d$ be a triangulated category {equipped} with a t-structure $(\d^{\leq 0},\d^{>0})$.
		\begin{enumerate}
			\item\label{prop:BBD2.1361} {\cite[Théorème~1.3.6]{BBD82}} The functor ${\sigma}^0= \sigma^{\leq 0}\sigma^{> -1}\colon \d\to \d^0$ is \textit{cohomological}, i.e. it sends distinguished triangles to long exact sequences. {Since $\Sigma^i$ is an auto-equivalence on $\d$ for every integer $i$, it follows that any cohomology functor ${\sigma}^i={\sigma}^0\circ \Sigma^i$ is cohomological.}

			\item\label{prop:BBD2.1362} {\cite[Remarque~3.1.17(ii)]{BBD82}} A complex 
			\begin{center}
				\begin{tikzcd}0\arrow[r]&U\arrow[r,"i"]& V\arrow[r,"p"]& W\arrow[r]&0\end{tikzcd}
			\end{center}
			in $\d^0$ is an exact sequence if and only if there exists a morphism ${W\to^{\partial} \Sigma U}$ in $\d$ such that 
			\begin{center}
				\begin{tikzcd}U\arrow[r,"i"]& V\arrow[r,"p"]& W\arrow[r,"\partial"]&\Sigma U\end{tikzcd}
			\end{center}
			is a distinguished triangle in $\d$.
			\item\label{prop:BBD2.1364}{\cite[Théorème~1.3.3(ii)]{BBD82} Let $X\in \d$. The truncation triangle {in} \eqref{eq:trunca} is the unique triangle up to isomorphism where the first term is in $\d^{\leq 0}$, the second term is $X$, and the third term is in $\d^{> 0}$. }
		\end{enumerate}
	\end{proposition}
	
	\begin{definition}[{\cite[Definition~1.3.16]{BBD82}}]\label{def:texact}
		Let $\d$ be a triangulated category {equipped} with a t-structure $(\d^{\leq 0},\d^{> 0})$, and let $\U$ be a triangulated category {equipped} with a t-structure $(\U^{\leq 0},\U^{> 0})$. A triangle functor $F\colon{\U}\to {\d}{}$ is 
		\begin{enumerate}
			\item \textit{left t-exact} if $F(\U^{> 0})\subseteq \d^{> 0}$,
			\item \textit{right t-exact} if $F(\U^{\leq 0})\subseteq \d^{\leq 0}$,
			\item \textit{t-exact} if it is both left t-exact and right t-exact.
		\end{enumerate}
		A \textit{t-exact subcategory} of $\d$ with respect to the t-structure $(\d^{\leq 0},\d^{> 0})$ is a triangulated subcategory $\s$ of $\d$ such that $\sigma^{\leq 0}X\in \s$ for any $X\in\s$.
	\end{definition}
	
	\begin{remark}\label{rem:tex}
		We might as well have defined a t-exact subcategory to be a triangulated subcategory which is closed under any truncation functor of the form $\sigma^{\leq i}$ or $\sigma^{> i}$, not just $\sigma^{\leq 0}$. First of all, a t-exact subcategory $\s\subseteq\d$ is closed under truncation by $\sigma^{\leq m}$, for all $m\in\Z$, since we have assumed that $\s$ is closed under suspension and desuspension.
		It is also {closed under} $\sigma^{>0}$ (and thus $\sigma^{> m}$, for all $m\in\Z$), which one proves by forming a truncation triangle 
		\begin{center}
			\begin{tikzcd} \sigma^{\leq 0} X \arrow[r,"\varepsilon_X"]& X \arrow[r,""]& \sigma^{> 0}X \arrow[r,""] &  \Sigma\sigma^{\leq 0} X, \end{tikzcd}
		\end{center}
		where $X\in \s$, and then pointing out that $\sigma^{> 0}X$ is a cone of $\varepsilon_X$, a morphism in {the triangulated subcategory} $\s$.
	\end{remark}
	
	{A triangulated subcategory $\s$ of $\d$ is t-exact precisely when we can induce a t-structure $(\d^{\leq 0}\cap~\s,\d^{> 0}\cap~\s)$ on $\s$ such that the inclusion functor $\s\to \d$ is t-exact \cite[{§1.3.19}]{BBD82}.} The heart of the induced t-structure is {$\s^0\defeq \d^0\cap \s$}, and the cohomological functor $\s\to\s^0$ is naturally isomorphic to the restriction of ${\sigma}^0\colon \d \to \d^0$. 
		
	We include the following well-known result to emphasise that t-exact {functors} induce {exact functor of hearts}.
	\begin{proposition}[{\cite[Proposition~1.3.17]{BBD82}}]\label{prop:BBD1.3.7}
		Let $F\colon \U \to \d$ be a t-exact functor. Then the composite functor
		\begin{center}
			\begin{tikzcd}
				\U^0 \arrow[r,hook] & \U \arrow[r,"F"] & \d \arrow[r,"{\sigma}^0"] & \d^0
			\end{tikzcd}
		\end{center}
		is an exact functor of abelian categories. In particular, if $\s$ is a t-exact subcategory of $\d$, then $\s^0$ is an exact abelian subcategory of $\d^0$.
	\end{proposition}
	
	The subcategory $\d^0=\d^{\leq 0}\cap \d^{>0}$ is an intersection of extension-closed subcategories. {If} $\s$ is a triangulated subcategory of $\d$, the intersection $\d^0\cap \s$ is extension-closed in {the triangulated category $\d$. It now follows from \Cref{prop:BBD2.136}\eqref{prop:BBD2.1362} that $\s^0= \d^0\cap \s$ is extension-closed in {the abelian category} $\d^0$}. {An extension-closed} exact abelian subcategory is often called a \textit{wide subcategory} (or a \textit{weak Serre subcategory}). Equivalently, it is a subcategory which is closed under kernels, cokernels and extensions. {In particular, every idempotent endomorphism admits a kernel. {This is to say that} wide subcategories are \textit{idempotent complete}. It follows that wide subcategories are closed under taking direct summands.}
	
	\begin{lemma}\label{lem:thicktex}
		Let $\d$ be a triangulated category equipped with a bounded t-structure $(\d^{\leq 0},\d^{>0})$, and let $\s$ be a t-exact subcategory of $\d$. Then $\s$ is thick.
	\end{lemma}
	\begin{proof}
		It is to be shown that $\s$ is closed under direct summands. Let $X,Y\in \d$ be such that $X\oplus Y\in \s$. Since $(\d^{\leq 0},\d^{>0})$ is bounded, only finitely many integers $i$ are such that ${\sigma}^{i}(X\oplus Y)$ is non-zero. We proceed by induction on the number $d$ of such integers $i$.
		
		The case $d=0$ concerns only zero objects, and is therefore trivial to establish. If $d=1$, we have that $X\oplus Y\in \s^i$ for some $i$. Since $\s^i$ is equivalent to $\s^0$, we may assume without loss of generality that $i=0$. {Having remarked above that wide subcategories are idempotent complete,} it follows that the direct summand $X$ of $X\oplus Y$ is in {the wide subcategory} $\s^0$, as desired.
		
		Assume that the statement is true whenever $d<\ell$ for some $\ell\geq 2$. If $d=\ell$, {suppose that $X\oplus Y\in \s$ has {non-vanishing} cohomology in $\ell$ positions, and let $m$ be the smallest integer such that $\sigma^{>m+1}(X\oplus Y)=0$. Let ${X_1= \sigma^{\leq m }X}$ and ${X_2= \sigma^{> m}X}$. In the truncation triangle}
		\begin{equation*}
		\begin{tikzcd}
		X_1\arrow[r]& X\arrow[r]& X_2\arrow[r]&\Sigma X_1,
		\end{tikzcd}
		\end{equation*}
		the object $X_1$ has {non-vanishing} cohomology in $\ell-1$ positions, and $X_2$ has {non-vanishing} cohomology in one position. In the same vein, {let $Y_1= \sigma^{\leq m }Y$ and $Y_2= \sigma^{> m}Y$. In the truncation triangle}
		\begin{equation*}
		\begin{tikzcd}
		Y_1\arrow[r]& Y\arrow[r]& Y_2\arrow[r]&\Sigma Y_1,
		\end{tikzcd}
		\end{equation*}
		{the objects $Y_1$ and $Y_2$ have the same properties as $X_1$ and $X_2$ above.} Then $X_1\oplus Y_1=\sigma^{\leq m}(X\oplus Y)$ and $X_2\oplus Y_2=\sigma^{> m}(X\oplus Y)$ are objects in $\s$ with cohomology in fewer than $\ell$ positions, whence $X_1,X_2\in \s$ by the induction hypothesis. {Since} $\s$ is a triangulated subcategory of $\d$, it follows that $X\in \s$, since it is an extension of $X_2$ {in} $X_1$.
	\end{proof}
	
	Following Zhang--Cai \cite[Definition~2.3]{ZC17}, we say that a thick subcategory $\s$ of $\d$ is \textit{${\sigma}^0$-stable} if ${\sigma}^0X\in \s$ for all $X\in\s$. {Having defined the cohomology functor ${\sigma}^i$ by ${\sigma}^0 \circ \Sigma^{{i}}$, it is equivalent to impose that $\s$ is closed under all cohomology functors ${\sigma}^i$.}
	
	\begin{lemma}\label{lem:boundedtex}
		Let $\d$ be a triangulated category {equipped} with a t-structure $(\d^{\leq 0 },\d^{> 0})$.
		A thick subcategory $\s$ {of} $\d$ is ${\sigma}^0$-stable if it is t-exact. If the t-structure is bounded, the converse also holds.
	\end{lemma} 
	\begin{proof}
		We explained in \Cref{rem:tex} that a t-exact subcategory $\s$ is closed {under application of} $\sigma^{\leq 0}$ and $\sigma^{> -1}$. Hence, it is closed under ${\sigma}^0=\sigma^{\leq 0}\sigma^{> -1}$, which is to be ${\sigma}^0$-stable.
		
		Conversely, suppose that the t-structure $(\d^{\leq 0 },\d^{> 0})$ is bounded and let $\s$ be an ${\sigma}^0$-stable subcategory of $\d$. For any $X\in\s$ and any integer $n$, we can form a truncation triangle
		\begin{center}
			\begin{tikzcd} \sigma^{\leq n-1} X \arrow[r,""]& \sigma^{\leq n}X \arrow[r,""]& \Sigma^{-n}{\sigma}^nX \arrow[r,""] &  \Sigma\sigma^{\leq n-1} X. \end{tikzcd}
		\end{center}
		The object {$\Sigma^{-n}{\sigma}^nX$} lies in $\s$, and hence $\sigma^{\leq n-1}X$ belongs to $\s$ if and only if $\sigma^{\leq n}X$ does. By induction, it follows that $\sigma^{\leq 0}X\in\s$ if and only if $\sigma^{\leq n}X\in \s$, where $n$ is an arbitrary integer.
		The boundedness of the t-structure ensures that $\sigma^{\leq n}X=X$ {for some $n$; indeed, the object $X$ lies in {some} subcategory of $\d$ of the form $\d^{[m,n]}$, where the truncation functor $\sigma^{\leq n}$ acts as the identity.} We conclude that $\s$ is t-exact, {having} shown that $\sigma^{\leq 0}X\in\s$ for any $X\in\s$. 
	\end{proof}
	
	{Fix a t-structure $(\d^{\leq 0},\d^{>0})$ on $\d$.} The set of t-exact subcategories of $\d$ with respect to $(\d^{\leq 0},\d^{>0})$ is closed under arbitrary intersections. {Ordering this set with respect to inclusion}, it has been endowed with the structure of a complete lattice. The join of a family $\{\s_i\}_{i\in I}$ is the smallest t-exact subcategory containing all $\s_i$, whereas the meet is computed by intersection. We denote this lattice by $\texact_{(\d^{\leq 0},\d^{>0})}(\d)$, and we let $\wide(\d^0)$ {denote} the lattice of wide subcategories in $\d^0$, {ordered by inclusion}. We conclude this subsection with a result providing mutually inverse isomorphism{s} between these lattices, {and a lemma {that will be {useful later}.}}
	
	\begin{theorem}[{\cite[Theorem~2.5]{ZC17}}]\label{ZC17.2.5}
		Let $\d$ be a triangulated category {equipped} with a bounded t-structure $(\d^{\leq 0},\d^{> 0})$. The cohomology functor ${\sigma}^0\colon \d\to\d^0$ induces an isomorphism of lattices
		\begin{equation*}
		{\sigma}^0\colon \texact_{(\d^{\leq 0},\d^{>0})}(\d)\to\wide(\d^0),
		\end{equation*}
		with inverse $\d_{(-)}\colon \wide(\d^0)\to\texact_{(\d^{\leq 0},\d^{>0})}(\d)$ sending a wide subcategory $\w\subseteq\d^0$ to the t-exact subcategory
		\begin{equation*}
		\d_{\w} \defeq \{X\in \d \mid {\sigma}^iX\in \w \text{ for all } i\in\Z\} 
		\end{equation*}
		of $\d$.
	\end{theorem}
	
	Note that $\d_{\w}$ clearly is ${\sigma}^0$-stable {by definition}. By \Cref{lem:boundedtex}, it is indeed a t-exact subcategory {of $\d$}.
		
	\begin{lemma}\label{lem:plus_rev}
	{Let $\d$ be a triangulated category {equipped} with a bounded t-structure $(\d^{\leq 0},\d^{> 0})$. }
	\begin{enumerate}
		\item\label{lem:plus_rev1} {The map ${\sigma}^0  \colon \texact_{(\d^{\leq 0},\d^{>0})}(\d)\to\wide(\d^0)$, as defined in \Cref{ZC17.2.5}, can be expressed by $\s \mapsto \s \cap \d^0$, intersecting a t-exact subcategory $\s$ with the heart $\d^0$.}
		\item\label{lem:plus_rev2} {For $\w\in \wide(\d^0)$, let $\d_{\w}$ be as in \Cref{ZC17.2.5}. We have that $\d_{\w}=\thick_{\d}(\w)$.}
	\end{enumerate}
	\end{lemma}
	\begin{proof}
	{We first prove \eqref{lem:plus_rev1}. Let $\s$ be a t-exact subcategory of $\d$. For any $X\in \s\cap \d^0$, we have that ${\sigma}^0X$ is isomorphic to $X$, putting $X$ in ${\sigma}^0\s$, namely the essential image of $\s$ under the cohomology functor. This shows that $\s\cap \d^0\subseteq {\sigma}^0\s$. {To show that the} reverse inclusion $\s\cap \d^0\supseteq {\sigma}^0\s$ also holds, {one uses that} ${\sigma}^0\s$ is contained in $\d^0$ by the definition of the cohomology functor, and in $\s$ by the ${\sigma}^0$-stability of $\s$ (see \Cref{lem:boundedtex}).}

	{Secondly and finally, we prove \eqref{lem:plus_rev2}. We will regard the wide subcategory $\w\subseteq \d^0$ as a subcategory of $\d$. Since $\d_{\w}$ is a thick subcategory of $\d$ containing $\w$, it immediately follows that $\d_{\w}\supseteq \thick_{\d}(\w)$. We complete the proof by showing that any object $X\in \d_{\w}$ can be built from objects in $\w$ using extensions, (de)suspensions and direct summands. Since the t-structure $(\d^{\leq 0},\d^{> 0})$ is bounded, there exists some interval $[m,n]$ of integers such that $X\in\d^{[m,n]}$. We proceed by induction of the length of this interval, namely the integer $\ell=n-m$. The case $\ell = 0$ concerns intervals of the form $[m,m]$. We can build any object in $\d^{[m,m]}\cap \d_{\w}$ by (de)suspending objects in $\w$, thus settling the anchor step. Suppose that the statement has been shown for all intervals of length {shorter} than $\ell$, where $\ell\geq 1$. Supposing that $X\in \d^{[m,n]}\cap \d_{\w}$, and that the interval $[m,n]$ is of length $\ell$, one considers the truncation triangle}
	\begin{center}
			\begin{tikzcd} \sigma^{\leq n-1} X \arrow[r,""]& X \arrow[r,""]& \sigma^{>n-1}X \arrow[r,""] &  \Sigma\sigma^{\leq n-1} X. \end{tikzcd}
	\end{center}
	{{By \Cref{lem:boundedtex}}, the objects $\sigma^{\leq n-1} X$ and $\sigma^{>n-1}X$ are in $\d_{\w}$ by t-exactness.
 By the induction hypothesis, both objects are in $\thick_{\d}(\w)$. Since $X$ is an extension of the two, we conclude that $X\in \thick_{\d}(\w)$. By the Principle of Induction, we {have} that $\d_{\w}\subseteq \thick_{\d}(\w)$, and in turn that these {triangulated} subcategories of $\d$ are equal.}
	\end{proof}

	\subsection{Silting and $\tau$-tilting}
	
	\begin{definition}\label{def:perfsilt}
		Let $\c$ be a triangulated category {with suspension functor $\Sigma$}. An object ${ X}\in\c$ is \textit{presilting} if $\c({ X},\Sigma^n { X})=0$ for all positive integers $n$. A presilting object ${ X}$ is a \textit{silting object} if, in addition, $\thick({ X})=\c$. 
	\end{definition}
 
	{Recall that if $\c$ is a Krull--Schmidt $k$-category, an object $X\in\c$ is \textit{basic} if any two distinct indecomposable direct summands of $X$ are non-isomorphic.}
	The sets of {basic} presilting and {basic} silting objects in $\c$ are denoted by $\presilt(\c)$ and $\silt(\c)$, respectively.
	The subsets containing the objects having a fixed {(basic)} presilting object ${ P}$ as a direct summand are denoted by $\presilt_{ P}(\c)$ and $\silt_{ P}(\c)$, respectively.
	
	\begin{definition}\label{def:perf2silt}
		For a fixed silting object ${ S}\in\c$, we say that an object ${ P}\in\c$ is \textit{$2_{ S}$-term} {if} it belongs to $\add({ S})*\Sigma\add({ S})$, i.e{.} if there exists a triangle
		\begin{equation*}
		{ S}_1\to { S}_0\to { P}\to \Sigma { S}_1,
		\end{equation*}
		where ${ S}_1,{ S}_0\in\add({ S})$. 
		The set of {basic} $2_{ S}$-term presilting objects is denoted by $\npresilt{2}{ S}(\c)$, and the subset of {basic} $2_{S}$-term silting objects by $\nsilt{2}{ S}(\c)$. The subsets containing those having {a (basic) presilting object} ${ P}$ is a direct summand are denoted by $\npresilt{2}{ S}_{ P}(\c)$ and $\nsilt{2}{ S}_{ P}(\c)$, respectively. {When the choice of silting object $S$ is clear from context, we refer to $2_{ S}$-term (pre)silting objects as \textit{two-term (pre)silting objects}}.
	\end{definition}
	
	{We note that $S$ in the above definition does not have to be basic. However, if $R$ denotes the largest basic direct summand of $S$, the set of basic $2_{S}$-term silting objects in $\c$ coincides with that of basic $2_{R}$-term silting objects in $\c$. {Thus,} we will not need to assume that $S$ is basic in our results.}
	
	If $\c=\Kb(\proj {\Lambda})$, the bounded homotopy category of a {finite dimensional} $k$-algebra ${\Lambda}$, then ${\Lambda}$ is a silting object in $\c$. {({Note that $\Lambda$ need not be a basic $k$-algebra;} removing redundant direct summands {of $\Lambda$} will yield a basic silting object {in $\Kb(\proj {\Lambda})$}.)} More generally, if $A$ is a non-positive dg $k$-algebra, then $A$ is silting in $\per(A)$. 
	We write $\nsilt{2}{}(A)$ for the set of $2_A$-term {basic} silting objects in $\per(A)$ (and so on), omitting the subscript beneath the symbol $2$.
	
	The set $\nsilt{2}{S}(\c)$ admits a partial order $\geq$, where ${ P}\geq { Q}$ provided that $\c({ P},\Sigma^n { Q}) = 0$ whenever $n>0$ \cite[Theorem~2.11 and Proposition~2.14]{AI12}.
	If $\c$ is Hom-finite and Krull--Schmidt, it turns out that any $2_{ S}$-term presilting object ${ P}\in\c$ can be completed to a $2_{ S}$-term silting object. One such silting object is the \textit{Bongartz completion} ${ T^{+}_P}\defeq{ P}\oplus { Q^+}$ of ${ P}$, where ${Q^+}$ is defined by the triangle
	\begin{equation}\label{eq:Bon-tri}
	{ S}\to { Q^+}\to { P}_0\to^{\beta_{\Sigma { S}}}\Sigma { S},
	\end{equation}
	in which $\beta_{\Sigma { S}}$ is a right $\add({ P})$-approximation of $\Sigma { S}$ \cite[{Lemma 4.2}]{IJY14}. Consequently, a $2_S$-term presilting object is precisely the same thing as a direct summand of a $2_S$-term silting object. If $|P|$ denotes the number of non-isomorphic indecomposable direct summands of a $2_S$-term presilting object $P$, we have that {$P$} is silting if and only if $|P|=|S|$ \cite[{Proposition~4.3}]{IJY14}. {If $P$ is basic, we {remove superfluous direct summands} from its Bongartz completion {$T^+_P$}, so that it becomes an element in $\nsilt{2}{S}(\c)$.} The Bongartz completion of $P$ is then the maximal element in $\nsilt{2}{S}(\c)$ (with respect to the partial order $\geq$ defined above) in which ${ P}$ is a direct summand. 
	
	There is a close connection between the $2_{ S}$-term silting theory of $\c$ and the support $\tau$-tilting theory of the finite dimensional $k$-algebra $\End_{\c}({ S})$. We now recall the definitions and results of $\tau$-tilting theory \cite{AIR14}.
	\begin{definition}\label{def:tautilt}
		For a finite dimensional $k$-algebra ${\Lambda}$, a finitely generated right ${\Lambda}$-module $M$ is \textit{$\tau$-rigid} if $\Hom_{{\Lambda}}(M,\tau M)=0$, where $\tau$ is the Auslander--Reiten translation. A pair $(M,Q)$ is \textit{$\tau$-rigid} if $M$ is a $\tau$-rigid module and $Q$ is a finitely generated projective right ${\Lambda}$-module such that $\Hom_{{\Lambda}}(Q,M)=0$. A $\tau$-rigid pair $(M,Q)$ is \textit{support $\tau$-tilting} {if} $|M|+|Q|=|{\Lambda}|$, where $|X|$ is the number of non-isomorphic indecomposable direct summands of $X$. A ${\Lambda}$-module $M$ is \textit{support $\tau$-tilting} if there exists a projective ${\Lambda}$-module {$Q$} such that $(M,Q)$ is a {support} $\tau$-tilting pair.
	\end{definition}
	
	{We identify the $\tau$-rigid pairs $(M,Q)$ and $(N,R)$ if $\add(M)=\add(N)$ and $\add(Q)=\add(R)$. We say that $(M,Q)$ is \textit{basic} {if $M$ and $Q$ are basic}. By the Krull--Schmidt property of $\fgMod{\Lambda}$, any $\tau$-rigid pair can be identified with a basic $\tau$-rigid pair.}
	A $\tau$-rigid pair $(N,R)$ is a \textit{direct summand} of the $\tau$-rigid pair $(M,Q)$ if $N$ is a direct summand of $M$ and $R$ is a direct summand of $Q$. {We say that a $\tau$-rigid pair $(M,Q)$ is \textit{indecomposable} if $M\oplus Q$ is indecomposable as a $\Lambda$-module.} {Note that indecomposable $\tau$-rigid pairs necessarily have {a} zero module {in} exactly one of the two components.} If $(M,Q)$ and $(M,R)$ are support $\tau$-tilting pairs, then $\add(Q)=\add(R)$ \cite[Proposition~2.3(b)]{AIR14}, {so $Q=R$ if these pairs are basic}. It follows that we have a one-to-one correspondence between basic support $\tau$-tilting modules and {basic} support $\tau$-tilting pairs, which truncates a pair $(M,Q)$ to $M$.
	
	We denote the set of {basic} $\tau$-rigid pairs in $\fgMod {\Lambda}$ by $\staurigidpair({\Lambda})$, and those having a $\tau$-rigid pair $(M,Q)$ as a direct summand by $\staurigidpair_{(M,Q)}({\Lambda})$. {Both of these sets contain subsets consisting of basic support $\tau$-tilting pairs,} that are denoted by $\stautiltpair({\Lambda})$ and $\stautiltpair_{(M,Q)}({\Lambda})$, respectively. We say that {the} $k$-algebra ${\Lambda}$ is \textit{$\tau$-tilting finite} if ${\stautiltpair}(\Lambda)$ is a finite set.
		
	The set of support $\tau$-tilting ${\Lambda}$-modules is inextricably linked with functorially finite torsion classes of $\fgMod {\Lambda}$. Recall that a \textit{torsion class} (resp.
	\textit{torsion-free class}) of $\fgMod {\Lambda}$ is a subcategory which is closed under factor modules (resp. submodules) and extensions {\cite{Dic66}}. If $\mathcal{G}$ is a torsion class, the right perpendicular category $\mathcal{G}^{\perp}$ is a torsion-free class. Dually, the left perpendicular category of a torsion-free class is a torsion class. A \textit{torsion pair} is a pair of subcategories $(\mathcal{G},\mathcal{F})$ where $\mathcal{G}^{\perp}=\mathcal{F}$ and ${^{\perp}\mathcal{F}}=\mathcal{G}$. It is indeed the case that $\mathcal{G}$ is a torsion class and that $\mathcal{F}$ is {a} torsion-free {class}.  {For every $X\in  \fgMod {\Lambda}${,} one has a canonical exact sequence
	\begin{equation}\label{eq:torsionpair_seq}
		0 \to tX\mono X\epi f X \to 0,
	\end{equation}
	in which the objects $tX$ and $fX$ are determined by functors $t\colon \fgMod {\Lambda} \to \mathcal{G}$ and $f\colon \fgMod {\Lambda} \to \mathcal{F}$, called the \textit{torsion radical} and the \textit{torsion-free functor}, respectively.}
	
	In the next theorem and throughout, {given a ${\Lambda}$-module $M$, let} $\gen(M)$ denote the full subcategory of $\fgMod {\Lambda}$ containing the ${\Lambda}$-modules $X$ {for which} there exists an epimorphism $M^{\oplus n}\epi X$ for some $n\geq 1$.
	
	\begin{theorem}[{\cite[Theorem~2.7]{AIR14}}]\label{thm:AIR14.2.7}
		Let $M$ be a support $\tau$-tilting ${\Lambda}$-module. Then {the pair $(\gen(M),M^{\perp})$ is a torsion pair in $\fgMod {\Lambda}$.} Moreover, the torsion class $\gen(M)$ is functorially finite in $\fgMod {\Lambda}$, and we have a bijection
		\begin{equation*}
		\gen\colon {\stautiltpair}({\Lambda})\to \ftors({\Lambda})
		\end{equation*}
		from the set of {basic} support $\tau$-tilting ${\Lambda}$-modules
 to the set of functorially finite torsion classes in $\fgMod {\Lambda}$. 
	\end{theorem}
	
	Since the set of torsion classes is partially ordered under inclusion, this bijection induces a partial order on ${\stautiltpair}({\Lambda})$. More explicitly, if $(M,Q)$ and $(N,R)$ are {basic} support $\tau$-tilting pairs, we {impose} that $(M,Q)\geq (N,R)$ if $M\geq N$ as {basic} support $\tau$-tilting modules{, which is to say that $\gen(M) \supseteq \gen(N)$ as torsion classes.}
	
	\begin{theorem}[{\cite[Proposition~4.5]{Jas15}, \cite[Proposition~6.2(3)]{IY08}}]\label{IY08.6.2}
		Let $\c$ be {a} Hom-finite Krull--Schmidt triangulated category and ${ S}\in\c$ be a silting object. Let $\add({ S})*\Sigma\add({ S})$ be the subcategory of $2_{ S}$-term objects in $\c$. The functor 
		\begin{equation*}
		\c({ S},-)\colon \c \to \fgMod\End_{\c}({ S})
		\end{equation*}
		induces an equivalence of categories
		\begin{equation*}
		\c({ S},-)\colon {\add({ S})*\Sigma\add({ S})\over [\Sigma { S}]} \to \fgMod\End_{\c}({ S}),
		\end{equation*}
		where $[\Sigma { S}]$ is the {ideal of $\c$ consisting} of morphisms factoring {through} ${\add(\Sigma S)}$.
	\end{theorem}
	
	The close connection between silting and support $\tau$-tilting is {made precise} by the following theorem.
	\begin{theorem}[{\cite[Theorem~4.5]{IJY14}}]\label{IJY14.4.6}
		Let $\c$ and ${S}$ be as in \Cref{IY08.6.2} and let ${\Lambda}=\End_{\c}({ S})$. We have a bijection
		\begin{equation*}
		\begin{tikzcd}
		\npresilt{2}{ S}(\c)\arrow[r,"H_{ S}"] & \staurigidpair({\Lambda}), \\
		{ X}\arrow[u,symbol=\in]\arrow[r,mapsto] & (\c({ S},{ X}),\c({ S},{ X}_1))\arrow[u,symbol=\in]
		\end{tikzcd}
		\end{equation*}
		where $\Sigma { X}_1$ is the maximal direct summand of ${ X}$ in $\Sigma{ S}$. It restricts to bijections
		\begin{center}
			\begin{tikzcd}
				\npresilt{2}{ S}_{ P}(\c)\arrow[r,"H_{ S}"] & \staurigidpair_{H_{ S}({ P})}({\Lambda}),
			\end{tikzcd}
		\end{center}
		and 
		\begin{equation}\label{eq:lastbij}
		\begin{tikzcd}
		\nsilt{2}{ S}_{ P}(\c)\arrow[r,"H_{ S}"] & \stautiltpair_{H_{ S}({ P})}({\Lambda}),
		\end{tikzcd}
		\end{equation}
		for each ${ P}\in\npresilt{2}{ S}(\c)$. The last bijection is an isomorphism of partially ordered sets.
	\end{theorem}
		
	The isomorphism in \eqref{eq:lastbij} sends the Bongartz completion of ${{ P}{\in \npresilt{2}{S}(\c)}}$ to a {unique} maximal {element in} $\stautiltpair_{H_{ S}({ P})}({\Lambda})$. We refer to this as the \textit{Bongartz completion} of the $\tau$-rigid pair $H_{ S}({ P})$. A module-theoretic construction of the Bongartz completion of $\tau$-rigid pairs is available \cite[Theorem~2.10]{AIR14}, \cite[Theorem~4.4]{DIRRT17}. We often denote the Bongartz completion of the $\tau$-rigid pair $(M,Q)$ by $(M^+,Q)$, noting that the second component $Q$ {is} unaltered.
	
	\begin{proposition-definition}[{\cite[Proposition~3.6]{Jas15}, \cite[Theorem~4.12(a)]{DIRRT17}, {\cite[Definition 3.2]{BH21}}}]\label{Jas15.3.6}
		Let ${\Lambda}$ be a finite dimensional $k$-algebra, and let $(M,Q)$ be a $\tau$-rigid pair in $\fgMod {\Lambda}$. {The pair $(M,Q)$} determines a wide subcategory 
		\begin{equation*}
		{J_{\Lambda}(M,Q) \defeq} M^{\perp}\cap {^{\perp}(\tau M)}\cap Q^{\perp}\subseteq \fgMod {\Lambda}
		\end{equation*}
		called the \textit{$\tau$-perpendicular category} of $(M,Q)$. {We denote $J_{\Lambda}(M,Q)$ by $J(M,Q)$ when no confusion may arise.}
		{A wide subcategory $W$ of $\fgMod {\Lambda}$ is a \textit{$\tau$-perpendicular subcategory} if it is the $\tau$-perpendicular category of some $\tau$-rigid pair in $\fgMod {\Lambda}$.}
	\end{proposition-definition}
	
	The next result shows that $\tau$-perpendicular categories are module categories.
		
	\begin{proposition-definition}[{\cite[Theorem~3.8]{Jas15}},{\cite[Theorem~4.12(b)]{DIRRT17}}]\label{DIRRT17.4.12}
		Let $(M,Q)$ be a $\tau$-rigid pair in $\fgMod {\Lambda}$, let $(M^+,Q)$ be the Bongartz completion of $(M,Q)$, and let ${\Gamma}=\End_{{\Lambda}}(M^+)/[M]${, where $[M]$ is the ideal of morphisms factoring through $\add(M)$.}
		We then have an exact equivalence
		\begin{equation*}
		F_{(M,Q)}\defeq\Hom_{{\Lambda}}(M^+,-)\colon J(M,Q)\to\fgMod {\Gamma}.
		\end{equation*}
		We will refer to the algebra ${\Gamma}$ as the \textit{$\tau$-tilting reduction} of ${\Lambda}$ with respect to $(M,Q)$.
	\end{proposition-definition}
	
Let $J(M,Q)$ be the $\tau$-perpendicular category of a $\tau$-rigid pair $(M,Q)$ with Bongartz completion $(M^+,Q)$, and let ${\Gamma}=\End_{{\Lambda}}(M^+)/[M]$. A \textit{$\tau$-rigid pair {(resp. support $\tau$-tilting pair)} in $J(M,Q)$} is a pair $(U,R)$ in $J(M,Q)$ such that {$(F_{(M,Q)}U,F_{(M,Q)}R)$ is a $\tau$-rigid pair {(resp. support $\tau$-tilting pair)} in $\fgMod \Gamma$, where $F_{(M,Q)}$ is the exact equivalence in \Cref{DIRRT17.4.12}. The set of {basic $\tau$-rigid pairs  {(resp. basic support $\tau$-tilting pairs)} in $J(M,Q)$} will be denoted by $\staurigidpair(J(M,Q))$ {(resp. $\stautiltpair(J(M,Q))$)}.} {Given a basic $\tau$-rigid pair $(U,R)$ in $J(M,Q)$, let $\staurigidpair_{(U,R)}(J(M,Q))$ {(resp. $\stautiltpair_{(U,R)}(J(M,Q))$)} denote the set of basic $\tau$-rigid pairs {(resp. support $\tau$-tilting pairs)} in $J(M,Q)$ containing $(U,R)$ as a direct summand.}
	
{{From \Cref{DIRRT17.4.12}}, it follows that any $\tau$-rigid pair $(M,Q)$ in $\fgMod {\Lambda}$ provides a bijection
\begin{equation}\label{eq:FMQbij}
	F_{(M,Q)}\colon \staurigidpair(J(M,Q)) \to  \staurigidpair({\Gamma}),
\end{equation}
where ${\Gamma}$ is {the {$\tau$-tilting reduction} of ${\Lambda}$ with respect to $(M,Q)$.}}
	
	\section{Silting t-structures and reduction}\label{sec:silttstr}
	
	Silting reduction is a well-developed reduction technique, particularly for Hom-finite Krull--Schmidt triangulated categories \cite{IY18}. The bijections in question restrict to two-term objects and are compatible with $\tau$-tilting reduction \cite[Theorem~4.12]{Jas15}. In this section, we study the interplay between silting reduction and t-structures. Iyama--Yang prove the following theorem, though in greater generality.
	
	\begin{theorem}[{\cite[Theorems~3.1, 3.6, and 3.7]{IY18}, \cite[Proposition 4.11]{Jas15}}]\label{IY18.3.7}
		Let $\c$ be a Hom-finite Krull--Schmidt triangulated category containing a silting object ${ S}$. Let ${ P}$ be a {basic} presilting object.
		\begin{enumerate} 
			\item\label{IY.3.6} Let $\z_{ P}={^{\perp_{>0}}P}\cap {P^{\perp_{>0}}}$. Then the composite
			\begin{equation*}
			\begin{tikzcd}
			\z_{ P}\arrow[r,hook] & \c \arrow[r,"{\rm loc}"] & \c/\thick({ P})
			\end{tikzcd}
			\end{equation*}
			induces a triangle equivalence $\z_{ P}/[{ P}]\to \c/\thick({ P})$, where the triangulation on $\z_{ P}/[{ P}]$ is described by Iyama--Yoshino \cite[Theorem~4.2]{IY08}. 
			{The suspension of $X\in \z_{ P}/[{ P}]$ is denoted by $X\langle 1\rangle$, and is defined by forming a triangle
			\begin{equation}\label{eq:rangle_tri}
			X \to^{\alpha_X} P^X \to^{} X\langle 1 \rangle \to \Sigma X,
			\end{equation}
			in which $\alpha_X$ is a left $\add(P)$-approximation.
			}
			{{In the same vein,} the inverse of the suspension of $X\in \z_{ P}/[{ P}]$ is denoted by $X\langle -1\rangle$, and is defined by forming a triangle
			\begin{equation}\label{eq:langle_tri}
			X\langle -1\rangle \to P_X\to^{\beta_X} X\to \Sigma{(}X\langle -1\rangle{),}
			\end{equation}
			in which $\beta_X$ is a right $\add(P)$-approximation.
			}
			\item\label{IY18.3.7.i} The additive quotient $\z_P \to \z_P/[P]$ (and also the Verdier localisation $\c\to \c/\thick(P)$) induces a bijection 
			\begin{equation}\label{presiltbij}
			\varphi_{ P}\colon \presilt_{ P}(\c) \to \presilt(\z_{ P}/[{ P}])\cong\presilt(\c/\thick({ P})).
			\end{equation}
			If ${P}$ is $2_{S}$-term, and has the Bongartz completion $T^{+}_P$, it restricts to a bijection
			\begin{equation}\label{2presiltbij}
			\varphi_{ P}\colon \npresilt{2}{{ S}}_{ P}(\c) \to \npresilt{2}{{ T^{+}_P}}(\z_{ P}/[{ P}])\cong\npresilt{2}{{ T^+_P}}(\c/\thick({ P})).
			\end{equation}
			\item\label{IY18.3.7.ii} The bijection \eqref{presiltbij} restricts to an isomorphism of posets 
			\begin{equation*}
			\varphi_{ P}\colon \silt_{ P}(\c) \to \silt(\z_{ P}/[{ P}])\cong\silt(\c/\thick({ P})),
			\end{equation*}
			which, if ${P}$ is $2_{S}$-term, further restricts to an isomorphism of posets
			\begin{equation}\label{presiltbij2}
			\varphi_{P}\colon \nsilt{2}{ S}_{P}(\c) \to \nsilt{2}{{ T^{+}_P}}(\z_{ P}/[{ P}]) \cong \nsilt{2}{{ T^{+}_P}}(\c/\thick({ P})).
			\end{equation}
		\end{enumerate}
	\end{theorem}
	
	Since the bijection in \eqref{presiltbij} is induced by the ideal quotient $\z_{ P}\to \z_{ P}/[{ P}]$, our next two lemmas are immediate consequences.
	\begin{lemma}\label{lem:indsiltred}
		Let $\c$, {and $S$} {be} as in \Cref{IY18.3.7}, {and let $P$ be a basic $2_S$-term presilting object in $\c$}. We define $\ind_{ P}\npresilt{2}{ S}_{ P}(\c)$ to be the objects in $\npresilt{2}{ S}_{ P}(\c)$ of the form ${ P}\oplus { X}$, where ${ X}$ is an indecomposable object which is not contained in $\add(P)$.
		Then the bijection in \eqref{2presiltbij} restricts to a bijection
		\begin{equation*}
		\varphi_{ P}\colon \ind_{ P}\npresilt{2}{{ S}}_{ P}(\c) \to \ind\npresilt{2}{{ T^{+}_P}}(\z_{ P}/[{ P}]),
		\end{equation*}	
		where the codomain is the set of indecomposable objects in $\npresilt{2}{{ T^{+}_P}}(\z_{ P}/[{ P}])$.
	\end{lemma}
	
	\begin{lemma}\label{lem:compphi}
	{Let $\c$ and $S$ {be} as in \Cref{IY18.3.7}.}
		Let $P$ and $Q$ be {basic} $2_S$-term presilting objects in $\c$ such that $P\oplus Q$ also is {basic} $2_S$-term presilting. 
		The bijection
		\begin{equation*}
		\varphi_{ P}\colon \npresilt{2}{ S}_{ P}(\c) \to \npresilt{2}{{ T^{+}_P}}(\z_{ P}/[{ P}])
		\end{equation*}
		restricts to a bijection
		\begin{equation*}
		\varphi_{ P}\colon\npresilt{2}{ S}_{{ P}\oplus {Q}}(\c) \to \npresilt{2}{{ T^{+}_P}}_{\varphi_{P}({P\oplus}Q)}(\z_{ P}/[{ P}]).
		\end{equation*}
	\end{lemma}
	
	\begin{lemma}\label{lem:functorphi}
	{Let $\c$ and $S$ {be} as in \Cref{IY18.3.7}, and} {let $P$ and $Q$ be {basic} $2_S$-term presilting objects in $\c$ such that $P\oplus Q$ also is {basic} $2_S$-term presilting.} Then the diagram
		\begin{equation}\label{eq:itphi}
		\begin{tikzcd}[row sep=3em,column sep=6em]
		\npresilt{2}{ S}_{{ P}\oplus { Q}}(\c)\arrow[d,"\varphi_{ Q}" description]\arrow[r,"\varphi_{ P}" description]\arrow[dr,"\varphi_{{ P}\oplus { Q}}" description]  & \npresilt{2}{ T^{+}_P}_{\varphi_{P}({P\oplus}Q)}(\z_{ P}/[{ P}])\arrow[d,"\varphi_{\varphi_{P}({{P \oplus}Q})}" description] \\
		\npresilt{2}{ T^+_Q}_{\varphi_{Q}(P{\oplus Q})}(\z_{Q}/[{Q}]) \arrow[r,"\varphi_{\varphi_Q({P{\oplus Q}})}" description] & \npresilt{2}{{T}^{+}_{{ P}\oplus { Q}}}(\z_{{ P}\oplus { Q}}/[{ P}\oplus { Q}])
		\end{tikzcd}
		\end{equation}
		commutes. 
	\end{lemma}
	\begin{proof}
		{For the sake of terseness, let $[Q]$ denote the ideal of $\mathcal{Z}_{{ P}\oplus { Q}}/[{P}]$ consisting of morphisms factoring through $\add(Q)$, and similarly for the ideal $[P]$ of $\mathcal{Z}_{{ P}\oplus { Q}}/[{Q}]$.} By the Third Isomorphism Theorem, we have additive equivalences 
		\begin{equation}\label{eq:functorphi1}
		{\mathcal{Z}_{{ P}\oplus { Q}}/[{ P}]\over [{Q}]}\simeq{\mathcal{Z}_{{ P}\oplus { Q}}/[{ P}]\over[{ P}\oplus { Q}]/[{ P}]} \simeq {\mathcal{Z}_{{ P}\oplus { Q}}\over[{ P}\oplus { Q}]},
		\end{equation}
		{where $[P\oplus Q]$ denotes an ideal of $\z_{P\oplus Q}$}. Similarly, {we have}
		\begin{equation}\label{eq:functorphi2}
		{\mathcal{Z}_{{ P}\oplus { Q}}/[{ Q}]\over[{ P}]}\simeq {\mathcal{Z}_{{ P}\oplus { Q}}\over[{ P}\oplus { Q}]}.
		\end{equation}
		The {equivalences} in \eqref{eq:functorphi1} and \eqref{eq:functorphi2} enable the construction of an essentially commutative diagram of ideal quotient functors
		\begin{equation}\label{eq:idealsq}
		\begin{tikzcd}[row sep=3em,column sep=3em]
		\z_{{ P}\oplus { Q}}\arrow[d]\arrow[r]\arrow[dr]  & \z_{{ P}\oplus { Q}}/[{ P}] \arrow[d] \\
		\z_{{ P}\oplus { Q}}/[{ Q}] \arrow[r] & \z_{{ P}\oplus { Q}}/[{ P}\oplus { Q}]
		\end{tikzcd}
		\end{equation}
		{\Cref{IY18.3.7}\eqref{IY18.3.7.i} asserts that the functors going out of $\z_{{ P}\oplus { Q}}$ induce the bijections going out of $\npresilt{2}{ S}_{{ P}\oplus { Q}}(\c)$ in \eqref{eq:itphi}. 
		
		We complete the proof by showing that the maps $\varphi_{\varphi_Q(P\oplus Q)}$ and $\varphi_{\varphi_P(P\oplus Q)}$ {in \eqref{eq:itphi}} have the claimed codomain, and that they factorise $\varphi_{P\oplus Q}$ as shown in \eqref{eq:itphi}.} {To give an argument as to why $\varphi_{\varphi_Q(P\oplus Q)}$ (and similarly $\varphi_{\varphi_P(P\oplus Q)}$) has the claimed codomain in \eqref{eq:itphi}, it suffices to show that the functor on the bottom of \eqref{eq:idealsq} has the following property:
		\begin{equation}\label{eq:functorphispec}
		\text{The ideal quotient functor $ \z_{P\oplus Q}/[Q]\to \z_{P\oplus Q}/[P\oplus Q]$ sends $T_{Q}^+$ to $T_{P\oplus Q}^+$.}
		\end{equation}
		Indeed, it follows from \Cref{IY18.3.7}\eqref{IY18.3.7.i} and \eqref{eq:functorphi2} that $\varphi_{\varphi_Q(P\oplus Q)}$ sends presilting objects to presilting objects, and the claim in \eqref{eq:functorphispec} will ensure that $2_{T_{Q}^+}$-term objects are sent to $2_{T_{P\oplus Q}^+}$-term objects. To establish the claim in \eqref{eq:functorphispec}, it suffices to prove that the ideal quotient functor $ \z_{P\oplus Q}\to \z_{P\oplus Q}/[P\oplus Q]$ sends $T_{Q}^+$ and $T_{P\oplus Q}^+$ to isomorphic objects. Using the definition of Bongartz completion (see \eqref{eq:Bon-tri}), it is easy to see that $T_{Q}^+$ to $T_{P\oplus Q}^+$ are both isomorphic to $S$ in the Verdier quotient $\c/\thick(P\oplus Q)$. It now follows from \Cref{IY18.3.7}\eqref{IY.3.6} that $T_{Q}^+$ and $T_{P\oplus Q}^+$ are isomorphic in $\z_{P\oplus Q}/[P\oplus Q]$, which is enough to conclude that \eqref{eq:functorphispec} holds.}
	
{By \Cref{lem:compphi} and the previous paragraph, the maps in the \eqref{eq:itphi} are induced by the functors in \eqref{eq:idealsq}. The commutativity of the former diagram now follows from the {essential} commutativity of the latter.}	
	\end{proof}

	We will now state our setup, for the sake of future reference. 	
	\begin{setup}\label{setup:siltingtstr}
		Let $A$ be a non-positive dg $k$-algebra {such that $H^iA$ is finite dimensional for all $i\in \Z$}, and let $P$ be a {basic} two-term presilting object in $\per(A)$, i.e. an object in $\npresilt{2}{}(A)$. 
	\end{setup}
	{{When we require} $H^iA$ to be finite dimensional for all $i\in \Z$, {we can be sure} that $\per(A)$ {is} Hom-finite \cite[Proof of Proposition 2.4]{Ami09} (and consequently Krull--Schmidt, by idempotent completeness).} 
	Note that finite dimensional $k$-algebras are included in our setup.
	
	It would be possible to generalise \Cref{setup:siltingtstr} to the setting of ST-triples $(\c,\d,S)$ \cite[§4]{AMY19}. We do not give the definition of ST-triples here, but we note that if $\c$ is an algebraic triangulated category  {(i.e. triangle equivalent to the stable category of a Frobenius category)}, {this triple} is equivalent to a triple of the form $(\per(A),\Der_{\rm fd}(A),A)$, where $A$ is as in \Cref{setup:siltingtstr} \cite[{{Corollary~6.13}}]{AMY19}. The assumption of algebricity does not reduce the scope of applications at time of writing, {as} ST-triples seem not to have {been used} in a {non-}algebraic context. 
	
	{In this section, we describe Iyama--Yang silting reduction and its interaction with silting t-structures.} Silting reduction with respect to $P$ occurs in the Verdier quotient $\per(A)/\thick(P)$, but we will show that subcategory 
	{$P^{\perp_{\Z}}_{\rm fd}{\defeq } P^{\perp_{\Z}} \cap \Der_{\rm fd}(A)$} of $ \Der_{\rm fd}(A)$ also plays a role{, where $P^{\perp_{\Z}} $ is a right perpendicular subcategory of $\Der(A)$} (see \eqref{eq:perpr}). If $\per(A)/\thick(P)$ is triangle equivalent to $\per(C)$ for some dg algebra $C$, {then we show that} $P^{\perp_{\Z}}_{{\rm fd}}$ is triangle equivalent to $\Der_{\rm fd}(C)$ (see \Cref{prop:PperpIsDfd}). Note that $P^{\perp_{\Z}}_{{\rm fd}}$ is a thick subcategory of $\Der_{\rm fd}(A)$. The t-structures defined by these silting objects will be restricted to $P^{\perp_{\Z}}_{{\rm fd}}$, which is achievable since $P^{\perp_{\Z}}_{{\rm fd}}$ is a t-exact subcategory of $\Der_{\rm fd}(A)$ ({to be shown in} \Cref{lem:perptexsummand}). 
	
	We first specify what it means for a t-structure to be generated by a silting object. {We employ the following notational convention: Given a {full subcategory $\mathcal{X}$} in $\per(A)$ and $I\subseteq \Z$, let ${\mathcal{X}}^{\perp_{I}}$ denote the right perpendicular of $\Der(A)$, as defined in \eqref{eq:perpr}, and define ${\mathcal{X}}^{\perp_{I}}_{\rm fd} \defeq {\mathcal{X}}^{\perp_{I}} \cap \Der_{\rm fd}(A)$. {If $\mathcal{X}=\add(T)$ for some object $T\in \per(A)$, we write ${T}^{\perp_{I}}_{\rm fd} \defeq {\add(T)}^{\perp_{I}}_{\rm fd}$.}
	
	\begin{proposition-definition}\label{lem:heartofS}
		Let $T$ be a silting object in $\per(A)$.
		{The pair $(T^{\perp_{> 0}},T^{\perp_{\leq 0}})$ forms a t-structure on $\Der(A)$ {\cite[Corollary 4.7(b)]{AI12}}. It restricts to a {bounded} t-structure $(T^{\perp_{> 0}}_{\rm fd},T^{\perp_{\leq 0}}_{\rm fd})$ on $\Der_{\rm fd}(A)$ {\cite[Proposition 4.6(b)]{AMY19}}}. A t-structure {of one of these forms} is said to be a \textit{silting t-structure}, and it is \textit{generated} by $T$.  
		Let ${\sigma}^0_T$ be the cohomology functor associated to the {silting} t-structure $(T^{\perp_{> 0}},T^{\perp_{\leq 0}})$. Then ${\sigma}^0_T$ sends $T$ to a projective generator of the heart $T^{\perp_{\neq 0}}$ {\cite[Proposition 4.3]{PV18}}, and the Hom-functor $\Der(A)(T,-)$ restricts to an exact equivalence of abelian categories {\cite[Proposition 4.6(a) (see also its proof)]{AMY19}}
		\begin{equation*}
		\Der(A)(T,-) \colon T^{\perp_{\neq 0}}_{{\rm fd}} \to \fgMod \End_{\Der_{\rm fd}(A)}(T).
		\end{equation*}
	\end{proposition-definition}
		
	\begin{lemma}\label{prop:newtex}
		Let $A$ and $P$ be as in \Cref{setup:siltingtstr} and let $(\U,\U^{\perp_{{0}}})$ be a t-structure on $\Der_{\rm fd}(A)$ such that $\U\subseteq P^{\perp_{>1}}_{{\rm fd}}$ {and $\U^{\perp_{{0}}}\subseteq P^{\perp_{\leq 0}}_{\rm fd}$}. Then the subcategory $P^{\perp_\Z}_{{\rm fd}}$ is t-exact with respect to $(\U,\U^{\perp_{{0}}})$. In particular, the pair $(\U\cap P^{\perp_\Z}_{{\rm fd}},\U^{\perp_{{0}}}\cap P^{\perp_\Z}_{{\rm fd}})$ is a t-structure on $P^{\perp_\Z}_{{\rm fd}}$.
	\end{lemma}
	\begin{proof}
		Let $\sigma_{\U}$ and $\sigma_{\U^{\perp_{{0}}}}$ be the truncation functors for $(\U,\U^{\perp_{{0}}})$. Fixing an arbitrary object $X\in P^{\perp_\Z}_{{\rm fd}}$, it suffices to show that $\sigma_{\U}X\in P^{\perp_\Z}_{{\rm fd}}$. {Rotating a} truncation triangle {gives}
		\begin{equation*}
		\begin{tikzcd}
		{\Sigma^{-1}\sigma_{\U^{\perp_{{0}}}}X} \arrow[r] & \sigma_{\U}X\arrow[r] & X\arrow[r] & \sigma_{\U^{\perp_{{0}}}}X.
		\end{tikzcd}
		\end{equation*}
		Since $\sigma_{\U^{\perp_{{0}}}}X\in \U^{\perp_{{0}}}\subseteq P^{\perp_{\leq 0}}_{{\rm fd}}$, the desuspension $\Sigma^{-1}\sigma_{\U^{\perp_{{0}}}}X$ is in $P^{\perp_{\leq 1}}_{{\rm fd}}$. Using that $P^{\perp_{\leq 1}}_{{\rm fd}}$ is extension-closed in $\Der_{\rm fd}(A)$, we deduce that $\sigma_{\U}X\in P^{\perp_{\leq 1}}_{{\rm fd}}$. Since $\sigma_{\U}X$ also belongs to $\U$, and thus to $P^{\perp_{>1}}_{{\rm fd}}$, it follows that $\sigma_{\U}X\in P^{\perp_{\leq 1}}_{{\rm fd}} \cap P^{\perp_{>1}}_{{\rm fd}} =P^{\perp_\Z}_{{\rm fd}}$, as desired.
	\end{proof}
	\begin{proposition}\label{lem:perptexsummand}
		Let $A$ be as in \Cref{setup:siltingtstr}.
		\begin{enumerate}
			\item\label{lem:perptexsummand1} Let $T$ be a silting object in {$\per(A)$} of which $Q$ is a direct summand. Then $Q^{\perp_{\Z}}_{{\rm fd}}$ is t-exact with respect to the t-structure $(T^{\perp_{>0}}_{{\rm fd}},T^{\perp_{\leq 0}}_{{\rm fd}})$ on $\Der_{\rm fd}(A)$.
			\item \label{lem:perptexsummand2} Let $P\in\npresilt{2}{}(A)$ (as in \Cref{setup:siltingtstr}). Then $P^{\perp_{\Z}}_{{\rm fd}}$ is t-exact with respect to the standard t-structure $(A^{\perp_{>0}}_{{\rm fd}},A^{\perp_{\leq 0}}_{{\rm fd}})$ on $\Der_{\rm fd}(A)$.
		\end{enumerate} 
	\end{proposition}
	\begin{proof}
		We first address \eqref{lem:perptexsummand1}. {It is clear that $T^{\perp_{>0}}_{{\rm fd}} \subseteq Q^{\perp_{>0}}_{{\rm fd}} \subseteq Q^{\perp_{>1}}_{{\rm fd}}$ {and $T^{\perp_{\leq 0}}_{{\rm fd}} \subseteq Q^{\perp_{\leq 0}}_{\rm fd}$}, since $Q$ is a direct summand of $T$.} {By \Cref{prop:newtex}, the claim in \eqref{lem:perptexsummand1} indeed holds.}
		
		We move on to \eqref{lem:perptexsummand2}.
		Since $P$ is {$2_S$}-term, one can by definition find a triangle
		\begin{equation*}
		A_1 \to A_0 \to P \to \Sigma A_1,
		\end{equation*}
		where $A_0,A_1\in \add(A)$. 
		It follows from a long-exact sequence argument that {$A^{\perp_{> 0}}_{\rm fd}\subseteq P^{\perp_{> 1}}_{\rm fd}$ and that $A^{\perp_{\leq 0}}_{\rm fd}\subseteq P^{\perp_{\leq 0}}_{\rm fd}$}. Since the conditions in \Cref{prop:newtex} again are met, we conclude that $P^{\perp_{\Z}}_{{\rm fd}}$ indeed is t-exact with respect to the standard t-structure on $\Der_{\mathrm{fd}}(A)$.
	\end{proof}
	
	We claimed in the introduction that the perpendicular category $P^{\perp_{\Z}}_{{\rm fd}}$ {has a connection to} the $\tau$-perpendicular category {of $H_A(P)$}. This assertion will now be made explicit.
	
	\begin{proposition}\label{prop:H0J}
		Let $A$ and $P$ be as in \Cref{setup:siltingtstr}, and let $(M,Q)=H_A(P)$ be the $\tau$-rigid {pair} {in $\fgMod{H^0A}$} corresponding to $P$.
		\begin{enumerate}
		\item\label{item:H0J1} 
		The lattice isomorphism
		\begin{equation}\label{eq:H0J}
		{-\cap A^{\perp_{\neq 0}}_{{\rm fd}}} \colon \texact_{(A^{\perp_{>0}}_{{\rm {fd}}},A^{\perp_{\leq 0}}_{{\rm {fd}}})}(\Der_{\rm fd}(A))\to\wide({A^{\perp_{\neq 0}}_{{\rm {fd}}}}),
		\end{equation}
		provided by \Cref{ZC17.2.5} {and \Cref{lem:plus_rev}},
		sends $P^{\perp_{\Z}}_{{\rm fd}}$ to {$M^{\perp_0}\cap {^{\perp_{0}}(\tau M)}\cap Q^{\perp_0}\cap A^{\perp_{\neq 0}}_{\rm fd}${, where we regard $M$, $\tau M$, and $Q$ as objects in $\Der_{{\rm fd}}(A)$ {in degree $0$}}. Consequently, by \Cref{lem:heartofS}, the lattice isomorphism}
		\begin{equation}\label{eq:H0JJ}
		{\Hom_{\Der_{{\rm fd}}(A)}}(A,-)\colon \texact_{(A^{\perp_{>0}}_{{\rm {fd}}},A^{\perp_{\leq 0}}_{{\rm {fd}}})}(\Der_{\rm fd}(A))\to\wide(H^0A)
		\end{equation}
		{in \Cref{ZC17.2.5} sends $P^{\perp_{\Z}}_{{\rm fd}}$ to $J(M,Q)$.}  
		\item\label{item:H0J2} {The isomorphism in \eqref{eq:H0JJ} restricts to an isomorphism from the poset of t-exact subcategories of the form $P^{\perp_{\Z}}_{\rm fd}$, for some two-term presilting object $P$, to the poset of $\tau$-perpendicular subcategories of $\fgMod{H^0A}$.}
		\end{enumerate}
	\end{proposition}
	\begin{proof}
		We first prove \eqref{item:H0J1}.
		By \Cref{lem:plus_rev}\eqref{lem:plus_rev2}, the bijection {$-\cap A^{\perp_{\neq 0}}_{\rm fd}$} in \eqref{eq:H0J} {is inverted by} $\thick_{\Der_{\rm fd}(A)}(-)$. Since these are mutually inverse isomorphisms of posets, it suffices to prove that
		\begin{align}
		P^{\perp_{\Z}}_{{\rm fd}}\cap A^{\perp_{\neq 0}}_{{\rm fd}}&\subseteq {I(M,Q)},\label{eq:H0J1} \\
		\thick_{\Der_{\rm fd}(A)}\big({I(M,Q)} \big)&\subseteq P^{\perp_{\Z}}_{{\rm fd}}.\label{prop:H0J2}
		\end{align}
		{where $I(M,Q)\defeq {M^{\perp_0}\cap {^{\perp_{0}}(\tau M)}\cap Q^{\perp_0}\cap A^{\perp_{\neq 0}}_{\rm fd}}\subseteq A^{\perp_{\neq 0}}$.}
		The inclusion in \eqref{eq:H0J1} will be addressed first. Let $X\in P^{\perp_{\Z}}_{{\rm fd}}\cap A^{\perp_{\neq 0}}_{{\rm fd}}$. It is to be {shown} that $X\in M^{\perp_0}\,\cap\, {^{\perp_0}(\tau M) \cap Q^{\perp_0}} \cap A^{\perp_{\neq 0}}_{{\rm fd}}$.
		{Decompose $P$ into {a direct sum of the form} $P_M\oplus \Sigma Q'$, where the second term {$\Sigma Q'$} is the largest direct summand of $P$ in $\add(\Sigma A)$. The definition of $H_A$ gives that $H^0Q' \simeq Q$ as $H^0A$-modules. {By considering the truncation triangle in $\Der_{{\rm fd}}(A)$}
		 \begin{equation*}
		 	\begin{tikzcd}
				\sigma_A^{\leq -1}Q' \arrow[r] & Q'\arrow[r] & Q \arrow[r] & \Sigma\sigma_A^{\leq -1}Q'
			\end{tikzcd}
		 \end{equation*}
		 {it follows from a long-exact sequence argument that}
		\begin{equation}\label{eq:Q'Qfd}
		(Q')^{\perp_{0}} \cap A^{\perp_{\neq 0}}_{\rm fd} \,{=}\, Q^{\perp_{0}}\cap A^{\perp_{\neq 0}}_{\rm fd}.
		\end{equation}
		Thus,
		\[P^{\perp_{\Z}}_{{\rm fd}} \cap A^{\perp_{\neq 0}}_{{\rm fd}} \subseteq (Q')^{\perp_{\Z}}_{{\rm fd}} \cap A^{\perp_{\neq 0}}_{{\rm fd}} \subseteq (Q')^{\perp_{0}} \cap A^{\perp_{\neq 0}} _{{\rm fd}} \subseteq Q^{\perp_{0}}\cap A^{\perp_{\neq 0}}_{{\rm fd}}.  \]
		This shows that $X\in Q^{\perp_0} \cap A^{\perp_{\neq 0}}_{{\rm fd}}$.}
		Consider the following truncation triangle in the t-structure {$(A^{\perp_{\geq 0}}_{{\rm fd}},A^{\perp_{<0}}_{{\rm fd}})$ on $\Der_{{\rm fd}}(A)$:}
		\begin{equation}\label{eq:H0jtri1}
		\begin{tikzcd}
		\sigma^{< 0}_{A}P \arrow[r] & P \arrow[r] & {\sigma^{\geq 0}_{A}P}\arrow[r]&\Sigma \sigma^{< 0}_{A}P
		\end{tikzcd}
		\end{equation}
		{Since $P$ is two-term (so $\sigma^{\geq 1}_{{A}}P\simeq 0$) and ${\sigma}^0_A P\simeq M$, the third term in this triangle is isomorphic to $M$.} 
		{The outer terms in \eqref{eq:H0jtri1} are in $A^{\perp_{\geq 0}}$, so there are no non-trivial maps from these objects to $X\in A^{\perp_{\neq 0}}_{\rm fd}\subseteq A^{\perp_{< 0}}_{\rm fd}$. Applying the contravariant functor $\Hom_{\Der_{{\rm fd}}(A)}(-,X)$ to the triangle in \eqref{eq:H0jtri1} thus yields an isomorphism} 
		\begin{equation}\label{eq:bodom}
		\begin{tikzcd} \Hom_{\Der_{{\rm fd}}(A)}(P,X) & \arrow[l,"\simeq"'] \Hom_{\Der_{{\rm fd}}(A)}(M,X), 
		\end{tikzcd}
		\end{equation}
		{so $X\in M^{\perp_0}$, {as $X\in P^{\perp_{\Z}}$.}}
		{So far, we have shown that $X\in M^{\perp_0}\cap Q^{\perp_0} \cap A^{\perp_{\neq 0}}_{{\rm fd}}$.}
		{To conclude}, we show that $\Hom_{{\Der_{{\rm fd}}(A)}}(X,\tau M)=0$. 
		Consider a triangle in $\Der_{{\rm fd}}(A)$
		\begin{equation*}
		\begin{tikzcd}
		P_1 \arrow[r,"p_M"] & P_0 \arrow[r,"{{\beta_M}}"] & P_M \arrow[r]&\Sigma P_1,
		\end{tikzcd}
		\end{equation*}
		{where $P_0,P_1 \in {\add}(A)$ and {$\beta_M$ is a minimal right $\add(A)$-approximation.}}
		The induced right exact sequence 
		\begin{equation}\label{eq:seqseq}
		\begin{tikzcd}
		\Hom_{\Der_{{\rm fd}}(A)}(A,P_1) \arrow[r,"p_M\circ -"] & \Hom_{\Der_{{\rm fd}}(A)}(A,P_0) \arrow[r,"{{\beta_M}}\circ -"] & \Hom_{\Der_{{\rm fd}}(A)}(A,P_M) \arrow[r]&0 
		\end{tikzcd}
		\end{equation}
		{is {then} a minimal projective presentation of the $H^0A$-module $\Hom_{\Der_{{\rm fd}}(A)}(A,P_M)$, {since no indecomposable object in $\add(\Sigma A)$ is a direct summand of $P_M$}. By \Cref{lem:heartofS}, a {quasi-}inverse of the exact equivalence $\Hom_{\Der_{{\rm fd}}(A)}(A,-) \colon A^{\perp_{\neq 0}}_{{\rm fd}}\to \fgMod H^0 A$ sends the right exact sequence in \eqref{eq:seqseq} to the right exact sequence} 
		\begin{equation*}
				\begin{tikzcd}[column sep=4em]
		\sigma^0_AP_1\arrow[r,"\sigma^0_A p_M\circ -"] & \sigma^0_A P_0 \arrow[r,"\sigma^0_A{{\beta_M}}\circ -"] & \sigma^0_A {P_M}\arrow[d,symbol=\simeq] \arrow[r]&0 \\
		&& M &
		\end{tikzcd}
		\end{equation*}
		{in $A^{\perp_{\neq 0}}$, where $\sigma^0_A$ is the cohomology functor for the silting t-structure $(A^{\perp_{>0}}_{{\rm fd}}, A^{\perp_{\leq 0}}_{{\rm fd}})$ on $\Der_{{\rm fd}}(A)$. Using \Cref{lem:heartofS} to identify $\fgMod H^0A$ with $A^{\perp_{\neq 0}}_{{\rm fd}}$, we may claim that $\Hom_{{\Der_{{\rm fd}}(A)}}(X,\tau M)=0$ if and only if the induced map} $$\begin{tikzcd}[column sep=3em] \Hom_{\Der_{{\rm fd}}(A)}(\sigma^0_A P_1,X) & \arrow[l,"-\circ \sigma^0_Ap_M"']  \Hom_{\Der_{{\rm fd}}(A)}(\sigma^0_AP_0,X) \end{tikzcd}$$ {surjects \cite[Proposition 2.4(b)]{AIR14}, or indeed that}
		\begin{equation*}
			\begin{tikzcd}[column sep=4em] \Hom_{\Der_{{\rm fd}}(A)}(\sigma^{\geq 0}_{A} P_1,X) & \arrow[l,"-\circ \sigma^{\geq 0}_{A} p_M"']  \Hom_{\Der_{{\rm fd}}(A)}(\sigma^{\geq 0}_{A}P_0,X) \end{tikzcd} 
		\end{equation*}
		{surjects, since $\sigma^0_A Y=\sigma^{\geq 0}_{A}Y$ for all $Y\in A^{\perp_{> 0}}$. Since $\sigma^{\geq 0}_{A}$ is left adjoint to the inclusion of $A^{\perp_{<0}}$ into $\Der_{{\rm fd}}(A)$, it is equivalent to show that}
		\begin{equation}\label{eq:mormor}
			\begin{tikzcd}[column sep=3em] \Hom_{\Der_{{\rm fd}}(A)}(P_1,X) & \arrow[l,"-\circ p_M"']  \Hom_{\Der_{{\rm fd}}(A)}(P_0,X) \end{tikzcd} 
		\end{equation}
		{surjects. This is shown by a long-exact sequence argument, applying $\Hom_{\Der_{{\rm fd}}(A)}(-,X)$ to the triangle in \eqref{eq:H0jtri1}. We have shown that $\Hom_{{\Der_{{\rm fd}}(A)}}(X,\tau M)=0$, and moreover that the inclusion in \eqref{eq:H0J1} holds.}
		
		To prove that the inclusion in \eqref{prop:H0J2} holds, it suffices to show that ${I(M,Q)}\subseteq P^{\perp_{\Z}}_{{\rm fd}}$, since $P^{\perp_{\Z}}_{{\rm fd}}$ is {a thick subcategory of $\Der_{\rm fd}(A)$. Since ${I(M,Q)}\subseteq A^{\perp_{\neq 0}}_{{\rm fd}}$ and $P$ is two-term, we can immediately assert that ${I(M,Q)}\subseteq P^{\perp_{\neq 0,1}}_{\rm fd}$. We {now} fix an $X$ in ${I(M,Q)}$, and apply the arguments in the previous paragraph in reverse. Using the isomorphism in \eqref{eq:bodom}, we see that $X\in P_{\rm fd}^{\perp_0}$. To show that ${I(M,Q)}\subseteq P_{{\rm fd}}^{\perp_1}$, one can use the decomposition $P\simeq P_M\oplus \Sigma {Q'}$ {to} show that ${I(M,Q)}\subseteq P_{M,{\rm fd}}^{\perp_1}\cap (\Sigma {Q'})^{\perp_{1}}_{\rm fd} = P_{M,{\rm fd}}^{\perp_1}\cap {(Q')}^{\perp_{0}}_{\rm fd}$.  The assumption that $\Hom_{{\Der_{{\rm fd}}(A)}}(X,\tau M)=0$, or rather the equivalent assertion that $-\circ p_M$ in \eqref{eq:mormor} surjects, combines with a long-exact sequence argument to show that $\Hom_{\Der_{{\rm fd}}(A)}(\Sigma^{-1}P_M,X)=0$, whence $X\in P_{M,{\rm fd}}^{\perp_1}$. Having also assumed that $\Hom_{\Der_{{\rm fd}}(A)}(Q, X)=0$, we have {by \eqref{eq:Q'Qfd}} that $X\in ({Q'})^{\perp_{0}}_{\rm fd}$, letting us conclude the proof of \eqref{item:H0J1}.}
			
	{We finish the proof by addressing \eqref{item:H0J2}. Indeed, since the map in \eqref{eq:H0J} injects, so does the restricted map from the poset of t-exact subcategories of the form $P^{\perp_{\Z}}$, for some two-term presilting object $P$, to the poset of $\tau$-perpendicular subcategories of $\fgMod{H^0A}$. Surjectivity follows from \eqref{item:H0J1} above, since it suffices to point out that $H^{-1}_{A}(M,Q)^{\perp_{\Z}}_{\rm fd}$ is a preimage of the arbitrary $\tau$-perpendicular subcategory $J(M,Q)$.}
	\end{proof}
	
	\Cref{section:tcmc} is devoted to the construction of the $\tau$-cluster morphism category in terms of two-term silting. We conclude this section with results that will turn out useful for this purpose. To prove them, we will apply some basic results from the localisation theory of compactly generated triangulated categories, especially in the {``large''} derived category $\Der(A)$, where the objects are possibly unbounded complexes of possibly infinitely generated modules. \Cref{prop:largeloc} addresses the localisation theory we need, whereas the subsequent corollaries will be applied in later sections.
	
	\begin{proposition}\label{prop:largeloc}
		Let $A$ and $P$ be as in \Cref{setup:siltingtstr} and let $\Loc(P)$ be the smallest thick subcategory of $\Der(A)$ {containing $P$} that is closed under set-indexed coproducts.
		\begin{enumerate}
			\item\label{prop:largeloc1} \cite[Theorem~5.6.1]{Kra08} There is a recollement of triangulated categories
			\begin{equation*}
			\begin{tikzcd}[column sep=5em]
			P^{\perp_{\Z}} \arrow[r, "\iota_P" description, hook] & \Der(A) \arrow[r, "\rho_P" description] \arrow[l, "{\varpi_P}"', bend right=49] \arrow[l, bend left=49,"\kappa_P"] & \Loc(P) \arrow[l, "\lambda_P"', hook, bend right=49] \arrow[l, bend left=49,tail,"\mu_P"]
			\end{tikzcd}
			\end{equation*}
			i.e. {\cite[§1.4.3]{BBD82}} the diagram displays four adjoint pairs of triangle functors, the composite ${\varpi_P} \lambda_P$ is zero, all functors into $\Der(A)$ are fully faithful, and for all $X\in \Der(A)$ we have triangles
			\begin{align}
			&\begin{tikzcd}[ampersand replacement=\&]
			\lambda_P\rho_P X \arrow[r,"\varepsilon"] \& X \arrow[r,"\eta"] \& \iota_P{\varpi_P} X \arrow[r] \& \Sigma (\lambda_P\rho_P X),
			\end{tikzcd} \label{eq:rectri1} \\
			&\begin{tikzcd}[ampersand replacement=\&]
			\iota_P\kappa_P X \arrow[r,"\varepsilon'"] \& X \arrow[r,"\eta'"] \& \mu_P\rho_P X \arrow[r] \& \Sigma (\iota_P\kappa_P X),
			\end{tikzcd} \label{eq:rectri2}
			\end{align}
			where $\eta$ and $\eta'$ (resp. $\varepsilon$ and $\varepsilon'$) are units (resp. co-units) of the adjunctions. 
			\item\label{cor:piBon=piA} Let $T^+_P$ be the Bongartz completion of $P$ as a two-term presilting object in $\per(A)$. Then ${\varpi_P} T^+_P \simeq {\varpi_P} A$ {in $P^{\perp_{\Z}}$}.
			\item\label{prop:largelocte} \cite[Proposition 3]{BH08} We have a triangle equivalence {$\mathbb{R}\Hom_{{{A}}}({\varpi_P}{A},-) \colon P^{\perp_{\Z}}\to^{\simeq} \Der(C_P)$}, {where $C_P$ {is the dg endomorphism algebra $\mathcal{E}nd_A({\varpi_P} A)$. Moreover, the dg $k$-algebra $C_P$} {satisfies \Cref{setup:siltingtstr}, which to say that it} is non-positive {and that $H^iC_P$ is finite dimensional for all $i\in\Z$}.}
			\item\label{prop:largeloccptpi} {The composite of ${\varpi_P}$ with the equivalence in {\eqref{prop:largelocte}} restricts to a functor} ${\pi_P\colon} \per(A)\to \per(C_P)$, which is a Verdier localisation functor with respect to $\thick(P)$. In particular, we have triangle equivalence $$\per(A)/\thick(P)\simeq \per(C_P).$$
			\item\label{prop:largelocbdd}  The functor $\Der(C_P)\simeq P^{\perp_{\Z}}\to \Der(A)$ preserves boundedness, i.e. it restricts to an embedding $\Der_{\rm fd}(C_P)\embed \Der_{\rm fd}(A)$.
		\end{enumerate}
	\end{proposition}
	\begin{proof}
		The assertion in \eqref{prop:largeloc1} is taken directly from the cited reference. {We note that $\Loc(P)^{\perp_0}=P^{\perp_{\Z}}$; the left hand side is contained in the right hand side since $\Loc(P)\supseteq {\Sigma^i}P$ for all $i\in \Z$, and the right hand side is contained in the left since $\Loc(P)$ can be built from $\add(P)$ using suspensions, desuspensions, extensions, and coproducts.} {The functors $\iota_P$ and $\lambda_P$ are naturally isomorphic to the inclusion functors, and we will simply treat them as inclusions for the remainder of this proof.}
		
		{Assertion \eqref{cor:piBon=piA} follows from \eqref{prop:largeloc1}, as we now show.
		Recall that the Bongartz completion $T^{+}_P$ is given by $P\oplus Q^+$, where $Q^+$ is defined by the triangle (see \eqref{eq:Bon-tri})
		\begin{equation*}
	{ A}\to { Q^+}\to { P}_0\to^{\beta_{\Sigma { A}}}\Sigma { A},
	\end{equation*}
	in which $\beta_{\Sigma { A}}$ is a right $\add(P)$-approximation. Since the functor ${\varpi_P}$ sends $P$ {and $\Sigma^{-1} P$} to zero object{s}, applying ${\varpi_P}$ to the triangle above thus gives that ${\varpi_P} T^+_P \simeq {\varpi_P} Q^{+} \simeq {\varpi_P} A$ {in $P^{\perp_{\Z}}$}.}
		
		{{We now prove the first claim in \eqref{prop:largelocte}.} {Since \eqref{prop:largeloc1} provides a recollement}, we may identify $P^{\perp_{\Z}}$ with the Verdier quotient $\Der(A)/\Loc(P)$ {\cite[Proposition 1.4.5]{BBD82}}. We may now point to the cited reference {\cite[Proposition 3]{BH08}} to prove the first claim in \eqref{prop:largelocte}, {concerning the existence of the dg $k$-algebra $C_P$.} {It will be useful to prove \eqref{prop:largeloccptpi} before showing that $C_P$ indeed satisfies \Cref{setup:siltingtstr}.}
				
		{We move on to \eqref{prop:largeloccptpi}. The assertions in \eqref{prop:largeloc1} and \eqref{prop:largelocte} yield a triangle equivalence 
		\begin{equation}\label{eq:largelocte1}
			\Der(A)/\Loc(P) \to^{\sim} \Der(C_P).
		\end{equation}
		{The subcategory of compact objects in $\Der(B)$ is $\thick(B)=\per(B)$, for any dg algebra $B$ \cite[§5.3]{Kel94}. Since the right adjoint $\iota_P$ preserves coproducts, it is indeed the case that the left adjoint ${\varpi_P}$ preserves compactness \cite[Theorem 5.1]{Nee96}.} {The composite functor
		\begin{equation}\label{eq:compofunc}
			\mathbb{R}\Hom_{{{A}}}({\varpi_P}{A},-)\circ \varpi_P\colon \per(A) \to \Der(C_P)
		\end{equation}
		thus takes values in $\per(C_P)$, so we can define a triangle functor}
		\begin{equation}\label{eq:piper}
		\pi_P\colon \per(A) \to \per(C_P), 
		\end{equation}
		{by restriction,} in turn inducing a triangle functor
		\begin{equation}\label{eq:pipers}
		 \per(A)/\thick(P) \to \per(C_P). 
		 \end{equation}
		 The functor {in \eqref{eq:pipers}} can be obtained by restricting the equivalence in \eqref{eq:largelocte1} to compact objects. It is thus an equivalence, thanks to the idempotent completeness of $\Der(C_P)$ \cite[Theorem~2.1]{Nee92}. This proves \eqref{prop:largeloccptpi}.}
		 
{Penultimately, we return to \eqref{prop:largelocte}.} Brüning--Huber show that the $k$-linear map}
		\begin{equation*}
		\begin{tikzcd}[column sep=10em]
			\Hom_{P^{\perp_{\Z}}}({\varpi_P} A, \Sigma^i {\varpi_P} A) \arrow[r,"{\mathbb{R}\Hom_{{{A}}}({\varpi_P}-,-)}"] & \Hom_{\Der(C_P)}(C_P, \Sigma^i C_P) \simeq H^i C_P
		\end{tikzcd}
		\end{equation*}
		is an isomorphism {for any $i\in \Z$ \cite[Proof of Proposition 3]{BH08}. We show that $\Hom_{P^{\perp_{\Z}}}({\varpi_P} A, \Sigma^i {\varpi_P} A)$ is finite dimensional for all $i\in \Z$, and that it vanishes for all $i>0$. By \eqref{cor:piBon=piA} above, it is equivalent to prove the same claim for $\Hom_{P^{\perp_{\Z}}}({\varpi_P} T^+_P, \Sigma^i {\varpi_P} T^+_P)$. As argued above, the composite functor in \eqref{eq:compofunc} sends the compact object $T^+_P$ of $\Der(A)$ to a compact object in $\Der(C_P)$. {Since $\per(C_P)$ is the subcategory of $\Der(C_P)$ consisting of compact objects \cite[§5.3]{Kel94}}, we may conclude by showing that $\Hom_{\per(C_P)}({\pi_P} T_P^+, \Sigma^i {\pi_P} T_P^+)$ is finite dimensional for all $i\in \Z$ and vanishes for $i>0$. By \Cref{IY18.3.7}\eqref{IY18.3.7.i} and the existence of the triangle equivalence in \eqref{eq:pipers}, we have that $\pi_PT_P^+$ is presilting in $\per(C_P)$, yielding vanishing for $i>0$. Moreover, the equivalence $$\per(C_P)\simeq \per(A)/\thick(P) \simeq \z_P/[P],$$
recalled in \Cref{IY18.3.7}\eqref{IY.3.6}, shows that $\per(C_P)$ is Hom-finite, whereby we conclude the proof of \eqref{prop:largelocte}.}
		
		Right adjoints between derived categories of dg algebras preserve boundedness if the left adjoint preserves compactness \cite[Lemma~4.2]{GP18}. {Since ${\varpi_P}$ preserves compactness, the last assertion \eqref{prop:largelocbdd} now follows from \eqref{prop:largeloccptpi}.}
	\end{proof}
	
	{
\begin{proposition}\label{item:H0J3}
Let $A$ be as in \Cref{setup:siltingtstr}, {and let $U,V\in\npresilt{2}{}(A)$. Then ${U}^{\perp_{\Z}}_{\rm fd}={V}^{\perp_{\Z}}_{\rm fd}$ as t-exact subcategories of $\Der_{\rm fd}(A)$ (see \Cref{lem:perptexsummand}\eqref{lem:perptexsummand2}) if and only if $\thick(U)=\thick(V)$ as thick subcategories of $\per(A)$.}
\end{proposition}
}
\begin{proof}
	{Suppose first that $\thick(U)=\thick(V)$. We then have that $$ {U}^{\perp_{\Z}}_{\rm fd}  = {\thick(U)}^{\perp_{\Z}}_{\rm fd} =  {\thick(V)}^{\perp_{\Z}}_{\rm fd}  = {V}^{\perp_{\Z}}_{\rm fd}, $$ where the outer equalities are straightforward to establish, using that $\thick(U)$ (resp. $\thick(V)$) is obtained from $\add(U)$ (resp. $\add(V)$) by closure under cones, extensions, and direct summands. To prove the converse, we show that $\thick(U)\neq \thick(V)$ implies that ${U}^{\perp_{\Z}}_{\rm fd}\neq {V}^{\perp_{\Z}}_{\rm fd}$. Without loss of generality, we may assume that $V\not\in\thick(U)$, from which we deduce that ${U}^{\perp_{\Z}}_{\rm fd} \not\subseteq {V}^{\perp_{\Z}}_{\rm fd}$. 	
	Let $C_U$ as in \Cref{prop:largeloc}\eqref{prop:largelocte}.
	By \Cref{prop:largeloc}\eqref{prop:largeloccptpi}, we have a triangle functor
	$$\pi_U\colon \per(A)\to \per(C_U), $$
	given by the restriction of the following composite {to $\per(A)$:}
	$$\Der(A) \to^{\varpi_U} U^{\perp_{\Z}} \to \Der(C_U).$$  By \Cref{prop:largeloc}\eqref{prop:largeloc1} and \eqref{prop:largelocte}, both functors {in the composition} admit a right adjoint. Let $G\colon \Der(C_U) \to \Der(A)$ denote a right adjoint of the composite.}
	
	 {Now, let $N=\bigoplus_i N_i$ denote a direct sum in the heart {$\Der^0_{\rm fd}(C_U)$} of the standard t-structure on $\Der_{\rm fd}(C_U)$, where the terms $N_i$ are all the simple objects in this heart up to isomorphism. 
	 {By \Cref{lem:heartofS}, the cohomology functor $\sigma^0_{C_U}$ sends $C_U$ to the projective generator of {$\Der_{\rm fd}^0(C_U)$}. Let $T_i$ denote the direct summand of $C_U$ such that $\Hom_{\Der_{\rm fd}(C_U)}(\sigma^0_{C_U} T_i,N_i)$ is non-zero and $T_i\in  {^{\perp_{\neq 0}}N_i}$. Note that the truncation triangle}
	 $$\sigma^{\leq -1}_{C_U} C_U \to C_U \to  \sigma^{0}_{C_U} C_U \to \Sigma \sigma^{\leq -1}_{C_U} C_U$$
	 {can be used in a long-exact sequence argument to show that $\Hom_{\Der_{\rm fd}^0(C_U)}(\sigma^0_{C_U} T_j,N_i)$ is isomorphic to $\Hom_{\Der_{\rm fd}(C_U)}(T_j,N_i)$, where $T_j$ is an arbitrary direct summand of $C_U$. In conjunction with Schur's Lemma, the isomorphisms
	 $$\Hom_{\Der_{\rm fd}(C_U)}(T_i,N_i) \to \Hom_{\Der_{\rm fd}^0(C_U)}(\sigma^0_{C_U} T_i,N_i) \to \Hom_{\Der_{\rm fd}^0(C_U)}(N_i,N_i)$$
	 lets us transfer a division $k$-algebra structure to $\Hom_{\Der_{\rm fd}(C_U)}(T_i,N_i)$ for all $i$.} 
	 A result of Garcia \cite[Lemma 4.20]{Gar24} {now} shows that 
	$$  {^{\perp_{\Z}}N_i} \cap \per(C_U) \ = \thick(\add(C_U)\setminus \add(T_i)),$$
	  Consequently,
	$$ {^{\perp_{\Z}}N} \cap \per(C_U) = \bigcap_i  \thick(\add(C_U)\setminus \add(T_i)) = \thick(\add(C_U)\setminus \add(\bigoplus_i T_i)).$$ The right hand side is trivial; indeed, for any indecomposable direct summand $T_j$ of $C_U$, the simple top $N_j$ of $\sigma_{C_U}^0(T_j)$ is such that $\Hom_{\Der(C_U)}(T_j,N_j)$ non-zero and $T_j\in  {^{\perp_{\neq 0}}N_j}$  \cite[Corollary 4.8] {AMY19}, so $\add(\bigoplus_i T_i) = \add(C_U)$.}
	
{
Consider the following subcategory of $\per(A)$:
$$\mathcal{F}_N \defeq \{X\in \per(A) \sth \pi_U X\in {^{\perp_{\Z}}N}  \}. $$
By the previous paragraph,
$$X\in \mathcal{F}_N \Longleftrightarrow \pi_U X \text{ is a zero object.}  \Longleftrightarrow X\in\thick(U), $$
so $\mathcal{F}_N = \thick(U)$.
The right adjoint $G$ preserves boundedness as a result of \Cref{prop:largeloc}\eqref{prop:largelocbdd}, as it is naturally isomorphic to $\iota_P$ composed with an equivalence. We thus have that $G N\in \thick(U)^{\perp_{\Z}}_{\rm fd} = U^{\perp_{\Z}}_{\rm fd}$. However, as $\thick(U)=\mathcal{F}_N$, the assumption that $V\not \in \thick(U)$ now leads to that $G N\not\in  V^{\perp_{\Z}}_{\rm fd}$. Having found an object which is in $U^{\perp_{\Z}}_{\rm fd}$, but not in $V^{\perp_{\Z}}_{\rm fd}$, the proof is complete.}
\end{proof}

	\begin{corollary}\label{prop:PperpIsDfd}
		Let $A$ and $P$ be as in \Cref{setup:siltingtstr}, {and} let $C_P$ as in \Cref{prop:largeloc}\eqref{prop:largelocte}.
		Then {there is a triangle equivalence $P^{\perp_{\Z}}_{{\rm fd}} \simeq \Der_{\mathrm{fd}}(C_P)$.} 		
		\end{corollary}
	\begin{proof}
		 \Cref{prop:largeloc}\eqref{prop:largelocte} provides an equivalence from the perpendicular {sub}category $P^{\perp_{\Z}}$ {of} $\Der(A)$ to $\Der(C_P)$, sending ${\varpi_P} A$ to $C_P$. {In order to prove that this equivalence restricts to the bounded derived categories, it remains to show that the {$k$-}vector space $ \bigoplus_{i\in \Z} \Hom_{\Der(A)}(A, \Sigma^{i} \iota_P X)$ is finite dimensional precisely when $\bigoplus_{i\in \Z}\Hom_{P^{\perp_{\Z}}}(\varpi_PA, \Sigma^{i} X)$ is finite dimensional, where $X\in P^{\perp_{\Z}}$. {Indeed, the former condition characterises the objects in $\Der_{\rm fd}(A)$ by definition, and since $\varpi_P A$ is a sent to $C_P$ by the equivalence $P^{\perp_{\Z}} \to^{\simeq}\Der(C_P)$ in  \Cref{prop:largeloc}\eqref{prop:largelocte}, the latter condition ensures that $P^{\perp_{\Z}}_{\rm fd}$ is sent to $\Der_{\rm fd}(C_P)$.}
		 }
		 {The adjoint pair ${\varpi_P} \dashv \iota_P$ provides an isomorphism 
		 $$  \bigoplus_{i\in \Z} \Hom_{\Der(A)}(A, \Sigma^{i} \iota_P X) \to^{\simeq} \bigoplus_{i\in \Z}\Hom_{P^{\perp_{\Z}}}(\varpi_PA, \Sigma^{i} X),  $$ which proves our claim.}
			\end{proof}
	
	\begin{corollary}\label{prop:PperpIsDfdcor}
		Let $A$ and $P$ be as in \Cref{setup:siltingtstr}. {The functor $\pi_P\colon \per(A)\to \per(C_P)$ from \Cref{prop:largeloc}\eqref{prop:largeloccptpi} induces a bijection
		$$\pi_P\colon \npresilt{2}{}_P(A) \to \npresilt{2}{}(C_P). $$}
		{It fits into a commutative diagram of bijections:}
		\begin{equation}\label{eq:PperpIsDfdcor1}
		\begin{tikzcd}	
		\npresilt{2}{}_{{ P}}({A}) \arrow[rd, "{\varphi}_{ P}"] \arrow[dd, "\pi_{P}"'] &                                                                                  \\
                                                                                 & {\npresilt{2}{{T^+_P}}({\z_P/[P]})} \arrow[ld,"\zeta_P"] \\
		\npresilt{2}{}(C_{P})                                                            &                                                                                 
		\end{tikzcd}
		\end{equation}	
		{where $\zeta_P$ is induced by {the composite} triangle equivalence ${\zeta_P\colon} \z_P/[P] \to^{} {\per(A)/\thick(P)} \to^{} \per(C_P)$ sending $T^+_P$ to $C_P$.}
		\end{corollary}
		\begin{proof}
			{{By \Cref{prop:largeloc}\eqref{prop:largeloccptpi},} we have a triangle equivalence $\per(A)/\thick(P) \to^{\sim} \per(C_P)$, sending $A$ to $C_P$. By \Cref{prop:largeloc}\eqref{cor:piBon=piA}, it also sends the Bongartz completion $T^+_P$ to $C_P$. Using \Cref{IY18.3.7}\eqref{IY.3.6}, one constructs a triangle equivalence $\z_P/[P] \to^{\zeta_P} \per(C_P)$ sending $T^+_P$ to $C_P$.  Since we can regard the functor $\pi_P$ as a Verdier localisation, we deduce from \Cref{IY18.3.7}\eqref{IY18.3.7.i} that it indeed induces a bijection
	$$\npresilt{2}{}_P(A) \to^{\varphi_P} {\npresilt{2}{{T^+_P}}({\z_P/[P]})}  \to^{\zeta_P} \npresilt{2}{}(C_P), $$ as claimed,} {and the commutativity of \eqref{eq:PperpIsDfdcor1} is immediate.}
			\end{proof}
	
	{
	\begin{corollary}\label{lem:functorphi'}
	Let $A$ and $P$ be as in \Cref{setup:siltingtstr}, and let $Q\in \npresilt{2}{}(C_P)$. {By \Cref{prop:largeloc}\eqref{prop:largelocte}, we have that $C_P$ and $Q$ also satisfy \Cref{setup:siltingtstr}, letting us} {consider the triangle functors}
	\begin{align*}
	\pi_P\colon \per(A) &\to \per(C_P) \\
	\pi_Q\colon \per(C_P) &\to \per(C_{P,Q}) \\
	\pi_{\pi_P^{-1}(Q)}\colon \per(A) &\to \per({C_{\pi_P^{-1}(Q)}}) 
	\end{align*}
	provided by \Cref{prop:largeloc}\eqref{prop:largeloccptpi}. {There is a triangle equivalence $\per(C_{P,Q})\to \per(C_{\pi_P^{-1}(Q)})$, which we may use to identify the} composite $\pi_Q\circ \pi_P$ {with} $\pi_{\pi_P^{-1}(Q)}$ {up to natural isomorphism}. 
	{Consequently, we have that the bijections}
	\begin{align*}
	\pi_P\colon \npresilt{2}{}_P(A) &\to \npresilt{2}{}(C_P) \\
	\pi_Q\colon \npresilt{2}{}_Q(C_P) &\to \npresilt{2}{}(C_{P,Q}) \\
	\pi_{\pi_P^{-1}(Q)}\colon \npresilt{2}{}_{\pi_P^{-1}(Q)}(A) &\to \npresilt{2}{}({C_{\pi_P^{-1}(Q)}}) 
	\end{align*}
	provided by \Cref{prop:PperpIsDfdcor}, satisfy $\pi_Q \circ \pi_P = \pi_{\pi_P^{-1}(Q)}$.
	\end{corollary}}
	\begin{proof}
	{Consider the diagram {of triangle functors, where the horizontal functors are fully faithful and the vertical functors are Verdier localisations:}}
	\begin{equation*}
	\begin{tikzcd}
\thick(P) \arrow[r, hook] & \thick({\pi_P^{-1}(Q)}) \arrow[r, hook] \arrow[d] & \per(A) \arrow[dd, "\pi_{\pi_P^{-1}(Q)}", bend left=49] \arrow[d, "\pi_P"'] \\
                          & \thick(Q) \arrow[r, hook]                         & \per(C_P) \arrow[d, "\pi_Q"']                                               \\
                          &                                                   & {\per({{C_{\pi_P^{-1}(Q)}}})}                                                          
\end{tikzcd}
	\end{equation*}
	{The universal property of Verdier localisations now induces a triangle equivalence $\per(C_{P,Q})\to \per(C_{\pi_P^{-1}(Q)})$. } {It follows that $\pi_{\pi_P^{-1}(Q)}$ {can be identified with} $\pi_Q\circ \pi_P$ {up to natural isomorphism}, as shown on the right in our diagram.}
	\end{proof}
	
	\section{Compatibility of the Iyama--Yang and Buan--Marsh bijections}\label{sec:comp}
	
	Reduction techniques have been developed for both silting objects (known as \textit{silting reduction} \cite{IY18}) and support $\tau$-tilting modules (\textit{$\tau$-tilting reduction} \cite{Jas15}), and they are compatible \cite[Theorem~4.12(b)]{Jas15}. In this section we generalise Jasso's compatibility theorem, by showing that the support $\tau$-tilting reduction of Buan--Marsh \cite[Section~3]{BM18w} is compatible with silting reduction. 
	
	{Buan--Marsh prove the following.}
	\begin{theorem}[{\cite[Theorem~3.6]{BM18w}}]\label{thm:BM18w.3.6}
		Let ${\Lambda}$ be a finite dimensional $k$-algebra, and let $(M,Q)$ be a $\tau$-rigid pair in $\fgMod {\Lambda}$. There is a bijection
		\begin{equation*}
		\mathcal{E}_{(M,Q)}\colon \staurigidpair_{(M,Q)}({\Lambda}) \to  \staurigidpair(J(M,Q)),
		\end{equation*}
		where $J(M,Q)$ is the $\tau$-perpendicular category of $(M,Q)$.
	\end{theorem}
	We now recall how the map $\mathcal{E}_{(M,Q)}$ is constructed. {For a $\tau$-rigid pair $(M,Q)$, let $\ind_{{(M,Q)}}\staurigidpair({\Lambda})$ denote set of $\tau$-rigid pairs in $\fgMod {\Lambda}$ of the form $(M,Q)\oplus (N,R)$, where $(N,R)$ is indecomposable and not a direct summand of $(M,Q)$.}
	Buan--Marsh first address the cases where $(M,Q)$ is either of the form $(M,0)$ or $(0,Q)$. In each of the {five} cases below, they define a map 
	\begin{equation*}
	\mathcal{E}_{(M,Q)}\colon \ind_{{(M,Q)}}\staurigidpair({\Lambda}) \to  \ind \staurigidpair(J(M,Q)),
	\end{equation*}
	{where the codomain is the set of indecomposable $\tau$-rigid pairs in $J(M,Q)$.}
	This is extended to a {bijection} in the obvious way. {We will then consider the general case, where $(M,Q)$ is arbitrary, which will yield the bijection in \Cref{thm:BM18w.3.6}.}
	
	For a $\tau$-rigid ${\Lambda}$-module ${N}$, we denote by ${{\mathbb{P}}}_{{N}}$ its minimal projective presentation, considered as a two-term presilting object in $\Kb(\proj {\Lambda})$. This is to say that one sets ${{\mathbb{P}}}_{{N}}\defeq H_{{\Lambda}}^{-1}({N},0)$, where $H_{{\Lambda}}$ is as defined in \Cref{IJY14.4.6}.
	
	\textbf{Case I:} Suppose that $Q=0$. {We define $\mathcal{E}_{(M,0)}(M,0) \defeq (0,0)$.}
	
	\textbf{Case I(a):} If ${N}$ is an indecomposable ${\Lambda}$-module such that $M\oplus {N}$ is $\tau$-rigid and ${N}\not\in\gen(M)$, define $\mathcal{E}_{(M,0)}({M\oplus}{N},0) \defeq (f_M({N}),0)$, where $$f_M\colon \fgMod {\Lambda}\to M^{\perp}$$ is the torsion-free functor for the torsion pair $(\gen(M),M^{\perp})$ {(see \Cref{thm:AIR14.2.7})}, that is, the natural functor $\fgMod {\Lambda}\to M^{\perp}$ {(cf. the paragraph surrounding \eqref{eq:torsionpair_seq})}.
	
	\textbf{Case I(b):} If ${N}$ is an indecomposable ${\Lambda}$-module {outside the subcategory $\add(M)$} such that $M\oplus {N}$ is $\tau$-rigid and ${N}\in\gen(M)$, define $\mathcal{E}_{(M,0)}({M\oplus}{N},0) \defeq \big(0,f_M{H^0 {\mathbb{P}}_{{N}}\langle -1\rangle}\big)$, where 
	\begin{equation*}
	\begin{tikzcd}
	{\Sigma^{-1}{ {\mathbb{P}}}_{{N}}} \arrow[r] & {{\mathbb{P}}_{{N}}\langle -1\rangle} \arrow[r] & ({{ {\mathbb{P}}}_M})_{{N}}\arrow[r,"\beta_{{{\mathbb{P}}}_{{N}}}"] & {{\mathbb{P}}}_{{N}} 
	\end{tikzcd}
	\end{equation*}
	is a distinguished triangle {in $\Kb(\proj \Lambda)$} and $\beta_{{{\mathbb{P}}}_{{N}}}$ is a minimal right $\add({{\mathbb{P}}}_M)$-approximation.
	
	\textbf{Case I(c):} If $R$ is an indecomposable projective {$\Lambda$-module} such that $\Hom_{{\Lambda}}(R,M)=0$, define 
	\begin{equation*}
	\mathcal{E}_{(M,0)}({M},R) \defeq \big(0,f_M{{H^0\big((\Sigma R)}\langle -1\rangle\big)}\big),
	\end{equation*}
	where 
	\begin{equation*}
	\begin{tikzcd}
	{R} \arrow[r] & {({\Sigma R})\langle -1\rangle} \arrow[r] & ({{\mathbb{P}}}_M)_{\Sigma R}\arrow[r,"\beta_{\Sigma R}"] & \Sigma R 	\end{tikzcd}
	\end{equation*}
	is a distinguished triangle {in $\Kb(\proj \Lambda)$} and $\beta_{\Sigma{R}}$ is a minimal right $\add({{\mathbb{P}}}_M)$-approximation. 
	
	\textbf{Case II:} Suppose that $M=0$. {We define $\mathcal{E}_{(0,Q)}(0,Q) \defeq (0,0)$.}

	\textbf{Case II(a):} If ${{N}}$ is an indecomposable $\tau$-rigid {$\Lambda$}-module such that $\Hom_{{\Lambda}}(Q,{{N}})=0$, define
	\begin{equation*}
	\mathcal{E}_{(0,Q)}({{N}},{Q})\defeq ({{N}},0).
	\end{equation*}
	
	\textbf{Case II(b):} If $R$ is an indecomposable projective {$\Lambda$}-module such that {$(0,R)\in \ind_{(0,Q)}\staurigidpair({\Lambda})$}, define $\mathcal{E}_{(0,Q)}(0,{Q\oplus}R)\defeq (0,f_Q(R))$.
	
	\textbf{In general:} {The general case, where neither $M$ nor $Q$ need be zero, requires some preparation. We will conclude this part in \eqref{eq:ingeneral} below.}
	
	Let $(M,Q)$ be a $\tau$-rigid pair in $\fgMod {\Lambda}$. Let $(M^{{b}},0)$ be the Bongartz completion of $(M,0)$ {and {let} $(M^{+},{Q})$ be the Bongartz completion of $(M,Q)$}. {Let ${\Gamma'=\End_{{\Lambda}}(M^b)/[M]}$ be the $\tau$-tilting reduction of ${\Lambda}$ with respect to $(M,0)$ {and ${\Gamma=\End_{{\Lambda}}(M^+)/[M]}$ the $\tau$-tilting reduction of ${\Lambda}$ with respect to $(M,Q)$}}. 
	{Case I above gives a bijection}
	\begin{equation*}
		{\mathcal{E}_{(M,0)}\colon \staurigidpair_{(M,0)}({\Lambda}) \to \staurigidpair J(M,0).}
	\end{equation*}
	{It sends the pair $({M},Q)$ to $(0,\widetilde{Q})$, where $\widetilde{Q}=f_M H^0\big((\Sigma Q)\langle -1\rangle\big)$ (see Case I(c)), whence we can restrict to a bijection}
	\begin{equation*}
		{\mathcal{E}_{(M,0)}\colon \staurigidpair_{(M,Q)}(\Lambda) \to \staurigidpair_{(0,\widetilde{Q})} J(M,0).}
	\end{equation*}
	By \Cref{DIRRT17.4.12}, we have an exact equivalence
	\begin{equation}\label{eq:mbigbih}
	F_{(M,0)}=\Hom_{{\Lambda}}(M^{{b}},-)\colon J({M},0)\to \fgMod {\Gamma'}.
		\end{equation}
	{It induces a bijection {(see \eqref{eq:FMQbij})}
	\begin{equation}\label{eq:bigbih}
	F_{(M,0)}{=\Hom_{\Lambda}(M^b,-)}\colon\staurigidpair(J(M,0))\to \staurigidpair({\Gamma'}),
	\end{equation}
	which restricts to a bijection}
	\begin{equation*}
	F_{(M,0)}\colon\staurigidpair_{(0,{\widetilde{Q}})}(J(M,0))\to \staurigidpair_{(0,Q')}({\Gamma'}),
	\end{equation*}
	where $Q'$ is the projective ${\Gamma'}$-module $\Hom_{{\Lambda}}(M^b,{\widetilde{Q}})$.
	{Applying} Case II above {to the $k$-algebra ${\Gamma'}$} gives a bijection
	\begin{equation}\label{eq:EGQ}
	\mathcal{E}^{{{\Gamma}'}}_{(0,Q')}\colon\staurigidpair_{(0,Q')}({\Gamma'})\to \staurigidpair J_{{\Gamma'}}(0,Q'),
	\end{equation}
	where $J_{{\Gamma'}}(0,Q')$ is the $\tau$-perpendicular category of $(0,Q')$ in $\fgMod {{\Gamma'}}$. 
	\Cref{DIRRT17.4.12} gives {an exact equivalence
	\begin{equation}\label{eq:bDD}
		F^{{\Gamma'}}_{(0,Q')}\colon J_{{\Gamma'}}(0,Q') \to \fgMod{\Delta},
	\end{equation}}
{where $\Delta$ is the $k$-algebra $\End_{\Gamma'}(B)$, and $(B,Q')$ is the Bongartz completion of $(0,Q')$ {in $\fgMod \Gamma'$}.
	{We thus have} a bijection
	\begin{equation}\label{eq:DD}
		F^{{\Gamma'}}_{(0,Q')}\colon \staurigidpair J_{{\Gamma'}}(0,Q') \to \staurigidpair({\Delta}).
	\end{equation}
	
	{It will be useful to know that the $k$-algebras $\Gamma$ and $\Delta$ are Morita equivalent. This is essentially shown by Buan--Hanson \cite[Theorem 6.4]{BH21} (see also Buan--Marsh \cite[Theorem 4.3]{BM18w}), who provide the equality of $\tau$-perpendicular subcategories
	\begin{equation*}
		J(M,Q) = J_{J(M,0)}\big( \mathcal{E}_{(M,0)}(M,Q)\big).
	\end{equation*}
	One now uses \Cref{DIRRT17.4.12} to construct an exact equivalence}
	\begin{equation}\label{eq:MoritaGD}
	\begin{tikzcd}[column sep=3em]
		\fgMod \Gamma \arrow[r,"{F_{(M,Q)}^{-1}}"] & J(M,Q) = J_{J(M,0)} \big( \mathcal{E}_{(M,0)}(M,Q) \big) = J_{J(M,0)}(0,\widetilde{Q})
		\arrow[r,"F_{(M,0)}","\eqref{eq:mbigbih}"'] & J_{\Gamma'}(0,Q') \arrow[r,"{F^{\Gamma'}_{(0,Q')}}","\eqref{eq:bDD}"'] & \fgMod \Delta,
	\end{tikzcd}
	\end{equation}
	{as claimed.}
	
	We {now} define $$\mathcal{E}_{(M,Q)}\defeq F^{-1}_{(M,Q)}\circ F^{{\Gamma'}}_{(0,Q')} \circ \mathcal{E}^{{\Gamma'}}_{(0,Q')}\circ {F_{(M,0)} \circ \mathcal{E}_{(M,0)}},$$ {as displayed on the {right} in \eqref{eq:ingeneral} below.}
	\begin{equation}\label{eq:ingeneral}
	\begin{tikzcd}
                                                                            & {\staurigidpair_{(M,Q)}({\Lambda})} \arrow[ld, "{\mathcal{E}_{(M,0)}}","{\mathrm{Case\, I}}"'] \arrow[dddddd, "{\mathcal{E}_{(M,Q)}}"] \\
{\staurigidpair_{(0,{\widetilde{Q}})} J(M,0)} \arrow[d, "{F_{(M,0)}}","{\eqref{eq:bigbih}}"'] &                                                                                                         \\
{\staurigidpair_{(0,Q')}({{\Gamma'}})} \arrow[d, "\mathcal{E}^{{\Gamma'}}_{(0,Q')}","{\eqref{eq:EGQ}}"']    &                                                                                                         \\
{\staurigidpair J_{{\Gamma'}}(0,Q')} \arrow[d, "{F^{{\Gamma'}}_{(0,Q')}}","{\eqref{eq:DD}}"']                  &                                                                                                         \\
{ \staurigidpair({\Delta})}         \arrow[d,"{\cong}","{\eqref{eq:MoritaGD}}"']            &                                                                                                         \\
                        \staurigidpair({\Gamma}) \arrow[rd, "{F^{-1}_{(M,Q)}}","{\text{\Cref{DIRRT17.4.12}}}"']                                                        &  \\
                                                                            & {\staurigidpair J(M,Q)}                                                                                
\end{tikzcd}
	\end{equation}
	
	{This concludes our recollection of Buan--Marsh's bijection $\mathcal{E}_{(M,Q)}$.}
	Our aim is to link $\mathcal{E}_{(M,Q)}$ to Iyama--Yang silting reduction. This is achievable {in the framework of \Cref{setup:siltingtstr}. Some of our lemmas apply more generally, so we fix a more general setup where we also fix notation to be used in this section}.
	
	\begin{setup}\label{setup:comp}
		A Hom-finite Krull--Schmidt triangulated category $\c$ is fixed, as well as a silting object $S\in\c$. Let ${P}$ be a {basic} $2_{S}$-term presilting object {in $\c$}, let ${\Lambda}$ denote the endomorphism algebra $\End_{\c}({ S})$, {and let $\Gamma=\End_{\z_P/{[P]}}(T^+_P)$, where $T^{+}_P$ is the Bongartz completion of $P$.}
	\end{setup}
	
	{Note that if $\c$ is an algebraic triangulated category, then \Cref{setup:comp} reduces to \Cref{setup:siltingtstr}. Indeed, there then exists a dg $k$-algebra $A$ and a triangle equivalence 
	\begin{equation}\label{eq:PhiS}
	\Phi_S \colon \c \to \per(A),
	\end{equation}
	sending $S$ to $A$ \cite[Theorem 3.8(b)]{Kel06}. The Hom-finiteness of $\c$ and the isomorphisms $${{\c}(S,\Sigma^i S)} \to^{\Phi_{S}} \Hom_{\per(A)}(A,\Sigma^iA)\simeq H^iA,$$ 
	{where $i\in \Z$ is arbitrary,} guarantee that $A$ satisfies the conditions imposed in \Cref{setup:siltingtstr}. {In particular, the dg $k$-algebra $A$ is non-positive because $S$ is silting. We also note that s}etting $i=0$ yields a $k$-algebra isomorphism $H^0A \to^{\simeq} \Lambda$.} {The triangle equivalence $\Phi_S$ in \eqref{eq:PhiS} induces a bijection from $\npresilt{2}{S}(\c)$ to $\npresilt{2}{}(A)$, which sends the object $P$ in \Cref{setup:comp} to an object in $\per(A)$ having the properties imposed on the object $P$ in \Cref{setup:siltingtstr}.}

	{In order to establish} the compatibility of the two reduction techniques described above, we prove {two} important lemmas.
	
	\begin{lemma}\label{lem:addquotrad}
		Let $\c$, ${ S}$, $P$, and ${\Lambda}$ be as in \Cref{setup:comp}. {We fix an object ${ X}\in \add({ S})*\Sigma\add({ S})$.}
		\begin{enumerate}
			\item\label{lem:addquotrad1} {W}e have a natural isomorphism
		${\c\over [{ P}]}({ S},{ X}) \cong f_{\c({ S},{ P})}\c({ S},{ X})$, as ${\Lambda}$-modules, where $f_{\c({ S},{ P})}\colon \fgMod {\Lambda}\to \c({ S},{ P})^{\perp}$ is the torsion-free functor for {the torsion pair} $(\gen(\c({ S},{ P})),\c({ S},{ P})^{\perp})$.
		\item\label{lem:addquotrad2} {Suppose $\c$ is algebraic and let $X\in\c$ be such that $P\oplus X$ is basic and $2_S$-term presilting. Assume that $\c(S,X)\in\gen(\c(S,P))$, or equivalently that there exists a triangle in $\c$}
\begin{equation}\label{eq:Ctri}
	\begin{tikzcd}
	\Sigma^{-1} X \arrow[r] & X\langle -1 \rangle \arrow[r] & P_X \arrow[r,"\beta_X"] &  X,
	\end{tikzcd}
\end{equation}
{where $\beta_X$ is a right $\add(P)$-approximation and $ X\langle -1\rangle \in \add(S) * \Sigma \add(S)$ (adapting Buan--Marsh \cite[Lemma 2.5(a)]{BM18t}). For $Y\in \add(S) * \Sigma \add(S)$, let $FY$ denote $\mathbb{P}_{\c(S,Y)}\oplus \Sigma\c(S,Y_1) \in \Kb(\proj\Lambda)$, where $Y_1$ is the largest direct summand of $Y$ in $\Sigma\add(S)$.
We then have a well-defined isomorphism of $\Lambda$-modules 
$$f_{\c(S,P)}\c(S,X\langle -1\rangle) \simeq f_{\c(S,P)} H^0\big( (FX)\langle -1 \rangle \big),$$ where $(FX)\langle -1 \rangle $ is defined by the following triangle in $\Kb(\proj\Lambda)$}
\begin{equation}\label{eq:Ktri}
	\begin{tikzcd}
	\Sigma^{-1} FX  \arrow[r] & (FX) \langle -1 \rangle \arrow[r] &(FP)_{FX} \arrow[r,"\beta"] &  FX,
	\end{tikzcd}
\end{equation}
{and $\beta$ is a right $\add(FP)$-approximation.}
		\end{enumerate}
	\end{lemma}
	\begin{proof}
		{First, we prove \eqref{lem:addquotrad1}.}
		We have an exact sequence 
		\begin{equation}\label{eq:addquotradseq}
		\begin{tikzcd}
		{[{ P}]}({ S},{ X})\arrow[r,tail,"i"]&\c({ S},{ X})\arrow[r,two heads] & {\c\over [{ P}]}({ S},{ X})
		\end{tikzcd}
		\end{equation}
		of ${\Lambda}$-modules, where ${[{ P}]}$ is the ideal {of $\c$ consisting of} morphisms factoring through $\add(P)$.
		It suffices to show that ${[{ P}]}({ S},{ X})\in\gen(\c({ S},{ P}))$ and that ${\c\over [{ P}]}({ S},{ X})\in\c({ S},{ P})^{\perp}$. 
		
		Let ${ P}'\to^{\beta_{ X}}{ X}$ be a right $\add({ P})$-approximation of ${ X}$. {It follows from the definition of right $\add({ P})$-approximations that} the map
		\begin{center}
			\begin{tikzcd}
				\c({ S},{ P'})\arrow[r,"{\beta_{ X}\circ -}"] & {[{ P}]}({ S},{ X})
			\end{tikzcd}
		\end{center}
		surjects, whence $[{ P}]({ S},{ X})\in\gen(\c({ S},{ P}))$.
		
		We now show that $\Hom_{{\Lambda}}(\c({ S},{ P}),{\c\over [{ P}]}({ S},{ X}))=0$. This $k$-vector space appears as the third term in the following long exact sequence:
		\begin{equation*}
		\begin{tikzcd}
		0\arrow[r]&\Hom_{{\Lambda}}(\c({ S},{ P}),{[{ P}]}({ S},{ X})) \arrow[r,"i\circ -"]& \Hom_{{\Lambda}}(\c({ S},{ P}),\c({ S},{ X}))\arrow[r]& \Hom_{{\Lambda}}(\c({ S},{ P}),{\c\over [{ P}]}({ S},{ X}))\arrow[lld] \\ & \Ext^1_{{\Lambda}}(\c({ S},{ P}),{[{ P}]}({ S},{ X})). & &
		\end{tikzcd}
		\end{equation*}
		Having already shown that ${[{ P}]}({ S},{ X})\in \gen(\c({ S},{ P}))$, a result of Auslander--Smalø \cite[Proposition~5.8]{AS81} asserts that the $\tau$-rigidity of $\c({ S},{ P})$ {implies} the vanishing of $\Ext^1_{{\Lambda
	}}(\c({ S},{ P}),{[{ P}]}({ S},{ X}))$.
		Thus it suffices to show that $i\circ -$ surjects. 
		We will conclude the proof by showing that 
		\begin{equation*}
		\Hom_{{\Lambda}}(\c({ S},{ P}),\c({ S},{ X}))=\Hom_{{\Lambda}}(\c({ S},{ P}),{[{ P}]}({ S},{ X})),
		\end{equation*}
		which would make $i\circ -$ an injective endomorphism of a finite dimensional $k$-vector space, and thus surjective.
		It follows from \Cref{IY08.6.2} that all ${\Lambda}$-homomorphisms from $\c({ S},{ P})$ to $\c({ S},{ X})$ are determined by an equivalence class of morphisms ${ P}\to { X}$. Consequently, all homomorphisms in $\Hom_{{\Lambda}}(\c({ S},{ P}),\c({ S},{ X}))$ have image in ${[{ P}]}({ S},{ X})$, and $i\circ -$ surjects. {This concludes the proof of \eqref{lem:addquotrad1}.}
		
	{Secondly and finally, we prove \eqref{lem:addquotrad2}.
Let $F$ denote the composite triangle functor}
\begin{equation*}
\begin{tikzcd}[column sep=4em]
	\c \arrow[r,"\Phi_S"] &  \per(A) \arrow[r,"{-\otimes^{\mathbf{L}}_A H^0A}"] & \per(H^0A) \arrow[r,"{\Psi}"] & \Kb(\proj \Lambda),
\end{tikzcd}
\end{equation*}
{where $\Phi_S$ is the triangle equivalence in $\eqref{eq:PhiS}$, the middle functor $-\otimes^{\mathbf{L}}_A H^0A$ denotes the induction functor, and $\Psi$ is induced by the $k$-algebra isomorphism $H^0A\to^{\simeq}\Lambda$.
It is clear that the functor $\c(S,-)\colon \c \to \fgMod \Lambda$ factors as follows:}
\begin{equation}\label{eq:factorr}
\begin{tikzcd}[column sep=10em]
	\c \arrow[r,"F"]\arrow[rr,bend right=15,"{\c(S,-)}"'] & \Kb(\proj \Lambda) \arrow[r,"{H^0=\Hom_{\Kb(\proj\Lambda)}(\Lambda,-)}"] & \fgMod \Lambda.
	\end{tikzcd}
\end{equation}
{The induction functor induces an additive equivalence $\add(A) * \Sigma \add(A) \to \add(H^0A) * \Sigma \add(H^0A)$, whence the functor $F$ induces an additive equivalence}
\begin{equation}\label{eq:BY}
\begin{tikzcd}[column sep=2em]
{\add(S) * \Sigma \add(S)} \arrow[r,"F"] & {\add(\Lambda) * \Sigma \add(\Lambda)}.
\end{tikzcd}
\end{equation}
{In particular, the functor $F$ sends right $\add(P)$-approximations to right $\add(FP)$-approximations. Moreover, the functor $F$ sends a object $Y$ in $\add(S) * \Sigma \add(S)$ to $FY = \mathbb{P}_{\c(S,Y)}\oplus \Sigma \c(S,Y_1)$ in $\Kb(\proj \Lambda)$, as defined above.
We may thus assume that $F$ sends the triangle in \eqref{eq:Ctri} to the triangle in \eqref{eq:Ktri}, and in turn that $(FX)\langle -1 \rangle $ is two-term. Using the factorisation in \eqref{eq:factorr}, one finds $\Lambda$-isomorphisms $$\c(S,X\langle -1 \rangle) \simeq H^0 \big((FX) \langle -1 \rangle \big) \simeq \Hom_{\Kb(\proj\Lambda)}(\Lambda,(FX)\langle -1 \rangle).$$
As the objects $S$, $X\langle -1\rangle$, and $P$ are all $2_S$-term, we may use the equivalence in \eqref{eq:BY} to induce an isomorphism
$${\c\over[P]}(S,X\langle -1 \rangle) \simeq \Hom_{{\Kb(\proj\Lambda) / [{FP}]}}(\Lambda,(FX)\langle -1 \rangle).$$
By {applying} \eqref{lem:addquotrad1} above {$\Kb(\proj\Lambda)$ and $FX$}, we conclude that there is a well-defined $\Lambda$-isomorphism 
$$f_{\c(S,P)}\c(S,X\langle -1 \rangle) {\simeq f_{H^0(FX)}H^0(FX\langle -1\rangle)} \simeq f_{\c(S,P)} H^0 \big((FX) \langle -1 \rangle \big),$$ 
as desired.}
	\end{proof}
		
	\begin{lemma}\label{lem:lowertri}
		{Let} $\c$, ${ S}$, ${ P}$, ${\Lambda}$, {and ${\Gamma}$ be} as in \Cref{setup:comp}. {Let $(M,Q)=H_{ S}({ P})$ (where $H_{ S}$ is as defined in \Cref{IJY14.4.6}). Consider the equivalence $H'_S$ defined as the composite}
		 \begin{equation*}
		 \begin{tikzcd}[column sep=5em]
		 	\Big({\add({T^{+}_P})*\add({ T^{+}_P}) {\langle 1 \rangle}\Big) / [{ T^{+}_P}\langle 1 \rangle]}\arrow[d,symbol=\subseteq] \arrow[r,"{{\z_{P}\over [P]}(T^+_{P},-)}"] & \fgMod {\Gamma} \arrow[r, "F_{H_{ S}({ P})}^{-1}"] & J(H_{ S}({ P})), \\
			\z_P/{[P]} &&
		 \end{tikzcd}
		  \end{equation*}
		 {(see \Cref{IY08.6.2} and \Cref{DIRRT17.4.12}.)}, {where $\langle 1\rangle$ is defined {in} \eqref{eq:rangle_tri}.}
		  {Then the resulting composite bijection} 
		  \begin{equation*}
		 \begin{tikzcd}[column sep=3em]
		 	\npresilt{2}{T^{+}_P}(\z_P/{[P]}) \arrow[r,"{H_{ T^{+}_P}}"] & \staurigidpair {\Gamma} \arrow[r, "F_{H_{ S}({ P})}^{-1}"] & \staurigidpair J(H_{ S}({ P}))
		 \end{tikzcd}
		  \end{equation*}
		{(see \Cref{IJY14.4.6} and \eqref{eq:FMQbij}) can be expressed as follows:}
		\begin{equation}\label{eq:lowertriform}
		\begin{tikzcd}[row sep=0.75em,column sep=2em]
		{\npresilt{2}{ T^{+}_P}(\z_{ P}/[{ P}])} \arrow[r, "H'_{ S}"] & \staurigidpair J(H_{ S}({ P}))                                                                                                            \\
Y \arrow[r] \arrow[u,symbol=\in]                                         & {{\left(f_{\c({S},{P})}\c({ S},{Y}),f_{\c({ S},{ P})} \c({ S},{ Y}_1\langle -1\rangle)\right)}.} \arrow[u,symbol=\in]
\end{tikzcd}
		\end{equation}
		where ${ Y}_1$ is the largest direct summand of ${ Y}$ in ${ T^{+}_P}\langle 1\rangle$. 	
		\end{lemma}
	\begin{proof}
		{Fixing $Y\in \npresilt{2}{T^{+}_P}(\z_P/{[P]})$, we first show that the ${\Gamma}$-module $F_{H_S{(P)}}f_{\c({ S},{ P})}\c({ S},{ Y})$ is isomorphic to ${{\z_P} \over [P]}(T^{+}_P,Y)$, yielding the claimed re-expression of the first component. In order to use the definition of $F_{H_S{(P)}}$ provided in \Cref{DIRRT17.4.12}, we note that if $(M^+,Q)$ is the Bongartz completion of $H_S{(P)}$, then $M^+\simeq \c(S,T^+_P)$.
		We get
		\begin{equation*}
		F_{H_S{(P)}} f_{\c({ S},{ P})}\c({ S},{ Y}) =\Hom_{{\Lambda}} \Big(\c(S,T^+_P) ,  f_{\c({ S},{ P})}\c({ S},{ Y})) \Big).
		\end{equation*}
		A result of Jasso \cite[Proposition~4.15]{Jas15} proves that the right hand side is isomorphic to ${{\z_P}  \over [P]}(T^+_P, { Y})$. We conclude that the first component can be expressed as claimed.}
		
		It follows from the {definition of $\langle -1\rangle$ in \eqref{eq:langle_tri}} that $Y_1\langle -1\rangle \in \npresilt{2}{T^+_P}(\z_P/[P])$. One obtains an expression for the second component by repeating the reasoning in the previous paragraph, replacing $Y$ with $Y_1\langle -1\rangle$. Thus, the second component also is as claimed, and the proof is complete. 
		\end{proof} 
	
	We now have all ingredients to prove the main theorem of this section.

	\begin{theorem}\label{th:Jas15.4.12}
	{Let $A$ and $P$ be as in \Cref{setup:siltingtstr}. Setting $\c=\per(A)$ and $S=A$, let $\Gamma$ be as in \Cref{setup:comp}.}
		We have a commutative diagram of bijections
		\begin{center}
			\begin{tikzcd}[column sep=4em]
				\npresilt{2}{}_{{ P}}({A}) \arrow[r,"H_{{{A}}}"]\arrow[d,"\varphi_{ P}"] & \staurigidpair_{H_{{A}}({ P})}({{H^0A}})\arrow[d,"\psi_{H_{{A}}({ P})}"]\arrow[dd,bend left=89,"\mathcal{E}_{H_{{A}}({ P})}"] \\
				\npresilt{2}{ T^{+}_P}(\z_{ P}/[{ P}]) \arrow[r,"H_{ T^{+}_P}"]\arrow[rd,"H'_{{A}}"'] & \staurigidpair({\Gamma}) \\
				& \staurigidpair J(H_{{A}}({ P}))\arrow[u,"F_{{H_{{A}}}({ P})}"']
			\end{tikzcd}
		\end{center}
		where $\psi_{H_{{A}}({ P})} \defeq F_{{H_{{A}}}({ P})} \circ \mathcal{E}_{H_{{A}}({ P})}$.
	\end{theorem}
	\begin{proof}
	{Throughout the proof, we simplify the notation by setting $\c=\per(A)$. This gives a more compact notation for the Hom-{bifunctor} in $\per(A)$, {which now can be denoted $\c(-,-)$}.}
		
		The commutativity of the lower triangle was shown in \Cref{lem:lowertri}, and the right triangle (with $\psi_{H_{{A}}({ P})}$ along the diagonal) commutes by definition. What remains is proving that
		\begin{equation}\label{eq:Jas15.4.12}
		H'_{ {A}}\circ \varphi_{ P}({ X})= \mathcal{E}_{H_{ {A}}({ P})}\circ H_{ {A}}({ X}),
		\end{equation}
		which should hold for all ${ X}\in\npresilt{2}{}_{{ P}}({A})$. All maps in question are constructed to distribute over direct sums, whence it suffices to consider the case where ${X}$ {is of the form $P\oplus X'$ for some indecomposable object $X'$ which is not in $\add(P)$.} We treat {five} cases {and a general case}, corresponding to the definition of $\mathcal{E}_{H_{{A}}({ P})}$ we reviewed in the discussion following \Cref{thm:BM18w.3.6}. {We note that $\varphi_P(X) \simeq X'$ in $\z_P/[P]$.}
		
		\textbf{Case I:} Suppose that ${ P}$ has no {non-zero} direct summands in {$\Sigma { {A}}$}.
		
		\textbf{Case I(a):} Suppose that ${ X'}$ has no {non-zero} direct summands in $\Sigma{{A}}$ and {that} $\c({{A}},{ X'})\not\in\gen(\c({ {A}},{ P}))$.		
		Using the definitions of $H_{{A}}$ and $\mathcal{E}_{H_{ {A}}({ P})}$, the right hand side of \eqref{eq:Jas15.4.12} becomes
		\begin{equation*}
			\mathcal{E}_{H_{ {A}}({ P})}\circ H_{ {A}}({ X}) = \mathcal{E}_{H_{ {A}}({ P})}(\c({A},X),0) = \big({f_{\c({A},P)}}\c({A},X),0\big) { = \big({f_{\c({A},P)}\c({A},P\oplus X{'}),0\big) = \big({f_{\c({A},P)}}\c({A},X{'}),0\big)}}.
			\end{equation*}
		{Using \Cref{lem:lowertri} and that $\varphi_P(X) \simeq X'$ in $\z_P/[P]$, we re-express the left hand side of \eqref{eq:Jas15.4.12} as follows:}
		\begin{equation}\label{eq:3.5.1}
			{H'_{{A}}\circ \varphi_{ P}({ X}) = \big({f_{\c({A},P)}}\c({A}, X'),{f_{\c({A},P)}}\c({A}, X'_1\langle -1\rangle)\big)},
		\end{equation}
		where $X'_1$ is the largest {common} direct summand of $X'$ in $T_P^+\langle 1\rangle$, considered as object{s} in $\z_P/[P]$. 
		It is to be shown that the second component on the right hand side of \eqref{eq:3.5.1} vanishes. {Since $X'$ is indecomposable and $\c({A},X')\not\in\gen(\c({A},P))$, it suffices to show that $\c({A},T_P^+\langle 1\rangle)\in\gen(\c({A},P))$, which necessitates that $X'_1$ is a zero object in $\z_P/[P]$. Applying the functor $\c({A},-)$ to the triangle defining the object $T_P^+\langle 1\rangle$, we induce a long-exact sequence of ${H^0A}$-modules}
		\begin{equation*}
		\begin{tikzcd}
		\c({A},T_P^+) \arrow[r,"\alpha\circ -"] & \c({A},P') \arrow[r] &  \c({A},T_P^+\langle 1\rangle)  \arrow[r] & \c({A},{\Sigma T_P^+}),
		\end{tikzcd}
		\end{equation*}
	{where $T^+_P \to^{\alpha} P'$ is a left $\add(P)$-approximation. The last term vanishes, proving that $\c({A},T_P^+\langle 1\rangle)\in \gen(\c({A},P))$, and in turn that \eqref{eq:Jas15.4.12} holds in this case. Note that the definition of $T_P^+\langle 1\rangle$, as an object in $\c$, depends on the choice of $\alpha$, but our claim that $\c({A},T_P^+\langle 1\rangle)\in\gen(\c({A},P))$ does not.}

		When addressing cases {I(b),} I(c), and II(b), we will use the {equality obtained in \eqref{eq:helps} below}. {In these three cases, we {have} that $\c({{A}},{ X})\in \gen(\c({{A}},{ P}))$, or equivalently that $f_{\c({{A}},{ P})}\c({{A}},{ X})=0$. {It follows that} $f_{\c({{A}},{ P})}\c({{A}},{X'})=0$, and since $\varphi_P(X)\simeq X'$ in $\z_P/[P]$, it follows from \Cref{lem:lowertri} that} 
		\begin{equation*}
		{H'_{{A}}\circ\varphi_{ P}({ X})  = {\left(f_{\c({{A}},{P})}\c({ {A}},{X'}),f_{\c({ {A}},{ P})} \c({ {A}},X'_1\langle -1\rangle)\right)} = (0,f_{\c({ {A}},{ P})}\c({ {A}},{X'_1\langle -1\rangle})),}
		 \end{equation*}
		{where, again, $X'_1$ is the largest common direct summand of $X'$ and $T_P^+\langle 1\rangle$, considered as objects in $\z_P/[P]$. Since $X'$ is indecomposable, we have that $X'_1$ is either zero or $X'$. By \Cref{lem:lowertri}, the map $H'_{{A}}$ is equal to a composite of two maps preserving indecomposability (of objects or pairs), but if $X'_1$ were a zero object, then $H'_{{A}}$ would send $\varphi_P(X)$ to zero. {Since this is a contradiction,} we thus have that $X'_1=X'$, and in turn that}
		 \begin{equation}\label{eq:helps}
		{H'_{ {A}}\circ\varphi_{ P}({ X}) = (0,f_{\c({ {A}},{ P})}\c({ {A}},{X'\langle -1\rangle})).}
		 \end{equation}
		
		{\textbf{Case I(b):}} Suppose that ${ X'}$ has no {non-zero} direct summands in $\Sigma{{A}}$, {and that} $\c({ {A}},{ X'})\in\gen(\c({ {A}},{ P}))$. {It follows from the definition of $\mathcal{E}_{H_{ {A}}(P)}$ that} $\mathcal{E}_{H_{ {A}}({{P}})}{\circ}H_{ {A}}({ X}) = (0,{f_{\c({ {A}},{ P})}H^0 \big( \mathbb{P}_{\c({A},X')}\langle -1\rangle}\big))$. Since $\c({ {A}},{ X})\in\gen(\c({ {A}},{ P}))$, we have that ${f_{\c({ {A}},{ P})}}\c({{A}},{ X})=0$. By {\eqref{eq:helps}} we have that $$H'_{ {A}}\circ\varphi_{ P}({ X}) = (0,f_{\c({ {A}},{ P})}\c({ {A}},{X'\langle -1\rangle})).$$ {We use \Cref{lem:addquotrad}\eqref{lem:addquotrad2} to conclude that \eqref{eq:Jas15.4.12} holds.}
		
		\textbf{Case I(c):} Suppose that ${X}={P\oplus}\Sigma {Q}$, where ${ Q}\in\ind\add({{A}})$. Then $$\mathcal{E}_{H_{{A}}({ P})} {\circ} H_{{A}}({ X})=(0,f_{\c({{A}},{ P})}{H^0\big(\Sigma {\c({A},Q)}\langle -1\rangle\big)}).$$ {On the other hand, we may use  {\eqref{eq:helps}} to show that $$ {H'_{{A}}\circ\varphi_{{P}}({ X})  = (0,f_{\c({{A}},{ P})}\c({{A}},{\Sigma Q\langle -1\rangle})).} $$
By \Cref{lem:addquotrad}\eqref{lem:addquotrad2} again, we have that  \eqref{eq:Jas15.4.12} holds.} 
		
		\textbf{Case II:} Suppose that ${ P}\in{\Sigma}\add({{A}})$. 
		
		\textbf{Case II(a):} If ${ X'}$ has no {non-zero} direct summands in {$\Sigma {{A}}$}, 
		{one uses \Cref{lem:lowertri} to assert that}
		\begin{equation*}
			H'_{{A}}\circ \varphi_P(X) =  \left(f_{\c({{A}},{ P})}\c({{A}},{X'}),f_{\c({{A}},{ P})}\c({ {A}},X'_1\langle -1\rangle)\right).
		\end{equation*}
		{There are no non-zero morphisms from ${A}$ to $P$, whence $f_{\c({A},P)}$ is the identity functor, and we may simplify the right hand side above to $\left(\c({{A}},{X'}),\c({ {A}},X'_1\langle -1\rangle)\right)$.} 
		
		{We show that the second component $\c({ A},X'_1\langle -1\rangle)$ vanishes.}
		{In the triangle defining $T^+_P\langle 1\rangle$:}
		\begin{equation*}
			\begin{tikzcd}
				T_P^{+} \arrow[r,"\alpha"] & P' \arrow[r] & T_P^{+} \langle 1\rangle \arrow[r] & \Sigma T_P^{+},
			\end{tikzcd}
		\end{equation*}
		{we find that $T_P^{+} \langle 1\rangle$ is an extension of objects $Y$ having the property that $\c(A,Y)=0$, whence $\c(A,T_P^{+} \langle 1\rangle)=0$. Since $X'$ does not satisfy this property, it is impossible for the indecomposable object $X'$ to be a direct summand of $T_P^{+} \langle 1\rangle$, proving that $X'_1$ is zero and that  $\c({ A},X'_1\langle -1\rangle)=0$, as sought.}
		We deduce that
		\begin{equation*}
			H'_{{A}}\circ \varphi_P(X){= \big(\c({{A}},{X'}),0\big)}.
		\end{equation*}
		{Since {$H_{ {A}}({ X}) =(\c({A},X'),\c({A},\Sigma^{-1}P))$ and $H_{{A}}(P)=(0,\c({A},\Sigma^{-1}P))$, the definition of $\mathcal{E}_{H_{{A}}({ P})}$ gives $\mathcal{E}_{H_{{A}}({ P})}H_{{A}}({ X}) = (\c({A},X'),0)$ in this case,} {showing} that \eqref{eq:Jas15.4.12} holds.}
				
		\textbf{Case II(b):} If ${ X'}=\Sigma { R}$, where ${R}\in\ind\add({{A}})$, then $\mathcal{E}_{H_{{A}}({ P})}{\circ} H_{{A}}({ X})=(0,f_{\c({{A}},{\Sigma^{-1}P})}{\c({A},R)}))$.
		{On the other hand, Equation {\eqref{eq:helps}} and the fact that $\c({A},P)=0$ give that 
		$$H'_{ {A}}\circ\varphi_{ P}({ X}) = (0,f_{\c({ {A}},{ P})}\c({ {A}},{(\Sigma R)\langle -1\rangle})) = (0,\c({ {A}},{(\Sigma R)\langle -1\rangle})).$$
		We prove directly that $f_{\c({ {A}},\Sigma^{-1}P)} \c({A},R)$ and $\c({ {A}},{(\Sigma R) \langle -1\rangle})$ are isomorphic. Applying the functor $\c({A},-)$ to the triangle}
		\begin{equation}\label{eq:triesles}
		R \to (\Sigma R)\langle -1\rangle \to P_{\Sigma R}  \to^{\beta_{\Sigma R}}  \Sigma R,
		\end{equation}
		 {where $\beta_{\Sigma R}$ is a right $\add(P)$-approximation, yields a long-exact sequence}
		\begin{equation}\label{eq:leslesles}
		\cdots \to^{j} \c({A},\Sigma^{-1}P_{\Sigma R} ) \to \c({A},R) \to \c({A},(\Sigma R)\langle -1\rangle) \to 0.
		\end{equation}
		{
		Consider the short-exact sequence}
		\begin{equation}\label{eq:shortlesles}
		0\to \mathrm{cok}(j) \to \c({A},R) \to \c({A},(\Sigma R)\langle -1\rangle) \to 0.
		\end{equation}
		{Since $(\gen(\c({A},\Sigma^{-1}P)),\c({A},\Sigma^{-1}P)^{\perp})$ is a torsion pair in $\fgMod {H^0A}$, it suffices to show that the first term in \eqref{eq:shortlesles} is in $\gen(\c({A},\Sigma^{-1}P))$ and that the final term is in $\c({A},\Sigma^{-1}P)^{\perp}$. Our claim about the first term holds by construction. Furthermore, one uses the triangle in \eqref{eq:triesles} and a long-exact sequence argument to show that $\c(\Sigma^{-1}P,(\Sigma R)\langle -1\rangle)=0$. It is now a consequence of \Cref{IY08.6.2} that $\Hom_{{H^0A}}({\c({A},\Sigma^{-1}P)},{\c({A},(\Sigma R)\langle -1\rangle)})=0$, putting the final term in \eqref{eq:shortlesles} in $\c({A},\Sigma^{-1}P)^{\perp}$. We conclude that \eqref{eq:Jas15.4.12} holds in Case II(b).}
		
		{\textbf{In general:} We can find a decomposition $P = P' \oplus \Sigma P''$, where $\Sigma P''$ is the largest direct summand of $P$ in $\Sigma {A}$. We set $(M,Q)=H_{{A}}(P),$ {$(M^+,Q) = H_{{A}}(T^+_P)$, and $(M^b,0)=H_{{A}}(T^+_{P'})$}. 
		The goal is to show that the outer square in the following diagram {of bijections} commutes:}
		\begin{equation*}
			\begin{tikzcd}[column sep=1.5em,row sep=5em,font=\tiny]
\npresilt{2}{}_P({A}) \arrow[rdd, "\varphi_{P'}"] \arrow[ddddd, "\varphi_{P}"] \arrow[rrr, "H_{{A}}"] \arrow[rrd, "\texttt{1}", phantom] \arrow[ddddd, "\texttt{2}", phantom, bend left] &                                                                                                                                                                                                                                    &                                                                                            & {\staurigidpair_{(M,Q)}({{H^0A}})} \arrow[ld, "{\mathcal{E}_{(M,0)}}"] \arrow[ddddd, "{\mathcal{E}_{(M,Q)}}"] \\
                                                                                                                                                                                       &                                                                                                                                                                                                                                    & {\staurigidpair_{(0,{\widetilde{Q}})} J(M,0)} \arrow[d, "{F_{(M,0)}}"] \arrow[rd, "\texttt{4}", phantom] &                                                                                                        \\
                                                                                                                                                                                       & {\npresilt{2}{T^+_{P'}}_{{\varphi_{P'}(P)}}(\z_{P'}/[P'])} \arrow[d, "\varphi_{{\varphi_{P'}(P)}}"] \arrow[r, "H_{T^+_{P'}}",] \arrow[ru, "H''_{{A}}"] \arrow[rd,phantom,"\texttt{5}"] \arrow[ru, "\texttt{3}", phantom, bend right=15] & {\staurigidpair_{(0,Q')}({\Gamma'})} \arrow[d, "{\mathcal{E}^{{\Gamma'}}_{(0,Q')}}"]                    & {}                                                                                                     \\
                                                                                                                                                                                       &      {\tiny \npresilt{2}{T^+_{\varphi_{P'}({P})}}(\z_{\varphi_{P'}({P})}/[\varphi_{P'}({P})])}   \arrow[r, "H'_{T^+_{{P'}}}"] \arrow[rd,"H_{T_{\varphi_{P'}({P})}^+}" description]  \arrow[rd, phantom,"{\texttt{7}}",bend left=25]                                                                                                                                                                                                                              & {\staurigidpair J_{{\Gamma'}}(0,Q')} \arrow[d, "{F^{{\Gamma'}}_{(0,Q')}}"]                              &                                                                                                        \\
 {}                                                                                                                                                                                          &                                                                                                                                                                                                                                    & \staurigidpair({\Gamma}) \arrow[rd, "{F_{(M,Q)}^{-1}}"]    \arrow[d,phantom,"{\texttt{8}}"]                                      &                                                                                                        \\
{\npresilt{2}{T^+_P}(\z_P/[P])}  \arrow[rrr, "H'_{{T^+_P}}"'] \arrow[rru, "H_{T^+_P}", near end]  \arrow[ruu,"a"]         \arrow[ruu,phantom,"{\texttt{6}}", bend right=35, near end]                                     &          {}                                                                                                                                                                                                                         &                                                       {}                                     & {\staurigidpair J(M,Q)}                                                                               
\end{tikzcd}
\end{equation*}
{There are {eight} subdiagrams above, labelled $\texttt{1}$ to ${\texttt{8}}$. We complete the proof by showing that all these subdiagrams commute. {Some of the sets and maps in this diagram will be defined below.}}
\begin{itemize}
	\item[\texttt{1}:] The upper square commutes, since one can apply Case I above to the {two}-term presilting object $P'$ in {$\per(A)$}. The map $H''_{{A}}$ above is defined by $H'_{{A}}$ in \Cref{lem:lowertri}, but replacing $(M,Q)$ with $(M,0)$. {The $\tau$-rigid pair $(0,\widetilde{Q})$ {in $J(M,0)$} is given by $\mathcal{E}_{{(}M,0{)}}({M},Q) {= {\big(0,}f_M(H^0\big((\Sigma Q)\langle -1\rangle \big))}{\big)}$ (see Case I(c)).	}
	\item[\texttt{2}:] The left {square} commutes as a result of \Cref{lem:functorphi} and its proof. {We comment on the definition of the map $a$. Since the essential image of the composite functor 
	$$\z_P \subseteq  \z_{P'} \to  \z_{P'}/[P'] $$
	can be identified with $\z_{\varphi_{P'}({P})}$, one induces a triangle equivalence from $\z_P/[P]$ to $\z_{\varphi_{P'}({P})}/[\varphi_{P'}({P})]$ sending $T^+_P$ to $T^+_{\varphi_{P'}({P})}$ (see also \eqref{eq:functorphi2} and \eqref{eq:functorphispec}).  We denote this equivalence by $a$. The map $a$ above is induced by this equivalence.}
	\item[\texttt{3}:] Applying \Cref{lem:lowertri}, but replacing $P$ with $P'$, $(M,Q)$ {with} $(M,0)$, {and $\Gamma$ with $\Gamma'= \End_{\z_{P'}/[P']}(T^+_{P'})$}, shows this triangle to be commutative. 
	
	If $\langle 1\rangle$ denotes the suspension functor on $\z_{P'}/[P']$, {we now argue that $\Sigma P''$ is a direct summand of $T_{P'}^{+} \langle 1\rangle$ in $\z_{P'}/[P']$. By \Cref{IY18.3.7}\eqref{IY18.3.7.i}, we may regard $\Sigma P''$ as a $2_{T^+_{P'}}$-term presilting object in $\z_{P'}/[P']$. It thus suffices to show that ${\z_{P'}\over [P']}(T^+_{P'},\Sigma P'')=0$. We will use the following functorial isomorphism \cite[Lemma 4.14]{Jas15} 
	\begin{equation}\label{eq:JassoLemma}
	{\z_{P'}\over [P']}(T^+_{P'},X) \simeq {\Hom_{H^0A}(M^b,\c(A,X)) \over [M](M^b,{\c(A,X)})} 
	\end{equation}
	valid for all $X\in \z_{P'}$. Setting $X=\Sigma P''$ makes the right hand side vanish, so we conclude that $\Sigma P''$ is a direct summand of $T_{P'}^{+} \langle 1\rangle$ in $\z_{P'}/[P']$.}
	
	As a result, we have that $\Sigma P''$ is the largest common direct summand of $P$ and $T^+_{P'}\langle 1\rangle$ in $\z_{P'}/[P']$. Thus, the projective $\Gamma'$-module $Q'$, as it appears in \texttt{3}, is given by the second component of $H_{T^+_{P'}}(\varphi_{P'}(P))$, which is equal to ${\z_{P'}\over [P']}(T^+_{P'},(\Sigma P'')\langle -1\rangle)$. 
	
	\item[\texttt{4}:] The hexagon on the right appears in \eqref{eq:ingeneral}, a diagram we have {defined} to be commutative. {Here, the $k$-algebra $\Gamma$ is the $\tau$-tilting reduction of $H^0A$ with respect to $(M,Q)$, defined as $\End_{H^0A}(M^+)/[M]$. We check that the sets along the third column in our large diagram are consistent with the definitions given in \texttt{3}.}
	
	{Recall that $\Gamma'$ is defined to be $\End_{H^0A}(M^b)/[M]$ in this context. Jasso shows that there is an isomorphism of $k$-algebras \cite[Theorem 4.12(a)]{Jas15} 
	$$\End_{\z_{P'}/[P']}(T^+_{P'})  \to  \End_{H^0A}(M^b)/[M]=\Gamma',$$
	ensuring the consistency with the use of $\Gamma'$ in \texttt{3}.}
	
	{Consider the $\Gamma'$-module $F_{(M,0)} \widetilde{Q}$, where $\widetilde{Q}$ is as defined in \texttt{1}. 
	In order to be consistent with \texttt{3}, we ought to show that ${\z_{P'}\over [P']}(T^+_{P'} ,(\Sigma P'')\langle -1\rangle)$ (denoted $Q'$ in \texttt{3}), when regarded as a $\Gamma'$-module, is isomorphic to $F_{(M,0)} \widetilde{Q}$. By {\eqref{eq:JassoLemma}}, it is to be shown that 
	$$F_{(M,0)} \widetilde{Q} \simeq {\Hom_{H^0A}(M^b,\c(A,(\Sigma P'')\langle -1\rangle)) \over [M](M^b,{\c(A, (\Sigma P'')\langle -1 \rangle)})}$$ as $\Gamma'$-modules.
This is shown by setting up the following $\Gamma'$-isomorphisms:
	\begin{align*}
	F_{(M,0)} \widetilde{Q} &= \Hom_{H^0A}(M^b,\widetilde{Q}) \\
	&=  \Hom_{H^0A}(M^b, f_M H^0 \big( (\Sigma Q)\langle -1 \rangle \big)) \\
	&\simeq  \Hom_{H^0A}(M^b, f_M \c(A, (\Sigma P'')\langle -1 \rangle) \\
	&\simeq { \Hom_{H^0A}(M^b, \c(A, (\Sigma P'')\langle -1 \rangle)\over [M](M^b, \c(A, (\Sigma P'')\langle -1 \rangle)},
	\end{align*}
	where we have used \Cref{lem:addquotrad}{\eqref{lem:addquotrad2}} in the penultimate step. For the final step, one uses the same techniques as in the proof of \Cref{lem:addquotrad}\eqref{lem:addquotrad1}.
	We conclude that the definitions of $\Gamma'$ and $Q'$ in \texttt{3} and \texttt{4} are consistent up to isomorphism.}
	\item[\texttt{5}:] The inner square commutes, since one can apply Case II above to the $2_{T_{P'}^+}$-term presilting object ${\varphi_{P'}(P)}$ in $\z_{P'}/[P']$. 
	\item[{\texttt{6}}:] {The commutativity of this diagram is immediate from the definition of the map $a$ (see \texttt{2} above).}
	\item[{\texttt{7}}:] {To prove that this triangle commutes, it suffices to show that $\Gamma=\End_{\Lambda}(M^+)/[M]$ is isomorphic to the endomorphism $k$-algebra of $T_{\varphi_{P'}({P})}^+$ in $\z_{\varphi_{P'}({P})}/[{\varphi_{P'}({P})}]$, so that we may apply \Cref{lem:lowertri}. Transferring along the equivalence described in \texttt{2}, one finds that the latter is isomorphic to the endomorphism $k$-algebra of $T^+_P$ in $\z_P/[P]$. This is known \cite[Theorem 4.12(b)]{DIRRT17} (it is a more general version of {\eqref{eq:JassoLemma}}), so we are done.}
	\item[{\texttt{8}}:] The commutativity here follows from \Cref{lem:lowertri} applied to $\z_P/[P]$ and $T_P^+$.
\end{itemize}
{The proof is complete.}
		  \end{proof}
	We have accomplished the task of linking the Buan--Marsh bijection to that of Iyama--Yang.
	
	\section{\texorpdfstring{$\tau$}{tau}-cluster morphism categories}\label{section:tcmc}
	
	We are now ready to give a generalisation of $\tau$-cluster morphism categories. In this section, let $A$ as in {\Cref{setup:siltingtstr}, i.e.} a non-positive dg $k$-algebra {such that $H^iA$ is finite dimensional for all $i\in \Z$}.
	
	Given a two-term presilting object $U\in \per(A)$, we proved in \Cref{lem:perptexsummand}\eqref{lem:perptexsummand2} that the perpendicular category $U^{\perp_{\Z}}_{{\rm fd}}$ is t-exact in $\Der_{\rm fd}(A)$. The $\tau$-cluster morphism category should keep track of this information. 
	{To encode the information we want,} there will be a morphism in the $\tau$-cluster morphism category of the form
	\begin{equation*}
	\begin{tikzcd}
	\Der_{\rm fd}(A)\arrow[r,"U"] & U^{\perp_{\Z}}_{{\rm fd}}.
	\end{tikzcd}
	\end{equation*} 
	{By \Cref{prop:largeloc}\eqref{prop:largelocte} {and \eqref{prop:largeloccptpi}}, there is a dg $k$-algebra $C_U$ such that $\per(A)/\thick({U})\simeq \per(C_U)$.}
	 Consider the functor
	\begin{equation}\label{eq:piUP}
	{\pi_{U}\colon \per(A) \to  {\per(C_{U})}}
	\end{equation}
	{provided by \Cref{prop:largeloc}\eqref{prop:largeloccptpi}{, as well as the bijection 
	\begin{equation}\label{eq:pibijU}
		\pi_{U}\colon \npresilt{2}{}_{{U}}(A) \to \npresilt{2}{}(C_{U})
	\end{equation}
	provided by \Cref{prop:PperpIsDfdcor}.}}
	{Given $P\in \npresilt{2}{}(C_U)$,} the $\tau$-cluster morphism category will be constructed in such a way that it contains a morphism 
	\begin{equation*}
	\begin{tikzcd}
	U^{\perp_{\Z}}_{{\rm fd}}\arrow[r,"P"] & {(\pi_{{U}}^{-1}P)}^{\perp_{\Z}}_{{\rm fd}}.
	\end{tikzcd}
	\end{equation*}
	{Moreover,} it will contain a commutative diagram of the form
	\begin{equation*}
	\begin{tikzcd}[row sep=3em]
	\Der_{\rm fd}(A)\arrow[r,"U"]\arrow[rd] & U^{\perp_{\Z}}_{{\rm fd}}\arrow[d,"P"] \\
	& {(\pi_{{U}}^{-1}P)^{\perp_{\Z}}_{{\rm fd}}}
	\end{tikzcd}
	\end{equation*}
	
	{If $U$ and $V$ are different two-term presilting objects in $\per(A)$ such that $U^{\perp_{\Z}}_{{\rm fd}}$ and $V^{\perp_{\Z}}_{{\rm fd}}$ are the same t-exact subcategory of $\Der_{\rm fd}(A)$, it follows from {\Cref{item:H0J3}} and \Cref{prop:largeloc}\eqref{prop:largeloccptpi} that there are triangle equivalences}
		\begin{equation}\label{eq:leftrighteqs}
		\begin{tikzcd}
			\per(C_U) & \arrow[l,"{\overline{\pi_U}}"'] \per(A)/\thick({U})\arrow[d,"\simeq"] & {} \\
			{} & {\per(A)/\thick(V)} \arrow[r,"{\overline{\pi_V}}"]& \per(C_V) \\
		\end{tikzcd}
		\end{equation}
		{sending $A\in \per(A)/\thick({U})$ to $C_U \in \per(C_U)$ and {$A\in \per(A)/\thick(V)$ to} $C_V\in \per(C_V)$, respectively, {where $\overline{\pi_U}$ (resp. $\overline{\pi_V}$) is induced by $\pi_U\colon \per(A)\to \per(C_U)$ (resp. $\pi_V\colon \per(A)\to \per(C_V)$).} {Consequently, there is a triangle equivalence $\beta_{V,U}\colon \per(C_U)\to\per(C_V)$ sending $C_U$ to $C_V$. {{This functor is constructed so that the functor $\beta_{V,U}\pi_U$ is naturally isomorphic to $\pi_V$.}} It clearly induces a bijection}
		\begin{equation}\label{eq:beta}
		\begin{tikzcd}
			\beta_{V,U}\colon\npresilt{2}{C_U}(C_U) \arrow[r]& \npresilt{2}{C_V}(C_V).
		\end{tikzcd}
		\end{equation}
	
	\begin{definition}\label{def:tcmc}
		Let $A$ be {a non-positive dg $k$-algebra {such that $H^iA$ is finite dimensional for all $i\in \Z$}. The \textit{$\tau$-cluster morphism category} of $A$ is denoted by $\tcmc_A$, and is defined as follows:} the objects of $\tcmc_A$ are {(thick, by \Cref{lem:thicktex}) t-exact subcategories $\s$ of $\Der_{\rm fd}(A)$ such that $\s=U^{\perp_{\Z}}_{{\rm fd}}$ for some two-term presilting object $U$ in $\per(A)$. In this context, let $C_U$ be as in \Cref{prop:largeloc}\eqref{prop:largelocte}.} 		
		{Let $\s_1$ be an object in $\tcmc_A$. {Fix a two-term presilting object $U$ in $\per(A)$ such that  $\s_1=U^{\perp_{\Z}}_{{\rm fd}}$}. Then for each {two-term presilting object $P$ in $\per(C_U)$,} we add a morphism 
		\begin{equation}\label{eq:morphismP}
		\s_1 \to^{P} \s_2
		\end{equation} 
		provided that $\s_2= \big({\pi_U^{-1}}P\big)^{\perp_{\Z}}_{{\rm fd}}$, where $\pi_U$ is the bijection} 
		\begin{equation*}
			\pi_U \colon \npresilt{2}{}_U(A) \to  \npresilt{2}{}(C_U),
		\end{equation*}
		{induced by the functor $\pi_U\colon \per(A)\to \per(C_U)$, as provided by \Cref{prop:PperpIsDfdcor}.}
		
		{For any object $\mathcal{S}_1$ in $\tcmc_A$, as fixed above, the presilting object $P=0$ in $\npresilt{2}{}(C_U)$ is seen to determine a morphism $\s_1 \to^{0} \s_1$ in $\tcmc_A$. Morphisms of this form will be the identity morphisms with respect to the composition rule we now describe.}		
		Given two consecutive morphisms
		\begin{equation}\label{eq:tcmccomp}
		\begin{tikzcd}
		\s_1 \arrow[r,"P"]& \s_2\arrow[r,"Q"]& \s_3,
		\end{tikzcd}
		\end{equation}
		{where $P\in \npresilt{2}{}(C_U)$ for some $U\in\npresilt{2}{}(A)$, we may assume that $Q\in {\npresilt{{2}}{}}(C_V)$, where $V$ is the element $\pi_U^{-1}(P)$ of $\npresilt{2}{}(C_U)$.}
		We define the composite morphism $Q\circ P$ to be $\pi_{P}^{-1}(Q)$, {where $\pi_P$ is the bijection}
		\begin{equation}\label{eq:tcmc_piP}
		 \npresilt{2}{}_P(C_U) \to^{\pi_P}  \npresilt{2}{}(C_{U,P}),
		 \end{equation}
	 {obtained by applying} \Cref{prop:PperpIsDfdcor} {to $C_U$ and $P$. {By \Cref{lem:functorphi'}, we may identify the set $\npresilt{2}{}(C_{U,P})$ with $\npresilt{2}{}(C_{V})$.}}
	\end{definition}
	
	{Had we fixed a two-term presilting object $V$ in $\per(A)$ such that $V^{\perp_{\Z}}_{{\rm fd}}=U^{\perp_{\Z}}_{{\rm fd}}$, we would have transferred along $\beta_{V,U}$ in \eqref{eq:beta} and labelled the morphism in \eqref{eq:morphismP} by $\beta_{V,U}P$ instead. {We need to argue that this transferral process does not affect the domain and codomain of a morphism $\tcmc_{A}$, so that our construction of morphisms in $\tcmc_{A}$ is well-defined. Having imposed that $V^{\perp_{\Z}}_{{\rm fd}}=U^{\perp_{\Z}}_{{\rm fd}}$, it follows from \Cref{item:H0J3} that $\thick(U)=\thick(V)$ as thick subcategories of $\per(A)$. {Thus, the thick subcategories $\thick(\pi_U^{-1}P)$ and $\thick(\pi_V^{-1}(\beta_{V,U}P))$ of $\per(A)$ both contain $\thick(U)$.}} Since the functor $\beta_{V,U}$ is constructed to yield a natural isomorphism $\beta_{V,U}\pi_U=\pi_V$ of functors from $\per(A)$ to $\per(C_V)$, {these thick subcategories, that both contain $\thick(U)$, are both sent to $\thick(P)$ by the functor $\pi_U\colon \per(A)\to \per(C_U)$. We may regard $\pi_U$ as a Verdier localisation with respect to $\thick(U)$ (see \eqref{eq:leftrighteqs}), so it induces a bijection from the set of thick subcategories of $\per(A)$ containing $\thick(U)$ to the set of thick subcategories of $\per(C_U)$ \cite[{Proposition 2.3.1 (c)$^{\text{bis}}$ and (d)$^{\text{bis}}$}]{Ver96}.  In particular, we have $\thick(\pi_U^{-1}P)=\thick(\pi_V^{-1}(\beta_{V,U}P))$ as thick subcategories of $\per(A)$.}
	 {By \Cref{item:H0J3} again, we have}
	 $\s_2 = \big({\pi_U^{-1}}P\big)^{\perp_{\Z}}_{{\rm fd}} = \big({\pi_V^{-1}}(\beta_{V,U}P)\big)^{\perp_{\Z}}_{{\rm fd}}$, so transferral along $\beta_{V,U}$ does not affect the domain and codomain. The addition of the morphism $P$ displayed in \eqref{eq:morphismP} to $\tcmc_A$ is only well-defined up to this transferral process. When working with a morphism as in \eqref{eq:morphismP} above, the choice of $U$ will be implicit.}
	
	{We include a lemma to show that our composition rule is well-defined.}
	
	\begin{lemma}\label{lem:tcmcwelld}
	{Consider the diagram in \eqref{eq:tcmccomp}, where $\s_1=U^{\perp_{\Z}}_{{\rm fd}}$ {for some $U\in \npresilt{2}{}(A)$, and $Q\in\npresilt{2}{C_V}$, where $V$ is the element $\pi_U^{-1}(P)$ of $\npresilt{2}{}(C_U)$}. The following assertions hold.
	\begin{enumerate}
		\item\label{lem:tcmcwelld1} $\pi^{-1}_P(Q)\in \npresilt{2}{}_P(C_U)$,
		\item\label{lem:tcmcwelld3} $\s_3=\big(\pi_{U}^{-1}(Q\circ P) \big)^{\perp_{\Z}}_{{\rm fd}}$.
	\end{enumerate}}
	\end{lemma}
	\begin{proof}
	{The first assertion \eqref{lem:tcmcwelld1} follows directly from the existence of the bijection in \eqref{eq:tcmc_piP}.}
	
	{Secondly and finally, we prove the claim in \eqref{lem:tcmcwelld3}.
	By definition, we have that $\s_3 = (\pi_{V}^{-1}(Q))^{\perp_{\Z}}_{{\rm fd}}$. We show that $\big(\pi_{U}^{-1}(Q\circ P) \big)^{\perp_{\Z}}_{{\rm fd}}$ is equal to $\s_3$ as a t-exact subcategory of $\Der_{\mathrm{fd}}(A)$. Direct calculation shows the following:
	\begin{equation*}
	 \big(\pi_{U}^{-1}(Q\circ P) \big)^{\perp_{\Z}}_{{\rm fd}} = \big(\pi_{U}^{-1}\pi_{P}^{-1}(Q) \big)^{\perp_{\Z}}_{{\rm fd}} = \big(\pi_{\pi^{-1}_U(P)}^{-1}(Q) \big)^{\perp_{\Z}}_{{\rm fd}}  = \big(\pi_{V}^{-1}(Q) \big)^{\perp_{\Z}}_{{\rm fd}} = \s_3,
	\end{equation*}} 
	where we have used \Cref{lem:functorphi'} to get the second equality. The proof is complete. 
	\end{proof} 
	
	Having defined its objects, morphisms, and composition rule, we now check that $\tcmc_{A}$ is a category. It is clear that the two-term presilting object $P=0$ in $\npresilt{2}{}(C_U)$ is the identity morphism on ${U^{\perp_{\Z}}_{\rm fd}}$. However, it is not obvious that the composition rule is associative. Buan--Marsh show that their $\tau$-cluster morphism category has an associative composition rule  \cite[Corollary~1.8]{BM18w} by considering several cases.
	Igusa--Todorov's treatment of the representation finite hereditary case makes use of root systems \cite[Section~1]{IT17}.
	Our proof takes advantage of the functoriality of silting reduction.
	
	\begin{theorem}\label{thm:associative}
		The composition rule for $\tcmc_{A}$ is associative. Thus $\tcmc_{A}$ is a category.
	\end{theorem}
	\begin{proof}
		Given three composable morphisms 
		\begin{equation*}
		\begin{tikzcd}
		\s_1\arrow[r,"P"]& \s_2\arrow[r,"Q"]& \s_3 \arrow[r,"R"]& \s_4
		\end{tikzcd}
		\end{equation*}
		we verify the identity
		\begin{equation*}
		R\circ (Q\circ P) = (R\circ Q)\circ P.
		\end{equation*}
		The bijection in \eqref{eq:tcmc_piP} is induced by the functor $\pi_P$ (and similarly for $\pi_Q$), whence we may apply \Cref{lem:functorphi'} when expanding the right hand side.
		\begin{align*}
		(R\circ Q)\circ P &= (\pi_Q^{-1}(R))\circ P \\
		&= \pi_P^{-1}(\pi_Q^{-1}(R)) \\
		&= \pi^{-1}_{\pi_{ P}^{-1}(Q)}(R) \\
		&= \pi^{-1}_{Q\circ P}(R) \\
		&= R\circ(Q\circ P).
		\end{align*}
		The proof is complete.
	\end{proof}
	
	Our main aim of this section is to show that the category defined above generalises the $\tau$-cluster morphism category of Buan--Marsh (and Buan--Hanson), which we denote by $\tcmc^{\rm BM}_{{\Lambda}}$, where ${\Lambda}$ is a finite dimensional $k$-algebra. For a $\tau$-tilting finite algebra ${\Lambda}$, recall that the objects of $\tcmc^{\rm BM}_{{\Lambda}}$ are the wide subcategories of $\fgMod {\Lambda}$. For arbitrary finite dimensional $k$-algebras, its objects are the $\tau$-perpendicular subcategories {of $\fgMod\Lambda$}, namely {the wide subcategories of $\fgMod \Lambda$} that are of the form $J(M,{R})$ for some {$\tau$-rigid pair} $(M,{R})$ in $\fgMod {\Lambda}$ \cite{BH21}. {Buan--Hanson show that if $W$ is a $\tau$-perpendicular subcategory of $\fgMod{\Lambda}$, and $(M',R')$ is a $\tau$-rigid pair in $W$, then $J_{W}(M',R')$ is $\tau$-perpendicular subcategory of $\fgMod{\Lambda}$ \cite[Theorem 6.4]{BH21}.}
	
	{Let {$(M_1,R_1)$ be a $\tau$-rigid pair in $\fgMod{\Lambda}$ and let} $W_1=J(M_1,R_1)$ be {the} $\tau$-perpendicular subcategory of {$(M_1,R_1)$}.}
	There is a morphism 
	\begin{tikzcd}
		W_1\arrow[r,"{(M_{{2}},{R_{{2}}})}"]& W_2 
	\end{tikzcd}
	{in $\tcmc_{\Lambda}^{\rm BM}$ whenever} $(M_{{2}},{R_{{2}}})$ is a $\tau$-rigid pair in $W_1$ such that $J_{W_1}(M_{{2}},{R_{{2}}})=W_2$, where $J_{W_1}(M_{{2}},{R_{{2}}})$ is the $\tau$-perpendicular category of $(M_{{2}},{R_{{2}}})$ in $W_1$. 
	{If}
	\begin{equation}\label{eq:FMQbij'}
		{F_{(M_{{1}},{R_{{1}}})}}\colon \staurigidpair(W_1) \to \staurigidpair({\Gamma_1})
	\end{equation}
	{is the bijection {provided by} \eqref{eq:FMQbij} {(so $\Gamma_1$ is the $\tau$-tilting reduction in $\fgMod \Lambda$ with respect to $(M_{{1}},R_{{1}})$)}, and}
	\begin{equation*}
		\mathcal{E}_{{F_{(M_1,R_1)}}(M_{{2}},{R_{{2}}})}\colon \staurigidpair_{{F_{(M_1,R_1)}}(M_{{2}},{R_{{2}}})}({\Gamma_1)} \to \staurigidpair(W_2)
	\end{equation*}
	 {is the bijection in \Cref{thm:BM18w.3.6}, then}
	 the composite of
	\begin{equation}\label{eq:BHcomp}
	\begin{tikzcd}
	W_1 \arrow[r, "{(M_{{2}},{R_{{2}}})}"] & W_2 \arrow[r, "{(M_{{3}},{R_{{3}}})}"] & W_3
	\end{tikzcd}
	\end{equation}
	{in $\tcmc_{\Lambda}^{\rm BM}$} is defined by ${F^{-1}_{(M_{{1}},{R_{{1}}})}\circ}\mathcal{E}_{{F_{(M_1,R_1)}}(M_{{2}},{R_{{2}}})}^{-1}(M_{{3}},{R_{{3}}}){\in \staurigidpair(W_1)}$.
	
	Note that our approach is independent of that of Buan--Marsh and Buan--Hanson.
	Hence, it gives an alternative approach to defining $\tau$-cluster morphism categories for all finite dimensional $k$-algebras, and also for non-positive dg $k$-algebras {with finite dimensional cohomology in all degrees}.
	
	\begin{theorem}\label{thm:sameasBM}
		Let $A$ be {as in \Cref{setup:siltingtstr}, i.e.} {a non-positive dg $k$-algebra such that $H^iA$ is finite dimensional for all $i\in \Z$.} 
		{There is an} equivalence of categories
		\begin{equation}\label{eq:sameasBMfunctor}
		{E\colon}\tcmc_{A}\to \tcmc^{\rm BM}_{H^0A}.
		\end{equation}
		{Explicitly, this equivalence sends the object $\s\in \tcmc_A$ to $\s \cap \fgMod H^0 A \in \tcmc^{\rm BM}_{H^0A}$ (see {\Cref{prop:H0J}}\eqref{lem:plus_rev1}}) and the map $\s_1\to^{P} \s_2$ to $F^{-1}_{H_{{A}}({U})}\circ H_{C_U}(P)$ (see \Cref{IJY14.4.6} and \eqref{eq:FMQbij'}), where $U$ is a two-term presilting object in $\per(A)$ such that $U^{\perp_{\Z}}_{{\rm fd}} = \s_1$ and $C_U$ is the non-positive dg $k$-algebra defined in \Cref{prop:largeloc}\eqref{prop:largelocte}.
		In particular, if $A$ is a finite dimensional $k$-algebra, then $\tcmc_{A}$ is equivalent to $\tcmc_{A}^{\rm BM}$.
	\end{theorem}
	\begin{proof}
		The morphism set $\tcmc_{A}(\s_1,\s_2)$ is a subset of $\npresilt{2}{}(C_U)$. {For $i\in\{1,2\}$, let $W_i$ denote the wide subcategory $\s_i \cap \fgMod H^0 A$ of $\fgMod H^0A$.} The morphism set $\tcmc^{\rm BM}_{H^0 A}(W_1,W_2)$ is a subset of $\staurigidpair(W_1)$. Using \Cref{IJY14.4.6} {and \eqref{eq:FMQbij'}, we define a map} 
		\begin{equation}\label{eq:sameasBM}
		\begin{tikzcd}[column sep=3.5em]
		\tcmc_{A}(\s_1,\s_2)\arrow[r,"{H_{C_U}}"] &  {\staurigidpair({H^0C_U})} \arrow[r,"{F^{-1}_{{H_{A}({U})}}}"] &  {\staurigidpair({W_1})}.
		\end{tikzcd}
		\end{equation}
		{To show that the map of Hom-sets is well-defined, we show that the $\tau$-perpendicular category of {${{F^{-1}_{H_A({U})}}\circ} H_{C_U}(P)$} in $W_1$ is $W_2$. It follows from \Cref{prop:H0J}\eqref{item:H0J1} that $W_1 = J(H_A(U))$ and that $W_2 = J(H_A(\pi_U^{-1}(P))$. Using results from \Cref{sec:comp}, we find that
		\begin{align*}
			W_2 &= J(H_A(\pi_U^{-1}(P)) & \\
			&=  J_{W_1}(\mathcal{E}_{H_A(U)}(H_A(\pi_U^{-1}(P))) & \text{(\cite[Theorem 4.3]{BM18w})} \\
			&= J_{W_1}(H'_{A} \circ \varphi_{U} (\pi_U^{-1}(P))) & \text{(\Cref{th:Jas15.4.12})} \\
			&= J_{W_1}(F^{-1}_{H_A({U})}\circ H_{T_U^+}(\varphi_{U}\pi_U^{-1}P)) & \text{(\Cref{lem:lowertri})} \\
			&= J_{W_1}(F^{-1}_{H_A({U})}\circ H_{C_U}(P)) & \text{(we refer to the next sentence)}
		\end{align*}
		In the final step, we use that the functor $\pi_U \colon \per(A) \to \per(C_U)$ identifies $\per(C_U)$ with $\per(A)/\thick(U)$, whence \Cref{IY18.3.7}\eqref{IY.3.6} gives a triangle equivalence $\z_U/[U] \to \per(C_U)$ sending $T^+_U$ to $C_U$. {By \Cref{prop:PperpIsDfdcor}, it sends} $\varphi_{U}\pi_U^{-1}P$ to $P$. We have thus shown that} the image of {the map in \eqref{eq:sameasBM}} is a subset of $\tcmc^{\rm BM}_{H^0A}(W_1,W_2)$, so a map of Hom-sets has thus been constructed.
		
It is easy to check that {$E$} sends identity maps to identity maps; they are given by trivial presilting objects and trivial $\tau$-rigid pairs, {i.e. a zero object and a pair of zero objects, respectively}. It should also be shown that composition is respected. Let $P$ and $Q$ be a pair of composable morphisms in $\tcmc_A$, as shown in \eqref{eq:tcmccomp}. Their composite is then given by ${V=}\pi_{ P}^{-1}(Q)$. {Let $(M,R)$ denote the $\tau$-rigid pair {$E(P)=F_{H_{A}(U)}^{-1}(H_{C_U}(P))$ in $W_1$.}} Consider the following {commutative} diagram of bijections, {where we have referred to previous results to address the commutativity of the non-obviously commutative subdiagrams (recall that $\zeta_P$ in \Cref{prop:PperpIsDfdcor} is induced by a triangle equivalence sending $T^+_P$ to $C_V$).}
		\begin{equation}\label{eq:sameasBMsq}
		\begin{tikzcd}[column sep=1.5em,row sep=2em]
\npresilt{2}{}_{{ P}}({C_U}) \arrow[rd, "{\varphi}_{ P}"] \arrow[rr, "H_{C_U}"] \arrow[dd, "\pi_{P}"'] \arrow[rrrdd, "\eqref{th:Jas15.4.12}", phantom, bend left=5] \arrow[dd, "\eqref{prop:PperpIsDfdcor}", phantom, bend left=70] &                                                                                                                                                                                                     & \staurigidpair_{H_{{C_U}}({ P})}({H^0C_U}) \arrow[rdd, "\mathcal{E}_{H_{{C_U}}(P)}" description] \arrow[r, "{F^{-1}_{{H_{A}(U)}}}"] & {\staurigidpair_{(M,R)}({W_1})} \arrow[dd, "{\mathcal{E}_{H_{{C_U}}(P)}\circ F_{{H_A(U)}}}"] \\
                                                                                                                                                                                                 & {\npresilt{2}{{T^+_P}}({\z_P/[P]})} \arrow[rrd, "H'_{C_U}"] \arrow[ld,"\zeta_P" description] \arrow[d, "H_{T^+_P}"] \arrow[rrd, "{\eqref{lem:lowertri}}", phantom, near start, bend right=10] &                                                                                                                                &                                                                                           \\
\npresilt{2}{}(C_{V}) \arrow[r, "H_{C_{V}}"']                                                                                                                                                    & \staurigidpair({H^0C_{V}}) \arrow[rr, "{F^{-1}_{H_{{A}}(V)}}"']                                                                                                                                        &                                                                                                                                & \staurigidpair(W_2)                                                                      
\end{tikzcd}
		\end{equation}
		{We have defined the morphism $E(Q\circ P)=E(\pi_{P}^{-1}(Q))$ in $\tcmc^{\rm BM}_{H^0A}(W_1,W_2)$ as the image of $\pi_{P}^{-1}(Q)$ along the top row. By the commutativity of the diagram:}
		$$ {E(Q\circ P) = F_{{H_A(U)}}^{-1}\circ \mathcal{E}_{H_{{C_U}}(P)}^{-1} {\circ} {F^{-1}_{{H_A(V)}}} \circ H_{C_V}(Q) = F_{{H_A(U)}}^{-1}\circ \mathcal{E}_{H_{{C_U}}(P)}^{-1}(E(Q)) = E(Q)\circ E(P), }$$
		{where the final equality uses the definition of composition in $\tcmc_{\Lambda}^{\rm BM}$ given in \eqref{eq:BHcomp}. Indeed, the $\tau$-rigid pair $H_A(U)$ in $\fgMod H^0A$ has $W_1$ as its $\tau$-perpendicular subcategory (see \Cref{prop:H0J}\eqref{item:H0J1}), and $H_{C_U}(P)$ is sent to $E(P)$ by $F^{-1}_{H_A(U)}$. Having shown} {that composition is respected,} we have thus defined a functor {$E$} from $\tcmc_A$ to $\tcmc_{H^0A}^{\rm BM}$.
		
		{Since the functor $E$ in \eqref{eq:sameasBMfunctor} is a bijection as a map of objects, as was shown in \Cref{prop:H0J}{\eqref{item:H0J2}}, it is essentially surjective. We conclude the proof by showing that $E$ is fully faithful.
		It is to be shown that the maps in \eqref{eq:sameasBM} inject for all $\s_1,\s_2\in \tcmc_A$, and that the image is $\tcmc^{\rm BM}_{H^0A}(W_1,W_2)$.  
		Since the maps of the form \eqref{eq:sameasBM} are restrictions of bijections, they inject, whence $E$ is faithful.
		To show that $E$ {is} full, we fix an element $(M,R) \in \tcmc^{\mathrm{BM}}_{H^0{A}}(W_1,W_2)$ and show that it admits a preimage. Now, we have that $H_{C_U}^{-1}(M,R)\in \npresilt{2}{}(C_U)$.
		Since $W_2$ is the $\tau$-perpendicular category of $(M,R)$ in $W_1$, it follows from \Cref{prop:H0J}\eqref{item:H0J1} that $E$ sends $(H_{C_U}^{-1}{\circ F_{H_A(U)}})(M,R)^{\perp_{\Z}}_{{\rm fd}}$ to $W_2$.
		This puts $H_{C_U}^{-1}(M,R)\in \tcmc_{A}(\s_1,\s_2)$, and this morphism in $\tcmc_A$ is sent to $(M,R) \in \tcmc^{\mathrm{BM}}_{H^0{A}}(W_1,W_2)$ by the map in \eqref{eq:sameasBM}.
		This shows that $E$ is full, and in turn that it is fully faithful.}  
	\end{proof}
	
	\bibliographystyle{IEEEtranSA}
	\bibliography{tcmc}	
\end{document}